\documentclass[11pt]{article}

\usepackage[sort,numbers]{natbib}
\usepackage[bookmarks=true,bookmarksnumbered=true,colorlinks=true,allcolors=cyan]{hyperref}

\usepackage{amsmath,amssymb,amsfonts,amsthm}
\usepackage{lscape}
\usepackage{arydshln}
\renewcommand*{\baselinestretch}{1.25}

\usepackage{geometry}
\usepackage{indentfirst}

\usepackage{enumitem,color}

\newtheorem{theorem}{Theorem}[section]
\newtheorem{lemma}{Lemma}[section]
\newtheorem{proposition}{Proposition}[section]
\newtheorem{corollary}{Corollary}[section]
\theoremstyle{definition}
\newtheorem{definition}{Definition}[section]
\newtheorem*{rmk*}{Remark}
\newtheorem{rmk}{Remark}[section]

\DeclareMathOperator{\variance}{Var}
\DeclareMathOperator{\covariance}{Cov}

\DeclareMathOperator{\trace}{tr}

\DeclareMathOperator{\influence}{Inf}

\allowdisplaybreaks

\numberwithin{equation}{section}

\makeatletter
    \renewcommand*{\section}{\@startsection{section}{1}{\z@}%
    {6pt}{3pt}{\reset@font\normalsize\bfseries}}
\makeatother
    
\makeatletter
    \renewcommand*{\subsection}{\@startsection{subsection}{2}{\z@}%
    {3pt}{3pt}{\reset@font\normalsize\mdseries\itshape}}
\makeatother

\makeatletter
    \renewcommand*{\subsubsection}{\@startsection{subsubsection}{3}{\z@}%
    {3pt}{3pt}{\reset@font\normalsize\mdseries\itshape}}
\makeatother

\makeatletter
\def\@listi{\leftmargin\leftmargini
  \topsep=.5\baselineskip 
  \partopsep=0pt \parsep=0pt \itemsep=0pt}
\let\@listI\@listi
\@listi
\def\@listii{\leftmargin\leftmarginii
  \labelwidth\leftmarginii \advance\labelwidth-\labelsep
  \topsep=0pt \partopsep=0pt \parsep=0pt \itemsep=0pt}
\def\@listiii{\leftmargin\leftmarginiii
  \labelwidth\leftmarginiii \advance\labelwidth-\labelsep
  \topsep=0pt \partopsep=0pt \parsep=0pt \itemsep=0pt}
\def\@listiv{\leftmargin\leftmarginiv
  \labelwidth\leftmarginiv \advance\labelwidth-\labelsep
  \topsep=0pt \partopsep=0pt \parsep=0pt \itemsep=0pt}
\makeatother

\newcommand{\bs}[1]{\boldsymbol{#1}}

\title{Gaussian approximation of maxima of Wiener functionals and its application to high-frequency data}
\author{Yuta Koike\thanks{Mathematics and Informatics Center and Graduate School of Mathematical Sciences, The University of Tokyo, 3-8-1 Komaba, Meguro-ku, Tokyo 153-8914 Japan}
\thanks{Department of Business Administration, Graduate School of Social Sciences, Tokyo Metropolitan University, Marunouchi Eiraku Bldg. 18F, 1-4-1 Marunouchi, Chiyoda-ku, Tokyo 100-0005 Japan}
\thanks{The Institute of Statistical Mathematics, 10-3 Midori-cho, Tachikawa, Tokyo 190-8562, Japan}
\thanks{CREST, Japan Science and Technology Agency}
}

\begin{document}

\maketitle

\begin{abstract}

This paper establishes an upper bound for the Kolmogorov distance between the maximum of a high-dimensional vector of smooth Wiener functionals and the maximum of a Gaussian random vector. 
As a special case, we show that the maximum of multiple Wiener-It\^o integrals with common orders is well-approximated by its Gaussian analog in terms of the Kolmogorov distance if their covariance matrices are close to each other and the maximum of the fourth cumulants of the multiple Wiener-It\^o integrals is close to zero. This may be viewed as a new kind of fourth moment phenomenon, which has attracted considerable attention in the recent studies of probability. 
This type of Gaussian approximation result has many potential applications to statistics. To illustrate this point, we present two statistical applications in high-frequency financial econometrics: One is the hypothesis testing problem for the absence of lead-lag effects and the other is the construction of uniform confidence bands for spot volatility.   

\vspace{3mm}

\noindent \textit{Keywords}: Bootstrap; Fourth moment phenomenon; Malliavin calculus; Maximum; Stein's method; Uniform confidence bands.

\end{abstract}

\section{Introduction}

This study is originally motivated by the problem of testing whether there exists a (possibly) time-lagged correlation between two Brownian motions based on their high-frequency observation data. Roughly speaking, the setting considered here is described as follows. We discretely observe the following two continuous-time processes on the interval $[0,T]$:
\begin{equation}\label{hry-model}
X^1_t=x^1_0+\sigma_1B^1_t,\qquad
X^2_t=x^2_0+\sigma_2B^2_{t-\vartheta},\qquad t\in[0,T],
\end{equation}
where $x^1_0,x^2_0\in\mathbb{R}$, $\sigma_1,\sigma_2>0$, $B_t=(B_t^1,B_t^2)$ $(t\in\mathbb{R})$ is a bivariate two-sided Brownian motion with correlation $\rho\in(-1,1)$ and $\vartheta\in\mathbb{R}$. For each $\nu=1,2$, the process $X^\nu$ is observed at the time points $0\leq t^\nu_0<t^\nu_1<\cdots<t^\nu_{n_\nu}\leq T$, hence the observation times are possibly non-synchronous. Based on the observation data $(X^1_{t^1_i})_{i=0}^{n_1}$ and $(X^2_{t^2_j})_{j=0}^{n_2}$, we aim at solving the following statistical hypothesis testing problem:
\begin{equation}\label{hry-test}
H_0:\rho=0\qquad\text{vs}\qquad H_1:\rho\neq0.
\end{equation}
Model \eqref{hry-model} was introduced in \citet{HRY2013} (as a more general one) to model lead-lag effects in high-frequency financial data (see also \cite{RR2010}). In \cite{HRY2013} the problem of estimating the time-lag parameter $\vartheta$ is considered. To estimate $\vartheta$, \citet{HRY2013} have introduced the following contrast function:
\[
U_n(\theta)=\sum_{i,j}(X^1_{t^1_i}-X^1_{t^1_{i-1}})(X^2_{t^2_j}-X^2_{t^2_{j-1}})1_{\{(t^1_{i-1},t^1_i]\cap(t^2_{j-1}-\theta,t^2_j-\theta]\neq\emptyset\}}.
\]
$U_n(\theta)$ could be considered as the (sample) cross-covariance function between the returns of $X^1$ and $X^2$ at the lag $\theta$ computed by \citet{HY2005}'s method. \citet{HRY2013} have shown that  
\[
\widehat{\vartheta}_n=\mathrm{arg}\max_{\theta\in\mathcal{G}_n}|U_n(\theta)|
\]
is a consistent estimator for $\vartheta$ under some regularity conditions while one appropriately takes the finite set $\mathcal{G}_n$ as long as $\rho\neq0$. The condition $\rho\neq0$ is necessary because it is clearly impossible to identify the parameter $\vartheta$ if $\rho=0$. Therefore, unless we can believe $\rho\neq0$ due to some external information, we need to reject the null hypothesis in the above testing problem before we carry out estimation of $\vartheta$. A natural approach to solve testing problem \eqref{hry-test} is to reject the null hypothesis if the value of $\max_{\theta\in\mathcal{G}_n}|U_n(\theta)|$ is \textit{too} large. To implement this idea precisely, we need to derive or approximate the distribution of $\max_{\theta\in\mathcal{G}_n}|U_n(\theta)|$ under the null hypothesis $H_0$. One main purpose of this paper is to give an answer to this problem. More generally, we consider the problem of approximating the distributions of maximum-type statistics appearing in high-frequency financial econometrics. Indeed, we encounter such statistics in many problems of high-frequency financial econometrics, e.g.~construction of uniform confidence bands for spot volatility and other time-varying characteristics, family-wise error rate control for testing at many time points (cf.~\cite{COR2016,ABD2007joe}), change point analysis of volatility (cf.~\cite{BJV2017}), testing the absence of jumps (cf.~\cite{LM2008,PW2016}) and so on.

From a mathematical point of view, this paper is built on two recent studies developed in different areas. The first one is the seminal work of \citet*{CCK2013,CCK2014,CCK2015,CCK2016} which we call the \textit{Chernozhukov-Chetverikov-Kato theory}, or the \textit{CCK theory} for short. One main conclusion from the CCK theory is a bound for the Kolmogorov distance between the distributions of the maximum of a (high-dimensional) random vector and that of a Gaussian vector, which has an apparent connection to our purpose. However, their result is not directly applicable to our problem because their target random vector is a sum of independent random vectors \cite{CCK2013,CCK2014,CCK2016} or Gaussian \cite{CCK2015,CCK2016}. In fact, one of our main target random vectors, $(U_n(\theta))_{\theta\in\mathcal{G}_n}$, is a sum of dependent random vectors even under the null hypothesis where the dependence is caused by the non-synchronicity of the observation times. Although there are several extensions of the CCK theory to a sum of dependent random vectors (see e.g.~\cite{ZW2015,ZC2017,CCK2014c,Chen2017,CQYZ2017}), it still seems difficult to apply such a result to our problem because the non-synchronicity causes a quite complex, ``non-stationary'', dependence structure. 
In this aspect this paper aims at extending the CCK theory suitably to our purpose, and our results indeed generalize several results of \cite{CCK2015}. In particular, our results do not require that the target random vector should be written as a sum of random vectors and give a simpler bound than those of the preceding studies listed above.   


It turns out that in the CCK theory the independence/Gaussianity assumption on the target vector is crucial for the application of \textit{Stein's method}.\footnote{The independence assumption also plays a role in deriving maximal moment inequalities, but this issue may be considered separately.} In other words, we can naturally extend the CCK theory to a case without independence as long as Stein's method is effectively applicable. This viewpoint leads us to using another important theory for this work, \textit{Malliavin calculus}, in our problem. In fact, starting from the seminal work of \citet{NP2009PTRF}, the recent studies show that ``Stein's method and Malliavin calculus fit together admirably well'' (page 3 of \citet{Nourdin2013}). This paper shows that this statement continues to hold true in the application to the CCK theory. Our application of Malliavin calculus is based on a multivariate extension of the ideas from \cite{NP2009PTRF}, which is established in \cite{NPR2010} (see also \cite{NP2010}). We refer to the monograph \cite{NP2012} for more information about this subject.   

After developing the main Gaussian approximation results, we turn to the original problem of statistical applications in high-frequency data. In this paper we demonstrate two applications: One is testing the absence of lead-lag effects and the other is constructing uniform confidence bands for spot volatility. We have already explained the background of the former problem in the above, so we briefly discuss the latter one. Estimation of spot volatility is one of major topics in high-frequency financial econometrics (see Chapter 8 of \cite{AJ2014} and references therein). There are quite a few articles concerning construction of \textit{pointwise} confidence bands for spot volatility; see e.g.~\cite{Kristensen2010,MR2015,MZ2008,APPS2012}.
In contrast, only a few results are available on the behavior of \textit{uniform} errors in spot volatility estimation: \citet{Kristensen2010} and \citet{KK2016} give uniform convergence rates for kernel-type spot volatility estimators, while \citet{FW2008} consider a Gumbel approximation for the distribution of uniform errors of kernel-type spot volatility estimators. Besides, \citet{Sabel2014} implements multiscale inference for spot volatility via KMT construction. 
This paper contributes this relatively undeveloped areas by providing a new approach to construct uniform confidence bands for spot volatility in the spirit of the CCK theory: Construction of uniform confidence bands is a typical application of the CCK theory, cf.~\cite{KS2016,KK2017,CCK2014b}. 

In the first application, the Gaussian approximation itself is still statistically infeasible because the covariance structure of the objective statistics is unknown. For this reason we also develop a wild (or multiplier) bootstrap procedure to approximate the quantiles of the error distribution of the test statistic, which is the approach taken in the CCK theory. The Gaussian approximation result serves as validating such a bootstrap procedure.  

The remainder of this paper is organized as follows. 
Section \ref{sec:main} presents the main Gaussian approximation results obtained in this study. 
In Section \ref{sec:qf} we derive Gaussian approximation results for maxima of random symmetric quadratic forms as an application of the main results. 
We present two statistical applications of our results in high-frequency financial econometrics in Section \ref{sec:applications}. We especially propose a testing procedure for \eqref{hry-test}. The finite sample performance of this testing procedure is illustrated in Section \ref{sec:simulation}.  
We put most technical parts of the paper on the Appendix: Appendix \ref{preliminaries} collects the preliminary definitions and results used in Appendix \ref{sec:proofs}, which contains proofs of all the results presented in the main text of the paper. 

\section*{Notation}

Throughout the paper, $\mathfrak{C}=(\mathfrak{C}(i,j))_{1\leq i,j\leq d}$ denotes a $d\times d$ nonnegative definite symmetric matrix, and $Z=(Z_1,\dots,Z_d)^\top$ denotes a $d$-dimensional centered Gaussian random vector with covariance matrix $\mathfrak{C}$. 
For a vector $x=(x_1,\dots,x_d)^\top\in\mathbb{R}^d$, we write $x_\vee=\max_{1\leq j\leq d}x_j$. 
For any $\varepsilon>0$ and any subset $A$ of $\mathbb{R}$, we write $A^\varepsilon=\{x\in\mathbb{R}:|x-y|\leq\varepsilon\text{ for some }y\in A\}$. 
For a real-valued function $f$ defined on an interval $I\subset\mathbb{R}$ and $\eta>0$, we write $\|f\|_\infty=\sup\{|f(x)|:x\in I\}$ and $w(f;\eta)=\sup\{|f(s)-f(t)|:s,t\in I,|s-t|\leq\eta\}$. 
For a random variable $\xi$ and $p\geq1$, we write $\|\xi\|_p=\{E[|\xi|^p]\}^{1/p}$. 
For a matrix $A$, we denote by $\|A\|_{\mathrm{sp}}$ and $\|A\|_F$ its spectral norm and Frobenius norm, respectively. 

Finally, we enumerate the notation from Malliavin calculus which are necessary to state our main results. We refer to \cite{Nualart2006,NP2012,Janson1997} for a detailed description of Malliavin calculus. Also, see Section \ref{sec:malliavin} of Appendix \ref{preliminaries} for a concise overview of the notions from Malliavin calculus used in this paper.
\begin{itemize}

\item Throughout the paper, $H$ denotes a real separable Hilbert space. The inner product and the norm of $H$ are denoted by $\langle\cdot,\cdot\rangle_H$ and $\|\cdot\|_H$, respectively. 

\item We assume that an isonormal Gaussian process $W=(W(h))_{h\in H}$ over $H$ defined on a probability space $(\Omega,\mathcal{F},P)$ is given. 
We denote by $L^2(W)$ the space $L^2(\Omega,\sigma(W),P)$ for short. 

\item For a non-negative integer $q$, $H^{\otimes q}$ and $H^{\odot q}$ denote the $q$th tensor power and $q$th symmetric tensor power, respectively. 

\item For an element $f\in H^{\odot q}$ we denote by $I_q(f)$ the $q$th multiple Wiener It\^o integral of $f$. 

\item For any real number $p\geq1$ and any integer $k\geq1$, $\mathbb{D}_{k,p}$ denotes the stochastic Sobolev space of random variables which are $k$ times differentiable in the Malliavin sense and the derivatives up to order $k$ have finite moments of order $p$. If $F\in\mathbb{D}_{k,p}$, we denote by $D^kF$ the $k$th Malliavin derivative of $F$. We write $DF$ instead of $D^1F$ for short. 


\item $L$ denotes the Ornstein-Uhlenbeck operator. Also, the pseudo inverse of $L$ is denoted by $L^{-1}$. 

\end{itemize}

\section{Main results}\label{sec:main}

Throughout this section, $F=(F_1,\dots,F_d)^\top$ denotes a $d$-dimensional random vector such that $F_j\in\mathbb{D}_{1,2}$ and $E[F_j]=0$ for all $j=1,\dots,d$. For each $\beta>0$, we define the function 
$\Phi_\beta:\mathbb{R}^d\to\mathbb{R}$ by 
\[
\Phi_\beta(x)=\beta^{-1}\log\left(\sum_{j=1}^de^{\beta x_j}\right)\qquad(x=(x_1,\dots,x_d)^\top\in\mathbb{R}^d).
\]
Eq.(1) from \cite{CCK2015} states that
\begin{equation}\label{max-smooth}
0\leq \Phi_\beta(x)-x_\vee\leq \beta^{-1}\log d
\end{equation}
for any $x\in\mathbb{R}^d$.

We first give a generalization of Theorem 1 from \cite{CCK2015} as follows:
\begin{proposition}\label{stein}
Let $g:\mathbb{R}\to\mathbb{R}$ be a $C^2$ function with bounded first and second derivatives. Then, for any $\beta>0$ we have
\[
\left|E\left[g\left(\Phi_\beta(F)\right)\right]-E\left[g\left(\Phi_\beta(Z)\right)\right]\right|
\leq(\|g''\|_\infty/2+\beta\|g'\|_\infty)\Delta,
\]
where
\[
\Delta=E\left[\max_{1\leq i,j\leq d}|\mathfrak{C}(i,j)-\langle DF_i,-DL^{-1}F_j\rangle_H|\right].
\]
In particular, it holds that  
\[
\left|E\left[g\left(F_\vee\right)\right]-E\left[g\left(Z_\vee\right)\right]\right|
\leq(\|g''\|_\infty/2+\beta\|g'\|_\infty)\Delta+2\beta^{-1}\|g'\|_\infty\log d.
\]
\end{proposition}

\begin{rmk}
We can indeed derive Theorem 1 of \cite{CCK2015} from Proposition \ref{stein} in the following way. Suppose that the law of $F$ is the $d$-dimensional normal distribution with mean 0 and covariance matrix $\Sigma=(\Sigma(i,j))_{1\leq i,j\leq d}$. Without loss of generality, we may assume that $F$ is expressed as $F=\Sigma^{1/2}G$ with $G$ being a $d$-dimensional standard Gaussian vector. Then we can define the isonormal Gaussian process $W$ over $H:=\mathbb{R}^d$ by $W(h)=h^\top G$, $h\in H$ (cf.~Example 2.1.3 of \cite{NP2012}), and we have $F_i=\sum_{j=1}^d\gamma_{ij}W(e_j)$ for every $i=1,\dots,d$, where $\gamma_{ij}$ denotes the $(i,j)$-th component of the matrix $\Sigma^{1/2}$ and $(e_1,\dots,e_d)$ denotes the canonical basis of $\mathbb{R}^d$. In this case it holds that
\[
\langle DF_i,-DL^{-1}F_j\rangle_H=\sum_{k,l=1}^d\gamma_{ik}\gamma_{jl}\langle e_k,e_l\rangle_H
=\sum_{k=1}^d\gamma_{ik}\gamma_{jk}=\Sigma(i,j),
\]
hence we obtain the conclusion of Theorem 1 from \cite{CCK2015}.
\end{rmk}

Proposition \ref{stein} and some elementary approximation arguments lead the following useful lemma:
\begin{lemma}\label{coupling}
There is a universal constant $C>0$ such that 
\begin{align*}
P(F_\vee\in A)\leq P(Z_\vee\in A^{5\varepsilon})+C\varepsilon^{-2}(\log d)\Delta
\end{align*}
for any Borel set $A$ of $\mathbb{R}$ and any $\varepsilon>0$.
\end{lemma}

\begin{rmk}
Lemma \ref{coupling} is useful when we derive a Gaussian approximation for the supremum of statistics indexed by an infinite set (see Proposition \ref{spot} and its proof). In fact, Lemma \ref{coupling} can be considered as a counterpart of Theorem 3.1 from \cite{CCK2016}, which is used to derive their Gaussian approximation results for suprema of empirical processes. 
An advantage of Lemma \ref{coupling} over Theorem 3.1 from \cite{CCK2016} is that the second term of the estimate is proportional to $\varepsilon^{-2}$ in Lemma \ref{coupling}, while it is proportional to $\varepsilon^{-3}$ in Theorem 3.1 from \cite{CCK2016}. This difference generally leads a weaker condition and a better convergence rate in Gaussian approximation; see Remark \ref{rmk:spot} for details.
\end{rmk}


Combining Lemma \ref{coupling} with several technical tools developed in the CCK theory, we obtain the following main result of this paper, which can be considered as a generalization of Theorem 2 from \cite{CCK2015}:  
\begin{theorem}\label{kolmogorov}
(a) Suppose that $d\geq2$ and there are constants $\underline{\sigma},\overline{\sigma}>0$ such that $\underline{\sigma}^2\leq\mathfrak{C}(j,j)\leq\overline{\sigma}^2$ for all $j=1,\dots,d$. Set $a_d=E[\max_{1\leq j\leq d}(Z_j/\sqrt{\mathfrak{C}(j,j)})]$. Then 
\begin{equation}\label{eq:kolmogorov}
\sup_{x\in\mathbb{R}}\left|P(F_\vee\leq x)-P(Z_\vee\leq x)\right|
\leq C\Delta^{1/3}\left\{1\vee a_d^2\vee\log(1/\Delta)\right\}^{1/3}(\log d)^{1/3},
\end{equation}
where $C>0$ depends only on $\underline{\sigma}$ and $\overline{\sigma}$ (the right side is understood to be 0 if $\Delta=0$). 

\noindent(b) Suppose that $d\geq2$ and there is a constant $b>0$ such that $\mathfrak{C}(j,j)\geq b$ for all $j=1,\dots,d$. Then 
\begin{equation}\label{eq:kolmogorov2}
\sup_{x\in\mathbb{R}}\left|P(F_\vee\leq x)-P(Z_\vee\leq x)\right|
\leq C'\Delta^{1/3}(\log d)^{2/3},
\end{equation}
where $C'>0$ depends only on $b$. 
\end{theorem}

Since we have $\max_{1\leq j\leq d}|x_j|=\max\{\max_{1\leq j\leq d}x_j,\max_{1\leq j\leq d}(-x_j)\}$ for any real numbers $x_1,\dots,x_d$, we obtain the following result as a direct consequence of Theorem \ref{kolmogorov}:
\begin{corollary}\label{abs-kolmogorov}
Under the assumptions of Theorem \ref{kolmogorov}(b), we have
\begin{align*}
\sup_{x\in\mathbb{R}}\left|P\left(\max_{1\leq j\leq d}|F_j|\leq x\right)-P\left(\max_{1\leq j\leq d}|Z_j|\leq x\right)\right|
\leq C'\Delta^{1/3}(\log d)^{2/3},
\end{align*}
where $C'>0$ depends only on $b$.
\end{corollary}

In order to make Theorem \ref{kolmogorov} and Corollary \ref{abs-kolmogorov} useful, we need a reasonable bound for the quantity $\Delta$. When the random vector $F$ consists of multiple Wiener-It\^o integrals with common order, we have the following useful bound for $\Delta$:
\begin{lemma}\label{fourth-moment}
Let $q\geq2$ be an integer and suppose that $F_j=I_q(f_j)$ for some $f_j\in H^{\odot q}$ for $j=1,\dots,d$. Then we have
\[
\Delta\leq\max_{1\leq i,j\leq d}\left|\mathfrak{C}(i,j)-E[F_iF_j]\right|+C_q\log^{q-1}\left(2d^2-1+e^{q-2}\right)\max_{1\leq k\leq d}\sqrt{E[F_k^4]-3E[F_k^2]^2},
\]
where $C_q>0$ depends only on $q$.
\end{lemma}

\begin{rmk}
Lemma \ref{fourth-moment} implies that, in order to bound the Kolmogorov distance between $F_\vee$ and $Z_\vee$, we only need to control the convergence rate of the covariance matrix of $F$ to that of $Z$ and the fourth cumulants of the components of $F$, as long as $F$ consists of multiple Wiener-It\^o integrals with common order. This can be considered as a type of \textit{fourth moment phenomenon}, which was first discovered by \cite{NP2005} while they derived central limit theorems for sequences of multiple Wiener-It\^o integrals. For more information about the fourth moment phenomenon, we refer to \cite{NP2012} and references therein.  
\end{rmk}


It is often involved to compute the variables $L^{-1}F_j$ in the case that $F_j$'s are general Wiener functionals. It would be worth mentioning that we can avoid this issue if the variables $F_j$ are twice differentiable in the Malliavin sense and satisfy a suitable moment condition. To state such a result precisely, we make some definitions: 
For an $H\otimes H$-valued random variable $G$, we denote by $\|G\|_\text{op}$ the operator norm of the (random) operator $H\ni h\mapsto \langle h,G\rangle_H\in H$. 
Also, we say that a random variable $Y$ is \textit{sub-Gaussian relative to the scale $a>0$} if $E[e^{\lambda Y}]\leq \exp(\lambda^2a^2/2)$ for all $\lambda\in\mathbb{R}$. 
\begin{lemma}\label{poincare}
If $F_1,\dots,F_d\in\mathbb{D}_{2,4p}$ for a positive integer $p$, we have
\[
\Delta\leq\max_{1\leq i,j\leq d}\left|\mathfrak{C}(i,j)-E[F_iF_j]\right|
+d^{1/p}\sqrt{2p-1}\cdot\frac{3}{2}\left(\max_{1\leq i\leq d}\left\|\left\|D^2F_i\right\|_\text{op}\right\|_{4p}\right)\left(\max_{1\leq j\leq d}\left\|\left\|DF_j\right\|_H\right\|_{4p}\right).
\]
Moreover, if there is a constant $a>0$ such that both the variables $\left\|D^2F_i\right\|_\text{op}$ and $\left\|DF_i\right\|_H$ are sub-Gaussian relative to the scale $a$ for all $i=1,\dots,d$, we have
\begin{align*}
\Delta\leq\max_{1\leq i,j\leq d}\left|\mathfrak{C}(i,j)-E[F_iF_j]\right|
+Ca^2\log^{3/2}(2d^2-1+\sqrt{e}),
\end{align*}
where $C>0$ is a universal constant. 
\end{lemma}

\begin{rmk}
The above result (combined with Theorem \ref{kolmogorov}) can be viewed as an analogy of the so-called \textit{second-order Poincar\'e inequalities} proved in \citet{NPR2009}. Indeed, its proof is based on the lemmas proved there.
\end{rmk}

\section{Gaussian approximation of maxima of random symmetric quadratic forms}\label{sec:qf}

In this section we focus on the problem of approximating the distribution of maxima of symmetric quadratic forms. The next result can be easily derived from the results in the previous section:
\begin{theorem}\label{gqf}
For each $n\in\mathbb{N}$, let $\boldsymbol{\xi}_n$ be an $N_n$-dimensional centered Gaussian vector with covariance matrix $\Sigma_n=(\Sigma_n(k,l))_{1\leq k,l\leq N_n}$ and $d_n\geq2$ be an integer. Also, for each $k=1,\dots,d_n$, let $A_{n,k}$ be an $N_n\times N_n$ symmetric matrix and $Z_n=(Z_{n,1},\dots,Z_{n,d_n})^\top$ be an $d_n$-dimensional centered Gaussian vector with covariance matrix $\mathfrak{C}_n=(\mathfrak{C}_n(k,l))_{1\leq k,l\leq d_n}$. Set $F_{n,k}:=\boldsymbol{\xi}_n^\top A_{n,k}\boldsymbol{\xi}_n-E[\boldsymbol{\xi}_n^\top A_{n,k}\boldsymbol{\xi}_n]$ and suppose that the following conditions are satisfied:
\begin{enumerate}[label=(\roman*)]

\item There is a constant $b>0$ such that $\mathfrak{C}_n(k,k)\geq b$ for every $n$ and every $k=1,\dots,d_n$.

\item $\max_{1\leq k\leq d_n}(E[F_{n,k}^4]-3E[F_{n,k}^2]^2)\log^6d_n\to 0$ as $n\to\infty$.

\item $\max_{1\leq k,l\leq d_n}\left|\mathfrak{C}_n(k,l)-E[F_{n,k}F_{n,l}]\right|\log^2d_n\to0$ as $n\to\infty$.

\end{enumerate}
Then we have 
\begin{equation}\label{kol}
\sup_{x\in\mathbb{R}}\left|P\left(\max_{1\leq k\leq d_n}F_{n,k}\leq x\right)-P\left(\max_{1\leq k\leq d_n}Z_{n,k}\leq x\right)\right|
\to0
\end{equation}
and
\begin{align*}
\sup_{x\in\mathbb{R}}\left|P\left(\max_{1\leq k\leq d_n}|F_{n,k}|\leq x\right)-P\left(\max_{1\leq k\leq d_n}|Z_{n,k}|\leq x\right)\right|
\to0
\end{align*}
as $n\to\infty$.
\end{theorem}

\begin{rmk}
Since any symmetric Gaussian quadratic form can be written as a linear combination of independent $\chi^2$ random variables via eigenvalue decomposition (see e.g.~Section 3.2.1 of \cite{DY2011}), the readers may be wondering about whether it is possible to apply the original CCK theory to derive a similar result to Theorem \ref{gqf} using eigenvalue decomposition. This is however impossible in general because the matrices $\Sigma_n^{1/2}A_{n,1}\Sigma_n^{1/2},\dots,\Sigma_n^{1/2}A_{n,d_n}\Sigma_n^{1/2}$ are not necessarily simultaneously diagonalizable by an orthogonal matrix, which may induce an additional cross-sectional dependence after orthogonal transformation. To see this, suppose that $\Sigma_n$ is identity for simplicity. Then, the afore-mentioned eigenvalue decomposition argument reads as follows: For each $k=1,\dots,d_n$, we take an $N_n\times N_n$ real orthogonal matrix $U_{n,k}$ such that $U_{n,k}A_{n,k}U_{n,k}^\top$ is diagonal, and set $\bs{\varepsilon}_{n,k}=U_{n,k}\bs{\xi}_n$. Then the components of $\bs{\varepsilon}_{n,k}$ are independent and $F_{n,k}$ can be written as a linear combination of the squared components of $\bs{\varepsilon}_{n,k}$. However, for $k\neq l$, the covariance matrix of $\bs{\varepsilon}_{n,k}$ and $\bs{\varepsilon}_{n,l}$ is given by $U_{n,k}U_{n,l}^\top$, which is generally not diagonal; e.g.~we have
\[
\frac{1}{\sqrt{2}}\left(
\begin{array}{cc}
1  &  1  \\
1  &  -1
\end{array}
\right)\left(
\frac{1}{\sqrt{5}}\left(
\begin{array}{cc}
2  &  1  \\
1  &  -2
\end{array}
\right)
\right)^\top
=\frac{1}{\sqrt{10}}\left(
\begin{array}{cc}
3  &  -1  \\
1  &  3
\end{array}
\right).
\]
\end{rmk}

\begin{rmk}
Even if the matrices $\Sigma_n^{1/2}A_{n,1}\Sigma_n^{1/2},\dots,\Sigma_n^{1/2}A_{n,d_n}\Sigma_n^{1/2}$ are simultaneously diagonalizable, there is gain to use Theorem \ref{gqf} instead of the original CCK theory. To see this, suppose that each $F_{n,k}$ can be written as
\[
F_{n,k}=\sum_{i=1}^{N_n}\lambda_{n,k}(i)(\eta_i^2-1),
\]
where $\lambda_{n,k}(1),\dots,\lambda_{n,k}(N_n)\in\mathbb{R}$ and $(\eta_i)_{i=1}^\infty$ is a sequence of i.i.d.~standard normal variables. In this case, if we assume that there are constants $\overline{b},\underline{b}>0$ such that
\[
\underline{b}\leq \sum_{i=1}^{N_n}\lambda_{n,k}(i)^2\leq\overline{b}
\]
for all $n\in\mathbb{N}$ and $k=1,\dots,d_n$ and that the matrix $\mathfrak{C}_n$ is equal to the covariance matrix of the variables $F_{n,1},\dots,F_{n,d_n}$, Proposition 2.1 of \cite{CCK2017} yields the convergence \eqref{kol}, provided that $B_n^2\log^7(d_nN_n)=o(N_n)$ as $n\to\infty$, where 
\[
B_n=\sqrt{N_n}\max_{1\leq k\leq d_n}\max_{1\leq i\leq N_n}|\lambda_{n,k}(i)|.
\] 
Since we have
\begin{equation}\label{fourth}
\max_{1\leq k\leq d_n}(E[F_{n,k}^4]-3E[F_{n,k}^2]^2)\leq\frac{B_n^2}{N_n}\overline{b},
\end{equation}
the convergence \eqref{kol} is indeed implied by $B_n^2\log^6d_n=o(N_n)$ according to Theorem \ref{gqf}. 
In addition, the inequality \eqref{fourth} can be not tight. A cheap example is the case that
\[
\lambda_{n,k}(i)=\left\{
\begin{array}{ll}
1/N_n^{1/4}  & \text{if }i=k,     \\
1/\sqrt{N_n}  & \text{otherwise}.
\end{array}
\right.
\]
In this case we have $B_n^2/N_n=1/\sqrt{N_n}$, while it holds that
\[
\max_{1\leq k\leq d_n}(E[F_{n,k}^4]-3E[F_{n,k}^2]^2)=O(N_n^{-1}).
\]

See also Remark \ref{rmk:spot} for another advantage of using our results instead of the original CCK theory. 
\end{rmk}

\begin{rmk}[Discussion on the fourth moment condition (ii)]\mbox{}
\begin{enumerate}[label=(\roman*)]

\item In Theorem \ref{gqf}, the number $N_n$ does not necessarily diverge to get the convergence $E[F^4_{n,k}]-3E[F^2_{n,k}]^2\to 0$. This is because the variance of $\bs{\xi}_n$ is allowed to diverge in the setting of the theorem. To see this, suppose that $N_n=1$, $\bs{\xi}_n$ is a centered Gaussian variable with variance $n$, and $A_{n,k}=1/\sqrt{2n}$. In this case we have $F_{n,k}=(\bs{\xi}_n^2-n)/\sqrt{2n}$ and thus $E[F_{n,k}^2]=1$ and $E[F^4_{n,k}]-3E[F^2_{n,k}]^2=12/n\to0$.

\item If $\sup_{n\in\mathbb{N}}\max_{1\leq k\leq d_n}E[F_{n,k}^2]<\infty$, a sufficient condition to prove the condition (ii) of Theorem \ref{gqf} is $\max_{1\leq k\leq d_n}\|\Sigma_n^{1/2}A_{n,k}\Sigma_n^{1/2}\|_\mathrm{sp}\log^3d_n\to0$ as $n\to\infty$. This follows from the following inequality (see Eq.(11) of \cite{DY2011}):
\[
E[F^4_{n,k}]-3E[F^2_{n,k}]^2=48\trace\left[\left(\Sigma_n^{1/2}A_{n,k}\Sigma_n^{1/2}\right)^4\right]
\leq 24\|\Sigma_n^{1/2}A_{n,k}\Sigma_n^{1/2}\|_\mathrm{sp}^2E[F_{n,k}^2]
\]
(note that we always have $E[F^4_{n,k}]-3E[F^2_{n,k}]^2\geq0$; see Remark 5.2.5 of \cite{NP2012}). 
In practice, it is often easier to check the condition on $\|\Sigma_n^{1/2}A_{n,k}\Sigma_n^{1/2}\|_\mathrm{sp}$ than to directly check the condition on $E[F^4_{n,k}]-3E[F^2_{n,k}]^2$.

\item The condition $E[F^4_{n,k}]-3E[F^2_{n,k}]^2\to 0$ is necessary to approximate the distribution of the random variable $F_{n,k}$ by a Gaussian distribution if $\sup_{n\in\mathbb{N}}E[F_{n,k}^2]<\infty$ because there is a universal constant $c>0$ such that
\[
\sup_{n\in\mathbb{N}}E[F_{n,k}^8]\leq c\sup_{n\in\mathbb{N}}E[F_{n,k}^2]^4<\infty
\]
(see e.g.~Theorem 5.10 of \cite{Janson1997}), 
which implies the uniform integrability of the variables $F_{n,k}^2$ and $F_{n,k}^4$, $n=1,2,\dots$. Actually, adopting an analogous discussion to the one from \citet{CCK2017}, we can easily generalize the conclusion of Theorem \ref{gqf} to the convergence of the Kolmogorov distance between $F_n$ and $Z_n$ as follows:
\[
\sup_{x_1,\dots,x_{d_n}\in\mathbb{R}}\left|P\left(\bigcap_{k=1}^{d_n}\left\{F_{n,k}\leq x_k\right\}\right)-P\left(\bigcap_{k=1}^{d_n}\left\{Z_{n,k}\leq x_k\right\}\right)\right|
\to0.
\]
Therefore, if $\sup_{n\in\mathbb{N}}\max_{1\leq k\leq d_n}E[F_{n,k}^2]<\infty$, the condition $\max_{1\leq k\leq d_n}(E[F_{n,k}^4]-3E[F_{n,k}^2]^2)\log^6d_n\to 0$ as $n\to\infty$ is indeed a necessary condition \textit{when $d_n$ is fixed} (it is still unclear that this condition is necessary when $d_n\to\infty$ as $n\to\infty$, though). 

\end{enumerate}
\end{rmk}

In the next section we will apply Theorem \ref{gqf} to derive a Gaussian approximation of the null distribution of the test statistic for the absence of lead-lag effects. In order to implement the test in practice, we need to compute quantiles of the null distribution, but it is not easy to directly compute those of the derived Gaussian analog of the test statistic because its covariance structure contains unknown quantities for statisticians. For this reason we will apply a wild bootstrap procedure to approximately compute quantiles of the null distribution. Theorem \ref{gqf} is still applicable for ensuring the validity of such a procedure as long as \textit{Gaussian} wild bootstrapping is considered, but it turns out that a wild bootstrap procedure based on another distribution performs much better in finite samples. For this reason we partially generalize Theorem \ref{gqf} to a non-Gaussian case. 

For every $n\in\mathbb{N}$, let $N_n\geq1$ and $d_n\geq2$ be integers and let $\Gamma_{n,k}=(\gamma_{n,k}(i,j))_{1\leq i,j\leq N_n}$ be an $N_n\times N_n$ symmetric matrix for each $k=1,\dots,d_n$. We assume that $\gamma_{n,k}(i,i)=0$ for all $i=1,\dots,N_n$, $k=1,\dots,d_n$ and $n\in\mathbb{N}$. Given a sequence $\xi=(\xi_i)_{i=1}^\infty$ of random variables, we set
\[
Q_{n,k}(\xi):=\sum_{i,j=1}^{N_n}\gamma_{n,k}(i,j)\xi_i\xi_j,\qquad k=1,\dots,d_n
\]
for every $n\in\mathbb{N}$.

Let $Y=(Y_i)_{i=1}^\infty$ be a sequence of independent variables such that $E[Y_i]=0$ and $E[Y_i^2]=1$ for every $i$. Also, let $G=(G_i)_{i=1}^\infty$ be a sequence of independent standard Gaussian variables. For every $i\in\mathbb{N}$, we define the random variables $(W^{(i)}_j)_{j=1}^\infty$ by
\[
W^{(i)}_j=
\left\{
\begin{array}{ll}
Y_j  & \text{if }j\leq i,     \\
G_j  & \text{if }j>i.     
\end{array}
\right.
\]


\begin{theorem}\label{kolmogorov-general}
For each $n\in\mathbb{N}$, let $Z_n=(Z_{n,1},\dots,Z_{n,d_n})^\top$ be a $d_n$-dimensional centered Gaussian vector with covariance matrix $\mathfrak{C}_n=(\mathfrak{C}_n(k,l))_{1\leq k,l\leq d_n}$, and set
\begin{align*}
R_{n,1}&=\sum_{i=1}^{N_n}E\left[\max_{1\leq k\leq d_n}\left|\sum_{j=1}^{N_n}\gamma_{n,k}(i,j)W^{(i)}_j\right|^3\right](E[|Y_i|^{3}]+E[|G_i|^3]),\\
R_{n,2}&=\max_{1\leq k,l\leq d_n}\left|\mathfrak{C}_n(k,l)-E[Q_{n,k}(G)Q_{n,l}(G)]\right|,\\
R_{n,3}&=\max_{1\leq k\leq d_n}\sqrt{E[Q_{n,k}(G)^4]-3E[Q_{n,k}(G)^2]^2}.
\end{align*}
Suppose that there is a constant $b>0$ such that $\mathfrak{C}_n(k,k)\geq b$ for every $n$ and every $k=1,\dots,d_n$. Then we have
\[
\sup_{x\in\mathbb{R}}\left|P\left(\max_{1\leq k\leq d_n}|Q_{n,k}(Y)|\leq x\right)-P\left(\max_{1\leq k\leq d_n}|Z_{n,k}|\leq x\right)\right|\to0
\]
as $n\to\infty$, provided that $R_{n,1}\log^{\frac{7}{2}}d_n\vee R_{n,2}\log^2d_n\vee R_{n,3}\log^3d_n\to0$.
\end{theorem}

\begin{rmk}
The variables $W^{(i)}_j$ are related to the so-called \textit{Lindeberg method}. In fact, our proof of Theorem \ref{kolmogorov-general} is based on the generalized Lindeberg method developed in \cite{MOO2010,NPR2010aop} (see also Chapter 11 of \cite{NP2012}). 
\end{rmk}

\begin{rmk}
There is probably room for improvement in Theorem \ref{kolmogorov-general}. 
In particular, the truncation arguments used in the CCK theory (based on Lemma A.6 of \cite{CCK2013}) are apparently applicable to our case, which would significantly weaken the assumptions of Theorem \ref{kolmogorov-general}. 
On the other hand, it is less obvious whether the other techniques used in the CCK theory (especially in \citet{CCK2017}) are applicable to our case or not. Their excellent argument leads a very sharp bound, but  it seems crucial in their argument that the statistics considered there is a linear function of independent random variables. More precisely, to apply their argument to our case, the independence between the variables $U_i$ and $V_i$ appearing in the proof of Theorem \ref{kolmogorov-general} seems necessary, but this is not the case (such a structure is necessary to get an analogous estimate to Eq.(30) of \cite{CCK2017}, for example). 
This issue is left to future research. 
\end{rmk}

\begin{rmk}
Analogous quantities to $R_{n,2}$ and $R_{n,3}$ from Theorem \ref{kolmogorov-general} have already appeared in Theorem \ref{gqf} and it is usually not difficult to bound them. On the other hand, as long as the third moments of $Y_i$'s are uniformly bounded, the quantity $R_{n,1}$ is bounded by the third moment of the maximum of a sum of (high-dimensional) independent random vectors, so we have many inequalities which can be used to bound it (see e.g.~Chapter 14 of \cite{BvdG2011}). Here we give two examples of such inequalities. The first one only requires the uniform boundedness of the $p$-th moments of $Y_i$'s for some $p\geq3$, while the latter one is applicable when the variables $Y_i$ are sub-Gaussian. 
\end{rmk}
%
%
\begin{lemma}\label{lemma:hyper}\mbox{}
\begin{enumerate}[label=(\alph*)]

\item Suppose that $\sup_{i\in\mathbb{N}}\|Y_i\|_p<\infty$ for some $p\geq3$. Then
\begin{align*}
\sum_{i=1}^{N_n}E\left[\max_{1\leq k\leq d_n}\left|\sum_{j=1}^{N_n}\gamma_{n,k}(i,j)W^{(i)}_j\right|^3\right]
\leq 2d_n^{3/p}(p-1)^{3/2}\sup_{i\in\mathbb{N}}\|Y_i\|_p^3\sum_{i=1}^{N_n}\max_{1\leq k\leq d_n}\left(\sum_{j=1}^{N_n}\gamma_{n,k}(i,j)^2\right)^{3/2}
\end{align*}
for every $n$.

\item Suppose that there is a constant $a>0$ such that $Y_i$ is sub-Gaussian relative to the scale $a$ for all $i=1,2,\dots$. Then
\[
\sum_{i=1}^{N_n}E\left[\max_{1\leq k\leq d_n}\left|\sum_{j=1}^{N_n}\gamma_{n,k}(i,j)W^{(i)}_j\right|^3\right]
\leq
5^{3/2}a^3\log^{3/2}(2d_n-1+\sqrt{e})\sum_{i=1}^{N_n}\max_{1\leq k\leq d_n}\left(\sum_{j=1}^{N_n}\gamma_{n,k}(i,j)^2\right)^{3/2}
\]
for every $n$.

\end{enumerate}
\end{lemma}

Using the above lemma, we obtain a useful criterion to check the conditions appearing in Theorem \ref{kolmogorov-general} in terms of the so-called \textit{influence indices}: Given a symmetric matrix $\Gamma=(\gamma(i,j))_{1\leq i,j\leq N}$, the \textit{influence} of the variable $i$ of $\Gamma$ is defined by
\[
\influence_i(\Gamma)=\sum_{j=1}^N\gamma(i,j)^2
\]
for $i=1,\dots,N$. The influence indices play an important role in studies of the central limit theorem for random quadratic forms (and homogeneous sums more generally); see \cite{NPR2010aop,MOO2010,GT1999} for example.
\begin{corollary}\label{criterion}
Suppose that there is a constant $a>0$ such that $Y_i$ is sub-Gaussian relative to the scale $a$ for all $i=1,2,\dots$. 
Then the convergences $R_{n,1}\log^{\frac{7}{2}}d_n\to0$ and $R_{n,3}\log^3d_n\to0$ are implied by the following condition:
\begin{equation}\label{influence}
(\log d_n)^6\max_{1\leq k\leq d_n}\trace\left(\Gamma_{n,k}^4\right)+(\log d_n)^5\left(\max_{1\leq i\leq N_n}\sqrt{\Lambda_{n,i}}\right)\sum_{i=1}^{N_n}\Lambda_{n,i}\to0\quad\text{as }n\to\infty,
\end{equation}
where
\[
\Lambda_{n,i}=\max_{1\leq k\leq d_n}\influence_i(\Gamma_{n,k}),\qquad
i=1,\dots,N_n.
\]
\end{corollary}

\begin{rmk}[Implication of the condition \eqref{influence}]
Let us consider the case that there is a symmetric matrix $\bar{\Gamma}_n=(\bar{\gamma}_n(i,j))_{1\leq i,j\leq N_n}$ such that $\influence_i(\bar{\Gamma}_n)=\Lambda_{n,i}$ for all $i=1,\dots,N_n$. Namely, the influence indices of the matrices $\Gamma_{n,1},\dots,\Gamma_{n,d_n}$ are dominated by that of the matrix $\bar{\Gamma}_n$. In this case the condition \eqref{influence} reads as
\[
(\log d_n)^6\max_{1\leq k\leq d_n}\trace\left(\Gamma_{n,k}^4\right)+(\log d_n)^{10}\|\bar{\Gamma}_n\|_F^4\max_{1\leq i\leq N_n}\influence_i(\bar{\Gamma}_n)\to0\quad\text{as }n\to\infty.
\]
The quantity $\|\bar{\Gamma}_n\|_F^2$ is the variance of the quadratic form
\[
\bar{Q}_n(Y)=\sum_{i,j=1}^{N_n}\bar{\gamma}_n(i,j)Y_iY_j.
\] 
Therefore, it would be natural to assume $\sup_{n\in\mathbb{N}}\|\bar{\Gamma}_n\|_F^2<\infty$. Moreover, in many cases it is reasonable to expect $\max_{1\leq i\leq d_n}\influence_i(\bar{\Gamma}_n)=O(N_n^{-1})$ because we have by definition 
\[
\max_{1\leq i\leq d_n}\influence_i(\bar{\Gamma}_n)\leq\|\bar{\Gamma}_n\|^2_\mathrm{sp}.
\]
According to \cite{GT1999}, $\|\bar{\Gamma}_n\|_\mathrm{sp}$ gives an optimal convergence rate for the Kolmogorov distance between 
\[
\bar{Q}_n(Y)\left/\sqrt{\variance[\bar{Q}_n(Y)]}\right.
\] 
and a standard Gaussian variable, 
hence it seems reasonable to expect $\|\bar{\Gamma}_n\|_\mathrm{sp}=O(N_n^{-1/2})$ in view of the standard Berry-Esseen inequality. 
Moreover, since $\trace(\Gamma_{n,k}^4)\leq\|\Gamma_{n,k}\|_\mathrm{sp}^2\|\Gamma_{n,k}\|_F^2$, we might expect $\max_{1\leq k\leq d_n}\trace(\Gamma_{n,k}^4)=O(N_n^{-1})$ due to a similar reason. 
Consequently, the condition \eqref{influence} is typically satisfied when $\log^{10} d_n=o(N_n)$ as $n\to\infty$. 

A typical example satisfying the above condition is the situations where $\Gamma_{n,k}$'s correspond to sample auto-covariances:
\[
\gamma_{n,k}(i,j)=\left\{
\begin{array}{ll}
1/\sqrt{N_n} & \text{if }|j-i|=k,   \\
0  & \text{otherwise}. 
\end{array}
\right.
\]
In this case the quantities $\sum_{j=1}^{N_n}\gamma_{n,k}(i,j)^2$ does not depend on $k$, so we can take $\bar{\Gamma}_n=\Gamma_{n,1}$ for example. 
\end{rmk}

\section{Application to high-frequency data}\label{sec:applications}

\subsection{Testing the absence of lead-lag effects}\label{sec:test}

We turn to the problem of testing the absence of lead-lag effects which is mentioned at the beginning of the Introduction. Here we consider a more general setting than the one described in the Introduction by allowing (deterministic) time-varying volatilities as well as the presence of multiple lead-lag times under the alternative. 
 
Let $\rho_1,\dots,\rho_M$ be real numbers satisfying the condition $\sum_{m=1}^M|\rho_m|<1.$ 
Also, let $\theta_1,\dots,\theta_M$ be mutually different numbers. Then, by Proposition 2 from \cite{HK2018} there is a bivariate Gaussian process $B_t=(B^1_t,B^2_t)$ $(t\in\mathbb{R})$ with stationary increments such that both $B^1$ and $B^2$ are standard Brownian motions as well as $B^1$ and $B^2$ have the cross spectral density given by
\[
\mathfrak{s}(\lambda)=\sum_{m=1}^M\rho_me^{-\sqrt{-1}\theta_m\lambda},\qquad\lambda\in\mathbb{R}.
\] 
This means that we have
\[
E\left[\left(\int_{-\infty}^\infty f(t)dB^1_t\right)\left(\int_{-\infty}^\infty g(t)dB^2_t\right)\right]
=\sum_{m=1}^M\rho_m\int_{-\infty}^\infty f(t)g(t+\theta_m)dt
\]
for any $f,g\in L^2(\mathbb{R})$.

For each $\nu=1,2$, we consider the process $X^\nu=(X^\nu_t)_{t\geq0}$ given by
\begin{equation}\label{model}
X^\nu_t=X^\nu_0+\int_0^t\sigma_\nu(s)dB^\nu_s,\qquad t\geq0,
\end{equation}
where $\sigma_\nu\in L^2(0,\infty)$ is nonnegative-valued and deterministic. We observe the process $X^\nu$ on the interval $[0,T]$ at the deterministic sampling times $0\leq t^\nu_0<t^\nu_1<\cdots<t^\nu_{n_\nu}\leq T$, which implicitly depend on the parameter $n\in\mathbb{N}$ such that
\[
r_n:=\max_{\nu=1,2}\max_{i=0,1,\dots,n_\nu+1}(t^\nu_i-t^\nu_{i-1})\to0
\]
as $n\to\infty$, where we set $t^\nu_{-1}:=0$ and $t^\nu_{n_\nu+1}:=T$ for each $\nu=1,2$.

\begin{rmk}\label{non-ergodic}
It is not difficult to extend the following discussion to the case that the volatilities $\sigma_1,\sigma_2$ and the sampling times $(t^1_i)_{i=0}^{n_1},(t^2_j)_{j=0}^{n_2}$ are random but independent of the process $B$, but we focus on the deterministic case for the simplicity of notation. Extension to a situation where the volatilities depend on $B$ is non-trivial because of the non-ergodic nature of the problem (i.e.~the asymptotic covariance matrix of the statistics $(U_n(\theta))_{\theta\in\mathcal{G}_n}$ defined below generally depends on $B$) and we leave it to future research.
\end{rmk}

Our aim is to construct a testing procedure for the following statistical hypothesis testing problem based on discrete observation data $(X^1_{t_i^1})_{i=0}^{n_1}$ and $(X^2_{t_i^2})_{i=0}^{n_2}$:
\begin{equation}\label{llag-test}
H_0:\rho_m=0\text{ for all }m=1,\dots,M\qquad
\text{vs}\qquad 
H_1:\rho_m\neq0\text{ for some }m=1,\dots,M.
\end{equation}

We introduce some notation. For each $\nu=1,2$, we associate the observation times $(t^\nu_i)_{i=0}^{n_\nu}$ with the collection of intervals $\Pi^\nu_n=\{(t^\nu_{i-1},t^\nu_i]:i=1,\dots,n_\nu\}$. We will systematically employ the notation $I$ (resp.~$J$) for an element of $\Pi^1_n$ (resp.~$\Pi^2_n$). 
For an interval $S\subset[0,\infty)$, we set $\overline{S}=\sup S$, $\underline{S}=\inf S$ and $|S|=\overline{S}-\underline{S}$. In addition, we set $V(S)=V_{\overline{S}}-V_{\underline{S}}$ for a a stochastic process $(V_t)_{t\geq0}$, and $S_\theta=S+\theta$ for a real number $\theta$. We define the Hoffmann-Rosenbaum-Yoshida cross-covariance estimator by
\[
U_n(\theta)=
\sum_{I\in\Pi^1_n,J\in\Pi^2_n}X^1(I)X^2(J)K(I,J_{-\theta}),
\]
where we set $K(I,J)=1_{\{I\cap J\neq\emptyset\}}$ for two intervals $I$ and $J$. 
Now our test statistic is given by
\[
T_n=\sqrt{n}\max_{\theta\in\mathcal{G}_n}|U_n(\theta)|,
\]
where $\mathcal{G}_n$ is a finite subset of $\mathbb{R}$. 

To establish the asymptotic property of our test statistic $T_n$, we first investigate the asymptotic property of the following quantity: 
\[
F_n(\theta)=\sqrt{n}(U_n(\theta)-E[U_n(\theta)]).
\]
We impose the following conditions:
\begin{enumerate}[label={\normalfont[A\arabic*]}]

\item\label{vol-bound} $\sup_{t\in[0,T]}(\sigma_1(t)+\sigma_2(t))<\infty$.


\item\label{avar-bound} There are positive constants $\underline{v},\overline{v}$ such that $\underline{v}\leq V_n(\theta)\leq\overline{v}$ for all $n\in\mathbb{N}$ and $\theta\in\mathcal{G}_n$, where
\[
V_n(\theta)=
n\sum_{I\in\Pi^1_n,J\in\Pi^2_n}\left(\int_{I}\sigma_1(t)^2dt\right)\left(\int_{J}\sigma_2(t)^2dt\right)K(I,J_{-\theta}).
\]

\item\label{cross-vol} $\Sigma(\theta_m)>0$ for all $m=1,\dots,M$, where
\[
\Sigma(\theta)=
\left\{
\begin{array}{ll}
\int_0^{T-\theta}\sigma_1(t)\sigma_2(t+\theta)dt  & \text{if }\theta\geq0,  \\
\int_{0}^{T+\theta}\sigma_1(t-\theta)\sigma_2(t)dt  & \text{if }\theta<0.
\end{array}
\right.
\]

\item\label{grid} The grid $\mathcal{G}_n$ satisfies the following conditions:
\begin{enumerate}[label=(\roman*)]

\item There is a constant $\gamma>0$ such that $\#\mathcal{G}_n=O(n^\gamma)$ as $n\to\infty$.

\item There is a sequence $(\upsilon_n)_{n\in\mathbb{N}}$ of positive numbers such that
\[
\{\theta_1,\dots,\theta_M\}\subset\bigcup_{\theta\in\mathcal{G}_n}[\theta-\upsilon_n,\theta+\upsilon_n]
\]
and $\lim_{n\to\infty}\upsilon_n\min\{n_1,n_2\}=0$.

\end{enumerate}

\end{enumerate}

\begin{rmk}
Assumption \ref{vol-bound} is standard in the literature and satisfied when $\sigma_1$ and $\sigma_2$ are c\`adl\`ag, for example. \ref{avar-bound} roughly says that the scaling factor $\sqrt{n}$ is appropriate (the quantity $V_n(\theta)$ is related to the variance of $U_n(\theta)$). \ref{avar-bound} holds true e.g.~when $0<\inf_{t\in[0,T]}\sigma_\nu(t)\leq \sup_{t\in[0,T]}\sigma_\nu(t)<\infty$ for every $\nu=1,2$, $n\sum_{I\in\Pi^1_n}|I|^2+n\sum_{J\in\Pi^2_n}|J|^2=O(1)$ as $n\to\infty$ and there is a constant $c>0$ such that $n(|I|\wedge|J|)\geq c$ for every $n$ and all $I\in\Pi^1_n$, $J\in\Pi^2_n$. \ref{cross-vol} ensures that $\max_{1\leq m\leq M}|E[U_n(\theta_m)]|$ does not vanish under $H_1$. \ref{grid} ensures that $\mathcal{G}_n$ is sufficiently fine to capture the cross-covariance at the lag $\theta_m$ for every $m$. Note that \ref{grid} is also assumed in \cite{HRY2013} (see Assumption B3 of \cite{HRY2013}).  
\end{rmk}

\begin{proposition}\label{hry}
For each $n\in\mathbb{N}$, let $(Z_n(\theta))_{\theta\in\mathcal{G}_n}$ be a family of centered Gaussian variables such that $E[Z_n(\theta)Z_n(\theta')]=E[F_n(\theta)F_n(\theta')]$ for all $\theta,\theta'\in\mathcal{G}_n$. 
Under assumptions \ref{vol-bound}--\ref{avar-bound}, we have
\[
\sup_{x\in\mathbb{R}}\left|P\left(\max_{\theta\in\mathcal{G}_n}|F_n(\theta)|\leq x\right)-P\left(\max_{\theta\in\mathcal{G}_n}|Z_n(\theta)|\leq x\right)\right|\to0
\]
as $n\to\infty$, provided that $nr_n^2\log^6(\#\mathcal{G}_n)\to0$.
\end{proposition}

\begin{rmk}
It is impossible to apply the original CCK theory (at least naively) to prove Proposition \ref{hry} because we need to apply Theorem \ref{gqf} to a situation where the matrices $\Sigma_n^{1/2}A_{n,1}\Sigma_n^{1/2},\dots,\Sigma_n^{1/2}A_{n,d_n}\Sigma_n^{1/2}$ are not simultaneously diagonalizable. In fact, if we consider the synchronous and equidistant sampling with the step size $1/n$, the matrices corresponding to $U_n(\pm1/n)$ are of the form
\[
\left(\begin{array}{cc}
O & A \\
A^\top & O
\end{array}\right),
\]
where we take the matrix $A=(a_{ij})$ as
\[
a_{ij}=\left\{\begin{array}{ll}
1 & \text{if }j-i=\pm1, \\
0 & \text{otherwise}.
\end{array}\right.
\] 
We can easily check that those matrices are not commutative unless the size of $A$ is less than or equal to 2. 
\end{rmk}

The above proposition suggests that the null distribution of our test statistic $T_n$ could be approximated by that of $\max_{\theta\in\mathcal{G}_n}|Z_n(\theta)|$ for sufficiently large $n$. However, it is not easy to evaluate the distribution of $\max_{\theta\in\mathcal{G}_n}|Z_n(\theta)|$ directly, so we rely on a (wild) bootstrap procedure to construct critical regions for our test. The above Gaussian approximation result plays a role in validating the bootstrap procedure.  

Let $(w^1_I)_{I\in\Pi^1_n}$ and $(w^2_J)_{J\in\Pi^2_n}$ be mutually independent sequence of i.i.d.~random variables which are independent of the processes $X^1$ and $X^2$. Then we set
\[
U_n^*(\theta)=
\sum_{I\in\Pi^1_n,J\in\Pi^2_n}\left(w^1_IX^1(I)\right)\left(w^2_JX^2(J)\right)K(I,J_{-\theta}).
\]
Given a significance level $\alpha$, we denote by $q_n^*(1-\alpha)$ the $100(1-\alpha)$\% quantile of the bootstrapped test statistic $T_n^*=\sqrt{n}\max_{\theta\in\mathcal{G}_n}|U_n^*(\theta)|$, conditionally on $X^1$ and $X^2$:
\[
q_n^*(1-\alpha)=\inf\left\{z\in\mathbb{R}:P\left(T_n^*\leq z|\mathcal{F}^X\right)\geq1-\alpha\right\},
\]
where $\mathcal{F}^X$ is the $\sigma$-field generated by the processes $X^1$ and $X^2$.
\begin{rmk}
We generate the bootstrap observations under the null hypothesis $H_0$. This is a typical approach in the bootstrap test literature (see e.g.~\cite{BHLP2016}). Moreover, as discussed in Section 4 of \cite{DM1999} as well as Section 2 of \cite{DF2008}, this approach often serves as refining the performance of the test. 
\end{rmk}

\begin{proposition}\label{hry-boot}
Suppose that \ref{vol-bound}--\ref{grid} are satisfied. Suppose also that $E[w_I^1]=E[w^2_J]=0$, $E[(w^1_I)^2]=E[(w^2_J)^2]=1$ for all $I,J$ and there is a constant $a>0$ such that both $w_I^1$ and $w^2_J$ are sub-Gaussian relative to the scale $a$ for all $I,J$. Suppose further that $r_n=O(n^{-3/4-\eta})$ as $n\to\infty$ for some $\eta>0$. Then the following statements hold true for all $\alpha\in(0,1)$:
\begin{enumerate}[label=(\alph*)]

\item\label{null} Under $H_0$, we have $P\left(T_n\geq q_n^*(1-\alpha)\right)\to\alpha$ as $n\to\infty$. 

\item\label{alternative} Under $H_1$, we have $P\left(T_n\geq q_n^*(1-\alpha)\right)\to1$ as $n\to\infty$.

\end{enumerate}
\end{proposition}

By Proposition \ref{hry-boot}, given a significance level $\alpha\in(0,1)$, we obtain a consistent and asymptotically level $\alpha$ test for \eqref{llag-test} by rejecting the null hypothesis if $T_n\geq q_n^*(1-\alpha)$. Of course, in the practical implementation we replace $q_n^*(1-\alpha)$ by a simulated one. For example, given observation data, we generate i.i.d.~copies $T_n^*(1),\dots,T_n^*(R)$ of $T_n^*$ (conditionally on the observation data) with some sufficiently large integer $R$. Then we replace the function $P\left(T_n^*\leq z|\mathcal{F}^X\right)$ of $z$ by its empirical counterpart $\frac{1}{R}\sum_{r=1}^R1_{\{T_n^*(r)\leq z\}}$ and compute $q_n^*(1-\alpha)$ accordingly. Note that this is equivalent to computing the \textit{bootstrap p-value} $\hat{p}^*=\frac{1}{R}\sum_{r=1}^R1_{\{T_n^*(r)> T_n\}}$ and rejecting the null hypothesis if $\hat{p}^*\leq\alpha$.
\begin{rmk}
The proposed test is evidently invariant under multiplying a constant. In particular, the factor $\sqrt{n}$ can be dropped when we implement the test in practice. 
\end{rmk}

\begin{rmk}[Choice of the multiplier variables]\label{wild}
Choice of the distribution of the multiplier variables $(w^1_I)_{I\in\Pi^1_n}$ and $(w^2_J)_{J\in\Pi^2_n}$ are important for the finite sample property of the test. In our situation it turns out that choosing Rademacher variables induces a quite good finite sample performance of our testing procedure. Namely, the proposed test performs very well in finite samples when the distributions of $w^1_I$ and $w^2_J$ are chosen according to
\[
P(w^1_I=1)=P(w^1_I=-1)=P(w^2_J=1)=P(w^2_J=-1)=\frac{1}{2}.
\]
This is presumably because the above choice makes the \textit{unconditional} distribution of the bootstrapped test statistics of $T_n^*$ coincide with the distribution of $T_n$. This can be shown in the same way as in the proof of Theorem 1 from \cite{DF2008}. For this reason we recommend that we should use Rademacher variables as the multiplier variables for the above testing procedure (and we do so in the simulation study of Section \ref{sec:simulation}). 
\end{rmk}

\subsection{Uniform confidence bands for spot volatility}

To illustrate another potential application of our main results, we present an application of our result to constructing uniform confidence bands for spot volatility. This section is only for an illustration purpose, so we do neither pursue the generality of the theory nor discuss practical problems on implementation such as the choice of a bandwidth and a kernel function. We refer to Section 6 of \cite{Kristensen2010} for a discussion on the latter issue.

Let us consider the stochastic process $X=(X_t)_{t\in[0,T]}$ which is defined on a stochastic basis $(\Omega,\mathcal{F},(\mathcal{F}_t)_{t\in[0,T]},P)$ and of the form
\[
X_t=X_0+\int_0^t\sigma(s)dB_s,\qquad t\in[0,T],
\]
where $B=(B_t)_{t\in[0,T]}$ is a standard $(\mathcal{F}_t)$-Brownian motion and $\sigma=(\sigma(t))_{t\in[0,T]}$ is a continuous $(\mathcal{F}_t)$-adapted process. The aim of this section is to construct uniform confidence bands for the \textit{spot volatility} $\sigma^2$ based on the high-frequency observation data $\{X_{t_i}\}_{i=0}^n$, where $t_i=Ti/n$, $i=0,1,\dots,n$. 

Specifically, we consider the following kernel-type estimator for $\sigma^2$ (cf.~\cite{Kristensen2010,FW2008}):
\[
\widehat{\sigma}_n^2(t):=\sum_{i=1}^nK_h(t_{i-1}-t)(X_{t_i}-X_{t_{i-1}})^2,\qquad t\in[0,T],
\]
where $h:=h_n>0$ is a bandwidth parameter, $K_h(x)=K(x/h)/h$ for $x\in\mathbb{R}$ and $K:\mathbb{R}\to\mathbb{R}$ is a kernel function. We derive a Gaussian approximation result for the supremum of the Studentization of $\widehat{\sigma}_n^2(t)$. Let us set 
\[
\mathfrak{s}_n(t)=\sqrt{\frac{2}{n^2}\sum_{i=1}^nK_h(t_{i-1}-t)^2}
\]
for $t\in[0,T].$ In view of Theorem 3 from \cite{Kristensen2010}, $\sigma^2(t)\mathfrak{s}_n(t)$ can be seen as an approximation of the asymptotic standard error of $\widehat{\sigma}^2_n(t)$.  
We define the Gaussian analog of the Studentization of $\widehat{\sigma}_n^2(t)$ as follows. For each $n\in\mathbb{N}$, let $(z_i^n)_{i=1}^n$ be a sequence of i.i.d.~centered Gaussian variables with variance $2/n^2$, and we set
\[
Z_n(t)=\frac{1}{\mathfrak{s}_n(t)}\sum_{i=1}^nK_h(t_{i-1}-t)z_i^n,\qquad t\in[0,T].
\]
We impose the following conditions:
\begin{enumerate}[label={\normalfont[B\arabic*]}]

\item\label{vol-cond} $w(\sigma;\eta)=O_p(\eta^\gamma)$ as $\eta\to0$ for some $\gamma\in(0,1]$. Moreover, $\sigma^2(t)>0$ for all $t\in[0,T]$ almost surely.

\item\label{kernel} The kernel function $K:\mathbb{R}\to\mathbb{R}$ is Lipschitz continuous and compactly supported as well as satisfies $\int_{-\infty}^\infty K(t)dt=1$.

\end{enumerate}
We also impose the following strengthened version of assumption \ref{vol-cond} when deriving the convergence rate of Gaussian approximation: 
\begin{enumerate}[label={\normalfont[SB\arabic*]}]

\item\label{s-vol-cond} There is a constant $\Lambda>0$ such that $\Lambda^{-1}\leq|\sigma(t)|\leq\Lambda$ and $w(\sigma;\eta)\leq\Lambda \eta^\gamma$ for all $t\in[0,T]$ and $\eta\in(0,1)$. 

\end{enumerate}

\begin{proposition}\label{spot}
Suppose that \ref{vol-cond}--\ref{kernel} are satisfied. Suppose also that the bandwidth parameter $h$ satisfies $nh^{1+2\gamma}\log n\to0$ and $\log^6n/nh\to0$ as $n\to\infty$. Let $a_n$ be a sequence of positive numbers such that $a_n\to0$ and $a_n/h\to\infty$ as $n\to\infty$. Then we have
\begin{equation}\label{eq:spot}
\sup_{x\in\mathbb{R}}\left|P\left(\sup_{t\in[a_n,T-a_n]}\left|\frac{\widehat{\sigma}_n^2(t)-\sigma^2(t)}{\sigma^2(t)\mathfrak{s}_n(t)}\right|\leq x\right)-P\left(\sup_{t\in[a_n,T-a_n]}|Z_n(t)|\leq x\right)\right|\to0
\end{equation}
as $n\to\infty$. 
Moreover, if we additionally assume \ref{s-vol-cond}, we have
\begin{multline}\label{spot:rate}
\sup_{x\in\mathbb{R}}\left|P\left(\sup_{t\in[a_n,T-a_n]}\left|\frac{\widehat{\sigma}_n^2(t)-\sigma^2(t)}{\sigma^2(t)\mathfrak{s}_n(t)}\right|\leq x\right)-P\left(\sup_{t\in[a_n,T-a_n]}|Z_n(t)|\leq x\right)\right|\\
=O\left(\sqrt{nh^{1+2\gamma}\log n}\right)+O\left(\frac{\log n}{(nh)^{\frac{1}{6}}}\right)
\end{multline}
as $n\to\infty$. 
\end{proposition}

\begin{rmk}
We introduce the parameters $a_n$ in Proposition \ref{spot} to avoid boundary effects. See Section 4 of \cite{Kristensen2010} for more details about this topic. 
\end{rmk}

\begin{rmk}\label{rmk:spot}
Although we use Lemma \ref{coupling} to prove Proposition \ref{spot} (see the proof of Lemma \ref{spot-coupling}), we can indeed use Theorem 3.1 of \cite{CCK2016} instead to derive a similar result. However, the result requires a (slightly) stronger condition on the bandwidth $h$ and leads to a worse convergence rate. 
In fact, an inspection of the proof of Proposition 2.1 from \cite{CCK2017} implies that we need to replace $\varepsilon^{-2}$ by $\varepsilon^{-3}$ in the inequality \eqref{eq:spot-coupling} if we use Theorem 3.1 of \cite{CCK2016} instead of Lemma \ref{coupling} to prove Lemma \ref{spot-coupling}. Then, the optimal choice of $\varepsilon$ in the proof of Proposition \ref{spot} becomes $\varepsilon=(nh)^{-\frac{1}{8}}\log^{3/8} n$, which changes the order of the second term on the right side of \eqref{spot:rate} to $O((nh)^{-\frac{1}{8}}\log^{7/8} n)$. Hence we need the condition $\log^7 n/nh\to0$ as $n\to\infty$ to get the convergence \eqref{eq:spot}.
\end{rmk}

In contrast to the previous subsection, the process $Z_n(t)$ does not contain any unknown parameter, so Proposition \ref{spot} is readily applicable to construction of uniform confidence bands for $\sigma^2$: Given a significance level $\alpha\in(0,1)$, let $q_n(1-\alpha)$ be the $100(1-\alpha)$\% quantile of the variable $\sup_{t\in[a_n,T-a_n]}|Z_n(t)|$ (which can be computed e.g.~by simulation). Then, 
\[
\left[\frac{\widehat{\sigma}_n^2(t)}{1+\mathfrak{s}_n(t)\cdot q_n(1-\alpha)},\frac{\widehat{\sigma}_n^2(t)}{1-\mathfrak{s}_n(t)\cdot q_n(1-\alpha)}\right],\qquad t\in[a_n,T-a_n],
\]
give asymptotically uniformly valid $100(1-\alpha)$\% confidence bands for $\sigma^2(t)$, $t\in[a_n,T-a_n]$.

\begin{rmk}
The applications considered in this section concerns asymptotic settings where the terminal value $T$ of the sampling interval is fixed. Here, we briefly discuss applicability of our theory to asymptotic settings where the terminal value $T$ of the sampling interval tends to infinity. 
In such a setting, a typical problem which our theory seems to fit would be constructing uniform confidence bands for the coefficient functions of an ergodic diffusion process. Non-parametric estimation of the coefficient functions of a diffusion process from high-frequency data is extensively studied in the literature, but most studies focus only on point-wise inference (except for \citet{Kanaya2017}, where uniform convergence rates of kernel-based estimators have been derived; see also \citet{ST2016} where the authors construct uniform confidence bands for the drift coefficient of a diffusion process based on low-frequency observation data), so it would be important to consider this problem. 
In such a problem, estimators typically have deterministic asymptotic covariance matrices, hence the issue indicated in Remark \ref{non-ergodic} does not arise. 
Unfortunately, however, we encounter another issue that it seems difficult (at least not straightforward) to get a reasonable estimate for the quantity $\Delta$ in this problem. This is perhaps because we do not take account of special properties of the underlying diffusion process (such as the Markov and mixing properties) when deriving our estimate. 
Therefore, this issue might be resolved by adopting the approach from \citet{KY2000} where Malliavin calculus is locally applied to the underlying process with taking account of the Markov and mixing properties. However, a rigorous treatment of this idea is rather demanding, so we leave it to future work. 
\end{rmk}

\section{Numerical illustration}\label{sec:simulation}

In this section we illustrate the finite sample performance of the testing procedure for the absence of lead-lag effects, which is proposed in Section \ref{sec:test}.\footnote{The proposed testing procedure is implemented in the \textsf{R} package \textbf{yuima} as the function \texttt{llag.test} since version 1.7.2.} 
The setting of our numerical experiments is basically adopted from Section 5 of \cite{HRY2013}. 
Specifically, we simulate model \eqref{hry-model} with $T=1,\vartheta=0.1,x_0^1=x_0^2=0,\sigma_1=\sigma_2=1$. We vary the correlation parameter as $\rho\in\{0,0.25,0.5,0.75\}$ to examine the size and the power of the testing procedure.  
We consider both synchronous and non-synchronous sampling scenarios. For the synchronous sampling scenario $t^1_i=t^2_i=ih_n$, $i=0,1,\dots,\lfloor Th_n^{-1}\rfloor$, we examine three kinds of sampling frequencies: $h_n\in\{10^{-3},3\cdot 10^{-3},6\cdot 10^{-3}\}$. Also, in these scenarios we set $\mathcal{G}_n=\{kh_n:k\in\mathbb{Z},|kh_n|\leq0.3\}$ as the search grid. 
On the other hand, for the non-synchronous sampling scenario, we first simulate the processes on the equidistant times $i\cdot10^{-3}$, $i=0,1,\dots,1,000$, then we randomly pick 300 sampling times for $X^1$. We do so for $X^2$ independently of the sampling for $X^1$. In this scenario we set $\mathcal{G}_n=\{k\cdot10^{-3}:k\in\mathbb{Z},|k|\leq300\}$ as the search grid. 
For the testing procedure, we use Rademacher variables as the multiplier variables and 999 bootstrap replications to construct the critical regions.  
We run 10,000 Monte Carlo iterations in each experiment. 

Table \ref{table:simulation} reports the rejection rates of the proposed test in each experiment. For the case $\rho=0$, the numbers should be close to the corresponding significance levels, and this is true for all the experiments. Turning to the power performance, we find that in the low correlation case $\rho=0.25$ the power of the test is rather weak except for the most frequent sampling scenario. This is reasonable in view of the simulation results reported in \cite{HRY2013}, which indicate that the contrast function $U_n(\theta)$ becomes flat in that case. For the moderate and the high correlation cases $\rho=0.5$ and $\rho=0.75$, the power of the test is satisfactory. 


\begin{table}[ht]
\caption{Rejection rates of the proposed test}
\label{table:simulation}
\begin{center}
\begin{tabular}{lrrrr}
  \hline
 & $\rho=0$ & $\rho=0.25$ & $\rho=0.50$ & $\rho=0.75$ \\ \hline
 \multicolumn{5}{c}{Synchronous sampling scenario}  \\ 
  $h_n=10^{-3}$ &  &  &  &  \\ 
  $\alpha=0.01$ & 0.011 & 1.000 & 1.000 & 1.000 \\ 
  $\alpha=0.05$ & 0.050 & 1.000 & 1.000 & 1.000 \\ 
  $\alpha=0.10$ & 0.100 & 1.000 & 1.000 & 1.000 \\ 
  $h_n=3\cdot10^{-3}$ &  &  &  &  \\ 
  $\alpha=0.01$ & 0.010 & 0.139 & 0.977 & 1.000 \\ 
  $\alpha=0.05$ & 0.051 & 0.281 & 0.993 & 1.000 \\ 
  $\alpha=0.10$ & 0.101 & 0.382 & 0.997 & 1.000 \\ 
  $h_n=6\cdot10^{-3}$ &  &  &  &  \\ 
  $\alpha=0.01$ & 0.011 & 0.041 & 0.634 & 0.997 \\ 
  $\alpha=0.05$ & 0.050 & 0.131 & 0.802 & 1.000 \\ 
  $\alpha=0.10$ & 0.099 & 0.214 & 0.867 & 1.000 \\ 
  \multicolumn{5}{c}{Non-synchronous sampling scenario}  \\ 
  $\alpha=0.01$ & 0.010 & 0.056 & 0.753 & 1.000 \\ 
  $\alpha=0.05$ & 0.051 & 0.152 & 0.879 & 1.000 \\ 
  $\alpha=0.10$ & 0.099 & 0.235 & 0.919 & 1.000 \\ 
   \hline
\end{tabular}\vspace{2mm}

\parbox{9cm}{
\textit{Note}. $\alpha$ denotes the significance level.
}
\end{center}
\end{table}

\newpage
\appendix

\addcontentsline{toc}{section}{Appendix}
\section*{Appendix}

\section{Preliminaries}\label{preliminaries}

\subsection{Basic elements of Gaussian analysis and Malliavin calculus}\label{sec:malliavin}

In this section we briefly overview the basic elements of Gaussian analysis and Malliavin calculus for Gaussian processes. See \cite{Nualart2006,NP2012,Janson1997} for more details about these topics. 

Throughout the paper, $H$ denotes a real separable Hilbert space. The inner product and the norm of $H$ are denoted by $\langle\cdot,\cdot\rangle_H$ and $\|\cdot\|_H$, respectively. We assume that an \textit{isonormal Gaussian process} $W=(W(h))_{h\in H}$ over $H$ defined on a probability space $(\Omega,\mathcal{F},P)$ is given. Namely, $W$ is a centered Gaussian family of random variables such that $E[W(h)W(g)]=\langle h,g\rangle_H$ for any $h,g\in H$. 

We denote by $L^2(W)$ the space $L^2(\Omega,\sigma(W),P)$ for short. For every non-negative integer $q$, we denote by $\mathcal{H}_q$ the closed subspace of $L^2(W)$ spanned by the set $\{\mathrm{He}_q(W(h)):h\in H,\|h\|_H=1\}$, where $\mathrm{He}_q(x)=(-1)^qe^{x^2/2}\frac{d^q}{dx^q}(e^{-x^2/2})$ is the $q$th Hermite polynomial. The space $\mathcal{H}_q$ is called the \textit{$q$th Wiener chaos} of $W$. It is well-known that the spaces $\mathcal{H}_q$ and $\mathcal{H}_r$ are orthogonal whenever $q\neq r$ (cf.~Lemma 1.1.1 of \cite{Nualart2006}). Moreover, the space $L^2(W)$ possesses the following orthogonal decomposition: $L^2(W)=\oplus_{q=0}^\infty\mathcal{H}_q$ (cf.~Theorem 1.1.1 of \cite{Nualart2006}). This decomposition is called the \textit{Wiener-It\^o chaos decomposition} of $L^2(W)$. We denote by $J_q$ the orthogonal projection of $L^2(W)$ onto $\mathcal{H}_q$ for each $q$. Therefore, every $F\in L^2(W)$ has the decomposition $F=\sum_{q=0}^\infty J_qF$, which is called the Wiener-It\^o chaos decomposition of $F$. 

A closed subspace $\mathbb{H}$ of $L^2(\Omega,\mathcal{F},P)$ is called a \textit{Gaussian Hilbert space} if all the elements of $\mathbb{H}$ are centered Gaussian variables. We can easily check that $\{W(h):h\in H\}$ is a Gaussian Hilbert space. Given a Gaussian Hilbert space $\mathbb{H}$ and an integer $q\geq0$, we set
\[
\mathcal{P}_q(\mathbb{H})=\{\varphi(\xi_1,\dots,\xi_m):\text{$\varphi$ is a polynomial of degree $\leq q$};\xi_1,\dots,\xi_m\in\mathbb{H};m\in\mathbb{N}\}
\]
and denote by $\overline{\mathcal{P}}_q(\mathbb{H})$ the closure of $\mathcal{P}_q(\mathbb{H})$ in $L^2(\Omega,\mathcal{F},P)$. If $\mathbb{H}=\{W(h):h\in H\}$, $\overline{\mathcal{P}}_q(\mathbb{H})$ coincides with the space $\oplus_{p=0}^q\mathcal{H}_p$ (cf.~page 6 of \cite{Nualart2006}). The importance of the Gaussian Hilbert space in this paper is illuminated by Proposition \ref{gh-chaos}.   

For a non-negative integer $q$, $H^{\otimes q}$ and $H^{\odot q}$ denote the $q$th tensor power and $q$th symmetric tensor power, respectively (see Appendix E of \cite{Janson1997} for details). The $q$th \textit{multiple Wiener-It\^o integral} $I_q:H^{\odot q}\to\mathcal{H}_q$ is defined as the unique linear isometry between $H^{\odot q}$ (equipped with the scaled norm $\sqrt{q!}\|\cdot\|_{H^{\otimes q}}$) and $\mathcal{H}_q$ (equipped with the norm of $L^2(W)$) such that $I_q(h^{\otimes q})=\mathrm{He}_q(W(h))$ for all $h\in H$ with $\|h\|_H=1$.  

Let $p,q$ be positive integers. For $f\in H^{\odot p}$, $g\in H^{\odot q}$ and $r\in\{0,1,\dots,p\wedge q\}$, $f\otimes_rg$ denotes the $r$th \textit{contraction} of $f$ and $g$. The symmetrization of $f\otimes_rg$ is denoted by $f\widetilde{\otimes}_rg$. See Appendix B of \cite{NP2012} for more details on these concepts. 

A random variable $F$ is said to be \textit{smooth} if it can be written as
\begin{equation}\label{smooth}
F=f(W(h_1),\dots,W(h_m)),
\end{equation}
where $h_1,\dots,h_m\in H$ and $f:\mathbb{R}^m\to\mathbb{R}$ is a $C^\infty$ function such that $f$ and all of its partial derivatives have at most polynomial growth. We denote by $\mathcal{S}$ the set of all smooth random variables. For a smooth random variable $F$ of the form \eqref{smooth} and an integer $k\geq1$, we define the $k$th \textit{Malliavin derivative} of $F$ as the $H^{\odot k}$-valued random variable defined by 
\[
D^kF=\sum_{i_1,\dots,i_k=1}^m\frac{\partial^k f}{\partial x_{i_1}\dots\partial x_{i_k}}(W(h_1),\dots,W(h_m))h_i.
\]
For a real number $p\geq1$, let us denote by $L^p(\Omega;H^{\odot k})$ the set of all $H^{\odot k}$-valued random variables $Y$ such that $E[\|Y\|^p_{H^{\otimes k}}]<\infty$. 
We regard $L^p(\Omega;H^{\odot k})$ as the Banach space equipped with the norm defined by $\|Y\|_{L^p(\Omega;H^{\odot k})}=(E[\|Y\|^p_{H^{\otimes k}}])^{1/p}$. 
Then, it is well-known that the $k$th Malliavin derivative operator $D^k$ on $\mathcal{S}\subset L^p(\Omega,\mathcal{F},P)$ into $L^p(\Omega;H^{\odot k})$ is closable (cf.~Proposition 2.3.4 of \cite{NP2012}). 
Therefore, there is a unique closed operator on $\mathbb{D}_{k,p}\subset L^p(\Omega,\mathcal{F},P)$ into $L^p(\Omega;H^{\odot k})$, which is also denoted by $D^k$, such that its graph is equal to the closure of $\mathcal{S}$ with respect to the norm
\[
\|F\|_{k,p}=\left(E[|F|^p]+\sum_{j=1}^kE\left[\|D^jF\|^p_{H^{\otimes j}}\right]\right)^{1/p}.
\]
We write $D$ instead of $D^1$ for short. 
Malliavin derivatives enjoy the following chain rule (cf.~Proposition 1.2.3 of \cite{Nualart2006}): Let $F_1,\dots,F_m\in\mathbb{D}_{1,p}$ and let $\varphi:\mathbb{R}^m\to\mathbb{R}$ be a $C^1$ function with bounded partial derivatives. Then $\varphi(F_1,\dots,F_m)\in\mathbb{D}_{1,p}$ and we have
\begin{equation}\label{chain}
D\varphi(F_1,\dots,F_m)=\sum_{i=1}^m\frac{\partial\varphi}{\partial x_i}(F_1,\dots,F_m)DF_i.
\end{equation}
We also note that, for any integer $q\geq1$ and $f\in H^{\odot q}$, $I_q(f)\in\mathbb{D}_{1,p}$ and $DI_q(f)=qI_{q-1}(f)$ (cf.~Proposition 2.7.4 of \cite{NP2012}).

We denote by $\delta$ the \textit{divergence operator}, which is the adjoint of the operator $D$ on $\mathbb{D}_{1,2}\subset L^2(W)$ into $L^2(\Omega;H)$. The domain of $\delta$ is denoted by $\mathrm{Dom}(\delta)$. Therefore, for any $F\in\mathbb{D}_{1,2}$ and any $u\in\mathrm{Dom}(\delta)$ we have
\begin{equation}\label{duality}
E[F\delta(u)]=E[\langle DF,u\rangle_H].
\end{equation}
We also set
\[
\mathrm{Dom}(L)=\left\{F\in L^2(W):\sum_{q=1}^\infty q^2E[\|J_qF\|^2]<\infty\right\}.
\]
Then the \textit{Ornstein-Uhlenbeck operator} $L:\mathrm{Dom}(L)\to L^2(W)$ is defined by
\[
LF=-\sum_{q=1}^\infty qJ_qF,\qquad F\in\mathrm{Dom}(L),
\]
where the convergence of the series is considered in $L^2(W)$. It is well-known that $F\in L^2(W)$ belongs to $\mathrm{Dom}(L)$ if and only if $F\in\mathbb{D}_{1,2}$ and $DF\in\mathrm{Dom}(\delta)$, and this case we have $LF=-\delta DF$ (cf.~Proposition 1.4.3 of \cite{Nualart2006}). Finally, the \textit{pseudo inverse} of $L$, denoted by $L^{-1}$, is the operator on $L^2(W)$ into $\mathrm{Dom}(L)$ defined by
\[
L^{-1}F=-\sum_{q=1}^\infty\frac{1}{q}J_qF,\qquad F\in L^2(W).
\]
It holds that 
$LL^{-1}F=F-E[F]$ 
for all $F\in L^2(W)$ (cf.~Proposition 2.8.11 of \cite{NP2012}).  

\subsection{Technical tools from the Chernozhukov-Chetverikov-Kato theory}

This subsection collects the technical results of the Chernozhukov-Chetverikov-Kato theory, which are used in this paper. The first result is a corollary of Lemmas 3--4 of \cite{CCK2015} and Lemma 4.3 of \cite{CCK2014}:
\begin{lemma}\label{cck-derivative}
Let $g:\mathbb{R}\to\mathbb{R}$ be a $C^2$ function. Then we have
\[
\sum_{i,j=1}^d\left|\frac{\partial^2(g\circ\Phi_\beta)}{\partial x_i\partial x_j}(x)\right|\leq\|g''\|_\infty+2\beta\|g'\|_\infty
\]
for all $x\in\mathbb{R}^d$. Moreover, if $g$ is a $C^3$ function, we also have
\[
\sum_{i,j,k=1}^d\left|\frac{\partial^3(g\circ\Phi_\beta)}{\partial x_i\partial x_j\partial x_k}(x)\right|\leq\|g'''\|_\infty+6\beta\|g''\|_\infty+6\beta^2\|g'\|_\infty
\]
for all $x\in\mathbb{R}^d$. 
\end{lemma}

\begin{proof}
The first inequality is a direct consequence of Lemmas 3--4 from \cite{CCK2015}. The second inequality is Eq.(19) in Lemma 4.3 of \cite{CCK2014}. 
\end{proof}

The next result is taken from Lemma 5.1 of \cite{CCK2016}:
\begin{lemma}[\cite{CCK2016}, Lemma 5.1]\label{cck-approx}
For any $\varepsilon>0$ and any Borel set $A$ of $\mathbb{R}$, there is a $C^\infty$ function $g:\mathbb{R}\to\mathbb{R}$ satisfying the following conditions:
\begin{enumerate}[label=(\roman*)]

\item There is a universal constant $K>0$ such that $\|g'\|_\infty\leq\varepsilon^{-1}$, $\|g''\|_\infty\leq K\varepsilon^{-2}$ and $\|g'''\|_\infty\leq K\varepsilon^{-3}$.

\item $1_A(x)\leq g(x)\leq 1_{A^{3\varepsilon}}(x)$ for all $x\in\mathbb{R}$.

\end{enumerate}
\end{lemma}

The third one is a corollary of Lemma 4.1 from \cite{CCK2014} and Lemma 2.1 from \cite{CCK2016}:
\begin{lemma}\label{cck-kolmogorov}
Let $V,W$ be random variables. Suppose that there are constants $r_1,r_2>0$ such that
\[
P(V\in A)\leq P(W\in A^{r_1})+r_2
\]
for any Borel set $A$ of $\mathbb{R}$. Then we have
\[
\sup_{x\in\mathbb{R}}|P(V\leq x)-P(W\leq x)|
\leq\sup_{x\in\mathbb{R}}P(|W-x|\leq r_1)+r_2.
\]
\end{lemma}

\begin{proof}
By extending the probability space $(\Omega,\mathcal{F},P)$ if necessary, we may assume that there is a uniform random variable on $(0,1)$ independent of $V$ without loss of generality. Then, by Lemma 4.1 from \cite{CCK2014} there is a random variable $W'$ whose distribution is the same as that of $W$ such that $P(|V-W'|>r_1)\leq r_2$. Therefore, the desired result follows from Lemma 2.1 of \cite{CCK2015}.
\end{proof}

The fourth result is a so-called anti-concentration inequality for maxima of Gaussian variables, which is taken from Theorem 3 of \cite{CCK2015}:
\begin{lemma}[\cite{CCK2015}, Theorem 3]\label{cck-anti}
Let $X=(X_1,\dots,X_d)^\top$ be a $d$-dimensional centered Gaussian random vector with $\sigma_j^2:=E[X_j^2]>0$ for $j=1,\dots,d$. Suppose that there are two constants $\overline{\sigma},\underline{\sigma}>0$ such that $\underline{\sigma}\leq\sigma_j\leq\overline{\sigma}$ for all $j=1,\dots,d$. Then, there is a constant $C>0$ which only depends on $\overline{\sigma}$ and $\underline{\sigma}$ such that
\[
\sup_{x\in\mathbb{R}}P(|X_\vee-x|\leq\varepsilon)\leq C\varepsilon\{a_d+\sqrt{1\vee\log(\underline{\sigma}/\varepsilon)}\}
\]
for any $\varepsilon>0$, where $a_d=E[\max_{1\leq j\leq d}(X_j/\sigma_j)]$. 
\end{lemma}

The fifth result is another anti-concentration inequality for maxima of Gaussian variables, which is taken from Lemma 4.3 of \cite{CCK2016}:
\begin{lemma}[\cite{CCK2016}, Lemma 4.3]\label{cck-nazarov}
Let $X=(X_1,\dots,X_d)^\top$ be a (possibly uncentered) $d$-dimensional Gaussian random vector such that $\variance[X_j]>0$ for all $j=1,\dots,d$. Then for every $\varepsilon>0$,
\[
\sup_{x\in\mathbb{R}}P(|X_\vee-x|\leq\varepsilon)\leq \frac{2}{\underline{\sigma}}\varepsilon\left(\sqrt{2\log d}+2\right),
\]
where $\underline{\sigma}=\min_{1\leq j\leq d}\sqrt{\variance[X_j]}$. 
\end{lemma}

The last result is an anti-concentration inequality for supremum of a Gaussian process taken from Corollary 2.1 of \cite{CCK2014b}:
\begin{lemma}[\cite{CCK2014b}, Corollary 2.1]\label{cck-anti2}
Let $X=(X_t)_{t\in\mathbb{T}}$ be a separable Gaussian process indexed by a semi-metric space $\mathbb{T}$ such that $E[X_t]=0$ and $E[X_t^2]=1$ for all $t\in\mathbb{T}$. Assume that $\sup_{t\in\mathbb{T}}X_t<\infty$ a.s. Then $E[\sup_{t\in\mathbb{T}}|X_t|]<\infty$ and
\[
\sup_{x\in\mathbb{R}}P\left(\left|\sup_{t\in\mathbb{T}}|X_t|-x\right|\leq\varepsilon\right)\leq4\varepsilon\left(E\left[\sup_{t\in\mathbb{T}}|X_t|\right]+1\right)
\]
for all $\varepsilon\geq0$.
\end{lemma}

\subsection{Sub-Gaussian chaos property}\label{sec:sub-chaos}

This subsection presents the notion of \textit{sub-Gaussian chaos property} of random variables and stochastic processes, which is introduced in \cite{VV2007,VV2007jfa} and serves as deriving maximal inequalities used in this paper. 
\begin{definition}[\cite{VV2007}, Definition 4.1; \cite{VV2007jfa}, Definition 2.1]
Let $q$ be a positive integer. A (possibly uncentered) random variable $Y$ is said to have the \textit{sub-$q$th-Gaussian chaos property} (or is a \textit{sub-$q$th chaos random variable}, or is a \textit{sub-Gaussian chaos random variable of order $q$}, etc.) relative to the scale $M\geq0$ if
\[
E\left[\exp\left(\left(\frac{|Y|}{M}\right)^{2/q}\right)\right]\leq2.
\]
Here, when $M=0$, for $x\geq0$, $x/M$ is understood to be 0 if $x=0$ and $\infty$ otherwise. Hence, $Y$ is a sub-$q$th chaos random variable relative to the scale 0 if and only if $Y=0$ a.s.
\end{definition}
Note that, unlike \cite{VV2007,VV2007jfa}, we also allow sub-Gaussian chaos random variables to be uncentered. The following result is a useful criterion for this property.
\begin{lemma}\label{criteria-subchaos}
For any positive integer $q$ and constant $C>0$, there is a positive number $M$ (which depends only on $q$ and $C$) such that any random variable $Y$, which satisfies $\|Y\|_p\leq Cp^{q/2}\Lambda$ for some constant $\Lambda\geq0$ and any positive integer $p$, is a sub-$q$th chaos random variable relative to the scale $M\Lambda$. 
\end{lemma}

\begin{proof}
For any $M>0$ we have
\begin{align*}
E\left[\exp\left(\left(\frac{|Y|}{M\Lambda}\right)^{2/q}\right)\right]
&=\sum_{k=0}^\infty\frac{1}{k!}E\left[\left(\frac{|Y|}{M\Lambda}\right)^{2k/q}\right]
\leq1+\sum_{k=1}^\infty C^{2k/q}\frac{(2k/q)^k}{k!M^{2k/q}}.
\end{align*}
Set $a_k(M)=C^{2k/q}(2k/q)^k/k!M^{2k/q}$. Then we have
\begin{align*}
\frac{a_{k+1}(M)}{a_k(M)}=\frac{C^{2/q}}{M^{2/q}}\left(\frac{2k+2}{2k}\right)^k\frac{2k+2}{q(k+1)}
\to\frac{C^{2/q}}{M^{2/q}}\cdot e\cdot\frac{2}{q}
\end{align*}
as $k\to\infty$. Therefore, for sufficiently large $M$ we have $\lim_{k\to\infty}a_{k+1}(M)/a_k(M)<1$, hence we obtain $\sum_{k=1}^\infty a_k(M)<\infty$ by the d'Alembert ratio test. Since $a_k(M)$ is a decreasing function of $M>0$ for all $k\in\mathbb{N}$, the dominated convergence theorem yields $\sum_{k=1}^\infty a_k(M)\to0$ as $M\to\infty$. Therefore, there is a constant $M>0$ which depends only on $C$ and $q$ such that $\sum_{k=1}^\infty a_k(M)\leq1$. For this $M$ the claim of the lemma holds true.
\end{proof}

The condition on $Y$ in Lemma \ref{criteria-subchaos} is typically satisfied with $\Lambda=\|Y\|_2$. In particular, combining Lemma \ref{criteria-subchaos} with Theorem 5.11 and Remark 5.11 from \cite{Janson1997}, we obtain the following result.
\begin{proposition}\label{gh-chaos}
For any positive integer $q$, there is a constant $M>0$ which depends only on $q$ such that, for any Gaussian Hilbert space $\mathbb{H}$, $Y$ is a sub-$q$th chaos random variable relative to the scale $M\|Y\|_2$ for all $Y\in\bar{\mathcal{P}}_q(\mathbb{H})$.
\end{proposition}

The next result presents maximal moment inequalities obtained from the sub-Gaussian chaos property.   
\begin{proposition}\label{max-sub-chaos}
For each $k=1,\dots,d$, let $Y_k$ be a sub-$q$th chaos random variable relative to the scale $M_k\geq0$. Then we have
\begin{align*}
E\left[\max_{1\leq k\leq d}|Y_k|^p\right]
\leq\left(\max_{1\leq k\leq d}M_k^p\right)\log^{pq/2}\left(2d-1+e^{pq/2-1}\right)
\end{align*}
for any $p>0$ such that $pq\geq2$.
\end{proposition}

\begin{proof}
By Lemma 14.7 from \cite{BvdG2011} we have
\begin{align*}
E\left[\max_{1\leq k\leq d}|Y_k|^p\right]
\leq \left(\max_{1\leq k\leq d}M_k^p\right)\log^{pq/2}\left(E\left[\exp\left(\max_{1\leq k\leq d}\left(\frac{|Y_k|}{M_k}\right)^{2/q}\right)\right]-1+e^{pq/2-1}\right).
\end{align*}
Since we have
\begin{align*}
E\left[\exp\left(\max_{1\leq k\leq d}\left(\frac{|Y_k|}{M_k}\right)^{2/q}\right)\right]
&= E\left[\max_{1\leq k\leq d}\exp\left(\left(\frac{|Y_k|}{M_k}\right)^{2/q}\right)\right]
\leq\sum_{k=1}^dE\left[\exp\left(\left(\frac{|Y_k|}{M_k}\right)^{2/q}\right)\right]
\leq2d,
\end{align*}
we obtain the desired result.
\end{proof}

To conclude this subsection, we present an estimate for the modulus of continuity for sub-Gaussian chaos processes, which is established in \cite{VV2007jfa}. 
In the following, $(\mathbb{T},\mathfrak{d})$ denotes a semi-metric space.
\begin{definition}[\cite{VV2007}, Definition 4.3; \cite{VV2007jfa}, Definition 2.3]
Let $q$ be a positive integer. A centered real-valued process $X=(X_t)_{t\in\mathbb{T}}$ is said to be a \textit{sub-$q$th-Gaussian chaos process with respect to $\mathfrak{d}$} if the random variable $X(t)-X(s)$ has the sub-$q$th-Gaussian chaos property relative to the scale $\mathfrak{d}(s,t)$ for any $s,t\in\mathbb{T}$.
\end{definition}

\begin{definition}
Let $r$ be a positive number. An \textit{$r$-covering number of $\mathbb{T}$ with respect to $\mathfrak{d}$}, which is denoted by $N(\mathbb{T},\mathfrak{d},r)$, is the smallest positive integer $N$ such that there are points $t_1,\dots,t_N\in\mathbb{T}$ satisfying for any $t\in\mathbb{T}$ there is an index $i\in\{1,\dots,N\}$ with $\mathfrak{d}(t,t_i)<r$ (we set $N(\mathbb{T},\mathfrak{d},r)=\infty$ if there is no such an $N$).
\end{definition}

The following result follows from Remark 2.2, Corollary 3.2 and the discussion at the beginning of Section 5 from \cite{VV2007jfa}:
\begin{proposition}\label{vv-tail}
Let $q$ be a positive integer. There is a number $C_q>0$ which depends only on $q$ such that the variable $\sup_{s,t\in\mathbb{T}:\mathfrak{d}(s,t)\leq\eta}|X_s-X_t|$ is a sub-$q$th chaos random variable relative to the scale
\[
C_q\int_0^\eta\log^{q/2}N(\mathbb{T},\mathfrak{d},r)dr
\]
for any $\eta>0$ and any separable centered sub-$q$th-Gaussian chaos process $X=(X_t)_{t\in\mathbb{T}}$ with respect to $\mathfrak{d}$.
\end{proposition}

\begin{rmk}
For the case $q=1$ or $q=2$, related estimates to Proposition \ref{vv-tail} can be found in Section 2.2 of \cite{VW1996} (they are indeed sufficient for this paper). We also remark that \citet{Dirksen2015} gives a shaper estimate than Proposition \ref{vv-tail} for general values of $q$ (see Theorem 3.2 and Remark 3.3 of \cite{Dirksen2015} for details).  
\end{rmk}

\section{Proofs}\label{sec:proofs}

\subsection{Proof of Proposition \ref{stein}}

The proof of Proposition \ref{stein} relies on the following integration by parts formula established by Nourdin and Peccati:
\begin{lemma}[Nordin-Peccati's formula]\label{lemma:NP}
Let $G\in L^2(W)$ and $\psi:\mathbb{R}^d\to\mathbb{R}$ be a $C^1$ function with bounded partial derivatives. Then we have
\[
\covariance[\psi(F),G]=\sum_{i=1}^dE\left[\frac{\partial\psi}{\partial x_i}(F)\langle DF_i,-DL^{-1}G\rangle_H\right].
\] 
\end{lemma}
This lemma is a straightforward extension of Theorem 2.9.1 from \cite{NP2012}, so we omit its proof.

\begin{proof}[\upshape\bfseries Proof of Proposition \ref{stein}]
The latter claim immediately follows from the former and \eqref{max-smooth}. To prove the former claim, we may assume that $F$ and $Z$ are independent without loss of generality. Let us set $\varphi:=g\circ \Phi_\beta$ and define the function $\Psi:[0,1]\to\mathbb{R}$ by $\Psi(t)=E[\varphi(\sqrt{t}F+\sqrt{1-t}Z)]$, $t\in[0,1]$. We can easily check that $\Psi$ is continuous on $[0,1]$ and differentiable on $(0,1)$, and we have
\[
\Psi'(t)=\frac{1}{2}\sum_{j=1}^dE\left[\frac{\partial\varphi}{\partial x_j}(\sqrt{t}F+\sqrt{1-t}Z)\left(\frac{F_j}{\sqrt{t}}-\frac{Z_j}{\sqrt{1-t}}\right)\right]
\]
for every $t\in(0,1)$. Now, by Lemma \ref{cck-derivative} we have
\[
\sum_{i,j=1}^d\left|\frac{\partial^2\varphi}{\partial x_ix_j}(x)\right|\leq\|g''\|_\infty+2\beta\|g'\|_\infty
\]
for any $x\in\mathbb{R}^d$. In particular, $\partial^2\varphi/\partial x_i\partial x_j$ is bounded for all $i,j=1,\dots,d$. Therefore, noting that the independence between $Z$ and $F$, Stein's identity (e.g.~Lemma 2 of \cite{CCK2015}) implies that
\begin{align*}
\sum_{j=1}^dE\left[\frac{\partial\varphi}{\partial x_j}(\sqrt{t}F+\sqrt{1-t}Z)\frac{Z_j}{\sqrt{1-t}}\right]
=\sum_{i,j=1}^dE\left[\frac{\partial^2\varphi}{\partial x_i\partial x_j}(\sqrt{t}F+\sqrt{1-t}Z)\mathfrak{C}(i,j)\right].
\end{align*}
Moreover, Nourdin-Peccati's formula (Lemma \ref{lemma:NP}) yields
\begin{align*}
\sum_{j=1}^dE\left[\frac{\partial\varphi}{\partial x_j}(\sqrt{t}F+\sqrt{1-t}Z)\frac{F_j}{\sqrt{t}}\right]
=\sum_{i,j=1}^dE\left[\frac{\partial^2\varphi}{\partial x_i\partial x_j}(\sqrt{t}F+\sqrt{1-t}Z)\langle DF_i,-DL^{-1}F_j\rangle_H\right].
\end{align*}
Hence we conclude that
\begin{equation*}
\Psi'(t)=\frac{1}{2}\sum_{i,j=1}^dE\left[\frac{\partial^2\varphi}{\partial x_i\partial x_j}(\sqrt{t}F+\sqrt{1-t}Z)(\langle DF_i,-DL^{-1}F_j\rangle_H-\mathfrak{C}(i,j))\right]
\end{equation*}
for every $t\in(0,1)$. Consequently, we obtain
\[
\left|E\left[\varphi\left(F\right)\right]-E\left[\varphi\left(Z\right)\right]\right|
\leq\int_0^1|\Psi'(t)|dt
\leq(\|g''\|_\infty/2+\beta\|g'\|_\infty)\Delta,
\]
which completes the proof.
\end{proof}

\subsection{Proof of Lemma \ref{coupling}}

Let us set $\beta=\varepsilon^{-1}\log d$ (hence $\beta^{-1}\log d=\varepsilon$). From \eqref{max-smooth} we have
\[
P(F_\vee\in A)\leq P(\Phi_\beta(F)\in A^{\varepsilon})=E[1_{A^{\varepsilon}}(\Phi_\beta(F))].
\]
Next, by Lemma \ref{cck-approx} there is a $C^\infty$ function $g:\mathbb{R}\to\mathbb{R}$ and a universal constant $K>0$ such that $\|g'\|_\infty\leq\varepsilon^{-1}$, $\|g''\|_\infty\leq K\varepsilon^{-2}$, $\|g'''\|_\infty\leq K\varepsilon^{-3}$ and $1_{A^{\varepsilon}}(x)\leq g(x)\leq1_{A^{4\varepsilon}}(x)$ for any $x\in\mathbb{R}$. Then we obtain $E[1_{A^{\varepsilon}}(\Phi_\beta(F))]\leq E[g(\Phi_\beta(F))].$ 
Now, by Proposition \ref{stein} we have
\[
|E[g(\Phi_\beta(F))]-E[g(\Phi_\beta(Z))]|\leq C_1(\varepsilon^{-2}+\varepsilon^{-2}\log d)\Delta\leq C_2\varepsilon^{-2}(\log d)\Delta,
\]
where $C_1,C_2>0$ denote universal constants. 
Moreover, we also have 
\[
E[g(\Phi_\beta(Z))]
\leq E[1_{A^{4\varepsilon}}(\Phi_\beta(Z))]
\leq E[1_{A^{5\varepsilon}}(Z_\vee)]=P(Z_\vee\in A^{5\varepsilon}).
\]
This completes the proof.\hfill$\Box$

\subsection{Proof of Theorem \ref{kolmogorov}}

If $\Delta\geq1$, \eqref{eq:kolmogorov}--\eqref{eq:kolmogorov2} hold true by taking $C=C'=2$ (note that $d\geq2$). Therefore, it suffices to consider the case of $\Delta<1$. 

By Lemmas \ref{coupling} and \ref{cck-kolmogorov} there is a universal constant $C_0>0$ such that
\[
\sup_{x\in\mathbb{R}}\left|P(F_\vee\leq x)-P(Z_\vee\leq x)\right|
\leq\sup_{x\in\mathbb{R}}P(|Z_\vee-x|\leq 5\varepsilon)+C_0\varepsilon^{-2}(\log d)\Delta.
\]
Now, under the assumptions of claim (a), Lemma \ref{cck-anti} yields
\[
\sup_{x\in\mathbb{R}}P(|Z_\vee-x|\leq 5\varepsilon)\leq C_1\varepsilon\{a_d+\sqrt{1\vee\log(\underline{\sigma}/\varepsilon)}\}
\leq C_2\varepsilon\sqrt{1\vee a_d^2\vee\log(1/\varepsilon)},
\]
where $C_1,C_2>0$ depend only on $\underline{\sigma}$ and $\overline{\sigma}$, while we have
\[
\sup_{x\in\mathbb{R}}P(|Z_\vee-x|\leq 5\varepsilon)\leq C_1\varepsilon\sqrt{\log d},
\] 
where $C_3>0$ depends only on $b$, under the assumptions of claim (b) due to Lemma \ref{cck-nazarov}. Consequently, if $\Delta=0$, letting $\varepsilon$ tend to 0, we obtain the desired results. Otherwise, take
\[
\varepsilon=\Delta^{1/3}(1\vee a_d\vee\log^{1/2}(1/\Delta))^{-1/3}(2\log d)^{1/3}
\]
for claim (a). Then the same argument as the one in the proof of Theorem 2 from \cite{CCK2015} yields \eqref{eq:kolmogorov}. For claim (b), taking $\varepsilon=\Delta^{1/3}(\log d)^{1/6}$, we obtain \eqref{eq:kolmogorov2}.\hfill$\Box$

\subsection{Proof of Lemma \ref{fourth-moment}}

By the triangular inequality it suffices to prove
\[
E\left[\max_{1\leq i,j\leq d}|\Delta_{i,j}|\right]\leq C_q\log^{q-1}\left(2d^2-1+e^{q-2}\right)\max_{1\leq k\leq d}\sqrt{E[F_k^4]-3E[F_k^2]^2},
\]
where $\Delta_{i,j}=E[F_iF_j]-\langle DF_i,-DL^{-1}F_j\rangle_H$. 
From the proof of Lemma 6.2.1 from \cite{NP2012} we deduce that
\[
\langle DF_i,DF_j\rangle_H
=q^2\sum_{r=1}^q(r-1)!\binom{q-1}{r-1}^2I_{2q-2r}(f_i\widetilde{\otimes}_rf_j).
\]
Since $-L^{-1}F_j=q^{-1}F_j$, we obtain
\begin{equation}\label{delta-formula}
\Delta_{i,j}=-q\sum_{r=1}^{q-1}(r-1)!\binom{q-1}{r-1}^2I_{2q-2r}(f_i\widetilde{\otimes}_rf_j).
\end{equation}
In particular, by Proposition \ref{gh-chaos} there is a constant $M_q>0$ which only depends on $q$ such that $\Delta_{i,j}$ is a sub-$(2q-2)$th chaos random variable relative to the scale $M_q\|\Delta_{i,j}\|_2$. Therefore, by Proposition \ref{max-sub-chaos} we have
\[
E\left[\max_{1\leq i,j\leq d}|\Delta_{i,j}|\right]\leq M_q\log^{q-1}\left(2d^2-1+e^{q-2}\right)\max_{1\leq k\leq d}\|\Delta_{i,j}\|_2.
\]
Therefore, the proof is completed once we show that
\begin{equation*}
\max_{1\leq i,j\leq d}E[\Delta_{i,j}^2]\leq\left\{\sum_{r=1}^{q-1}\binom{2r}{r}\right\}\max_{1\leq k\leq d}\left(E[F_k^4]-3E[F_k^2]^2\right),
\end{equation*}
which follows from the equation $\Delta_{i,j}=E[F_iF_j]-q^{-1}\langle DF_i,DF_j\rangle_H$ and Eq.(6.2.6) of \cite{NP2012}.\hfill$\Box$

\subsection{Proof of Lemma \ref{poincare}}

\begin{lemma}\label{poincare-moment}
If $F_1,\dots,F_d\in\mathbb{D}_{2,4p}$ for a positive integer $p$, we have
\begin{multline*}
\max_{1\leq i,j\leq d}\left\|E[F_iF_j]-\langle DF_i,-DL^{-1}F_j\rangle_H\right\|_{2p}\\
\leq\sqrt{2p-1}\cdot\frac{3}{2}\left(\max_{1\leq i\leq d}\left\|\left\|D^2F_i\right\|_\text{op}\right\|_{4p}\right)\left(\max_{1\leq j\leq d}\left\|\left\|DF_j\right\|_H\right\|_{4p}\right).
\end{multline*}
\end{lemma}

\begin{proof}
For any $i,j=1,\dots,d$, we have $E[\langle DF_i,-DL^{-1}F_j\rangle_H]=E[F_iF_j]$ by Nourdin-Peccati's formula. Therefore, it holds that
\begin{align*}
&\left\|E[F_iF_j]-\langle DF_i,-DL^{-1}F_j\rangle_H\right\|_{2p}\\
&\leq\sqrt{2p-1}\left\|D\langle DF_i,-DL^{-1}F_j\rangle_H\right\|_{2p}~(\because\text{Lemma 5.3.7 of \cite{NP2012}})\\
&=\sqrt{2p-1}\left\|\langle D^2F_i,-DL^{-1}F_j\rangle_H+\langle DF_i,-D^2L^{-1}F_j\rangle_H\right\|_{2p}~(\because\text{Lemma 5.3.8 of \cite{NP2012}})\\
&\leq\sqrt{2p-1}\left(\left\|\langle D^2F_i,-DL^{-1}F_j\rangle_H\right\|_{2p}+\left\|\langle DF_i,-D^2L^{-1}F_j\rangle_H\right\|_{2p}\right)\\
&\leq\sqrt{2p-1}\left(\left\|\left\|D^2F_i\right\|_\text{op}\left\|DL^{-1}F_j\right\|_H\right\|_{2p}+\left\|\left\|DF_i\right\|_H\left\|D^2L^{-1}F_j\right\|_\text{op}\right\|_{2p}\right)\\
&\leq\sqrt{2p-1}\left(\left\|\left\|D^2F_i\right\|_\text{op}\right\|_{4p}\left\|\left\|DL^{-1}F_j\right\|_H\right\|_{4p}+\left\|\left\|DF_i\right\|_H\right\|_{4p}\left\|\left\|D^2L^{-1}F_j\right\|_\text{op}\right\|_{4p}\right)~(\because\text{the Schwarz inequality})\\
&\leq\sqrt{2p-1}\left(\left\|\left\|D^2F_i\right\|_\text{op}\right\|_{4p}\left\|\left\|DF_j\right\|_H\right\|_{4p}+\frac{1}{2}\left\|\left\|DF_i\right\|_H\right\|_{4p}\left\|\left\|D^2F_j\right\|_\text{op}\right\|_{4p}\right)~(\because\text{Lemma 5.3.7 of \cite{NP2012}})\\
&\leq\sqrt{2p-1}\cdot\frac{3}{2}\left(\max_{1\leq i\leq d}\left\|\left\|D^2F_i\right\|_\text{op}\right\|_{4p}\right)\left(\max_{1\leq j\leq d}\left\|\left\|DF_j\right\|_H\right\|_{4p}\right).
\end{align*}
This completes the proof. 
\end{proof}

\begin{proof}[\upshape\bfseries Proof of Lemma \ref{poincare}]
Since we have 
\begin{align*}
&E\left[\max_{1\leq i,j\leq d}|E[F_iF_j]-\langle DF_i,-DL^{-1}F_j\rangle_H|\right]\\
&\leq \left\|\max_{1\leq i,j\leq d}|E[F_iF_j]-\langle DF_i,-DL^{-1}F_j\rangle_H|\right\|_{2p}\\
&\leq d^{1/p}\max_{1\leq i,j\leq d}\left\|E[F_iF_j]-\langle DF_i,-DL^{-1}F_j\rangle_H|\right\|_{2p},
\end{align*}
the first inequality follows from Lemma \ref{poincare-moment}. Next, if both the variables $\left\|D^2F_i\right\|_\text{op}$ and $\left\|DF_i\right\|_H$ are sub-Gaussian relative to the scale $a$ for all $i=1,\dots,d$, we have
\begin{align*}
\left\|E[F_iF_j]-\langle DF_i,-DL^{-1}F_j\rangle_H\right\|_{p}
&\leq \left\|E[F_iF_j]-\langle DF_i,-DL^{-1}F_j\rangle_H\right\|_{2p}\\
&\leq C_0p^{3/2}a^2
\end{align*}
with some universal constant $C_0>0$ for all $i,j=1,\dots,d$ and any integer $p\geq1$ by Lemma \ref{poincare-moment} and Lemma 1 of \cite{BK1981}. Therefore, by Lemma \ref{criteria-subchaos} there is another universal constant $C>0$ such that $E[F_iF_j]-\langle DF_i,-DL^{-1}F_j\rangle_H$ is a sub-3rd chaos random variable relative to the scale $Ca^2$. Now the second inequality of the lemma follows from Proposition \ref{max-sub-chaos}.  
\end{proof}

\subsection{Proof of Theorem \ref{gqf}}

Without loss of generality we may assume that $\boldsymbol{\xi}_n$ can be written as $\boldsymbol{\xi}_n=\Sigma_n^{1/2}\boldsymbol{\eta}_n$, where $\boldsymbol{\eta}_n$ is an $N_n$-dimensional standard normal variable. Set $H_n=\mathbb{R}^{N_n}$ and set $W_n(h)=h^\top\boldsymbol{\eta}_n$ for $h\in H_n$. Then, $W_n=(W_n(h))_{h\in H_n}$ is an isonormal Gaussian process over $H_n$ and we have $\boldsymbol{\eta}_n=(W_n(e_1),\dots,W_n(e_{N_n}))^\top$, where $(e_1,\dots,e_{N_n})$ is the canonical basis of $H_n$.   

Now let us denote by $\gamma_{n,k}(i,j)$ the $(i,j)$-th component of the $N_n\times N_n$ matrix $\Sigma_n^{1/2}A_{n.k}\Sigma_n^{1/2}$. Then by the product formula for multiple Wiener-It\^o integrals (e.g.~Theorem 2.7.10 of \cite{NP2012}) we can rewrite $F_{n,k}$ as
\[
F_{n,k}=\sum_{i,j=1}^{N_n}\gamma_{n,k}(i,j)I^{W_n}_2(e_i\otimes e_j)
=I^{W_n}_2(f_{n,k}),
\]
where $I^{W_n}_2(\cdot)$ denotes the double Wiener-It\^o integral with respect to $W_n$ and
\[
f_{n,k}=
\sum_{i,j=1}^N\gamma_{n,k}(i,j)e_i\otimes e_j.
\]
Therefore, by applying Theorem \ref{kolmogorov}, Corollary \ref{abs-kolmogorov} and Lemma \ref{fourth-moment}, we obtain the desired result.\hfill$\Box$

\subsection{Proof of Theorem \ref{kolmogorov-general}}

We first derive some non-asymptotic results.
For each $k=1,\dots,d$, let $\Gamma_k=(\gamma_k(i,j))_{1\leq i,j\leq N}$ be an $N\times N$ symmetric matrix such that $\gamma_k(i,i)=0$ for every $i=1,\dots,N$. 
Given a sequence $\xi=(\xi_i)_{i=1}^\infty$ of random variables, we set
\[
Q_k(\xi)=\sum_{i,j=1}^N\gamma_k(i,j)\xi_i\xi_j,\qquad k=1,\dots,d
\]
and $Q(\xi)=(Q_1(\xi),\dots,Q_d(\xi))^\top$. Also, we set
\begin{align*}
R_{1}&=\sum_{i=1}^{N}E\left[\max_{1\leq k\leq d}\left|\sum_{j=1}^{N}\gamma_{k}(i,j)W^{(i)}_j\right|^3\right](E[|Y_i|^{3}]+E[|G_i|^3]),\\
R_{2}&=\max_{1\leq k,l\leq d}\left|\mathfrak{C}(k,l)-E[Q_{k}(G)Q_{l}(G)]\right|,\\
R_{3}&=\max_{1\leq k\leq d}\sqrt{E[Q_{k}(G)^4]-3E[Q_{k}(G)^2]^2}.
\end{align*}
\begin{lemma}\label{stein-general}
Let $g:\mathbb{R}\to\mathbb{R}$ be a $C^3$ function with bounded derivatives up to the third order. For any $\beta>0$ we have
\begin{align*}
\left|E\left[g\left(\Phi_\beta(Q(Y))\right)\right]-E\left[g\left(\Phi_\beta(Q(G))\right)\right]\right|
&\leq\frac{4}{3}\left(\|g'''\|_\infty+6\|g''\|_\infty\beta+6\|g'\|_\infty\beta^2\right)R_1.
\end{align*}
\end{lemma}

\begin{proof}
The proof is based on the generalized Lindeberg method developed in \cite{MOO2010,NPR2010aop}. 
We start with introducing some notation following the proof of Proposition 11.4.2 from \cite{NP2012}. For $p=1,\dots,d$ and $i=1,\dots,N$, we define the variables $U_{p,i}$ and $V_{p,i}$ by
\[
U_{p,i}=\sum_{\begin{subarray}{c}
j,k=1\\
j\neq i,k\neq i
\end{subarray}}^N\gamma_p(j,k)W^{(i)}_jW^{(i)}_k,\qquad
V_{p,i}=2\sum_{j=1}^N\gamma_p(i,j)W^{(i)}_{j}
\]
and set $U_i=(U_{p,i})_{p=1}^d,V_i=(V_{p,i})_{p=1}^d$. By construction both $Y_i$ and $G_i$ are independent of $(U_i,V_i)$ (note that $\gamma_p(i,i)=0$ for every $p$), and we have
\begin{align*}
U_{p,i}+Y_iV_{p,i}=Q_p(W^{(i)}),\qquad
U_{p,i}+G_iV_{p,i}=Q_p(W^{(i-1)}).
\end{align*}
Set $h=g\circ\Phi_\beta$. Noting that $E[Y_i]=E[G_i]=0$ and $E[Y_i^2]=E[G_i^2]=1$ as well as the independence between $Y_i,G_i$ and $(U_i,V_i)$, the Taylor theorem yields
\begin{align*}
E[h(U_i+\xi_iV_i)]
&=\frac{1}{2}\sum_{p,q=1}^dE\left[\frac{\partial^2h}{\partial x_p\partial x_q}(U_i)V_{p,i}V_{q_i}\right]\\
&\quad+\frac{1}{2}\sum_{p,q,r=1}^d\int_0^1(1-t)^2E\left[\frac{\partial^3h}{\partial x_p\partial x_q\partial x_r} (U_i+t\xi_iV_i)V_{p,i}V_{q,i}V_{r,i}\xi_i^{3}\right]dt
\end{align*}
when $\xi_i=Y_i$ or $\xi_i=G_i$. This implies that
\begin{align*}
&\left|E[h(Q(W^{(i)}))]-E[h(Q(W^{(i-1)}))]\right|\\
&\leq\frac{1}{2}\int_0^1(1-t)^2E\left[\sum_{p,q,r=1}^d\left|\frac{\partial^3h}{\partial x_p\partial x_q\partial x_r} (U_i+tY_iV_i)\right|\max_{1\leq p,q,r\leq d}|V_{p,i}V_{q,i}V_{r,i}||Y_i|^{3}\right]dt\\
&\quad+\frac{1}{2}\int_0^1(1-t)^2E\left[\sum_{p,q,r=1}^d\left|\frac{\partial^3h}{\partial x_p\partial x_q\partial x_r} (U_i+tG_iV_i)\right|\max_{1\leq p,q,r\leq d}|V_{p,i}V_{q,i}V_{r,i}||G_i|^3\right]dt.
\end{align*}
Therefore, by Lemma \ref{cck-derivative} we obtain
\begin{align*}
&\left|E[h(Q(W^{(i)}))]-E[h(Q(W^{(i-1)}))]\right|\\
&\leq\frac{1}{6}\left(\|g'''\|_\infty+6\|g''\|_\infty\beta+6\|g'\|_\infty\beta^2\right)E\left[\max_{1\leq p,q,r\leq d}|V_{p,i}V_{q,i}V_{r,i}|(|Y_i|^{3}+|G_i|^3)\right]\\
&\leq\frac{1}{6}\left(\|g'''\|_\infty+6\|g''\|_\infty\beta+6\|g'\|_\infty\beta^2\right)E\left[\max_{1\leq p\leq d}|V_{p,i}|^3\right](E[|Y_i|^{3}]+E[|G_i|^3]),
\end{align*}
where we use the independence between $Y_i,G_i$ and $V_i$ as well as the inequality of arithmetic and geometric means to obtain the last inequality. Now, using the identity
\begin{align*}
E[h(Q(Y))]-E[h(Q(G))]
&=\sum_{i=1}^N\left(E[h(Q(W^{(i)}))]-E[h(Q(W^{(i-1)}))]\right),
\end{align*}
we obtain the desired result.
\end{proof}


\begin{lemma}\label{coupling-general}
For any $\varepsilon>0$ and any Borel set $A$ of $\mathbb{R}$, we have
\begin{align*}
P(Q_\vee(Y)\in A)\leq P(Z_\vee\in A^{5\varepsilon})+C\left\{\varepsilon^{-3}(\log^2 d)R_{1}+\varepsilon^{-2}(\log d)R_{2}+\varepsilon^{-2}(\log^2d)R_{3}\right\},
\end{align*}
where $C>0$ is a universal constant.
\end{lemma}

\begin{proof}
Let us set $\beta=\varepsilon^{-1}\log d$ (hence $\beta^{-1}\log d=\varepsilon$). From \eqref{max-smooth} we have
\[
P(Q_\vee(Y)\in A)\leq P(\Phi_\beta(Q(Y))\in A^{\varepsilon})=E[1_{A^{\varepsilon}}(\Phi_\beta(Q(Y)))].
\]
Next, by Lemma \ref{cck-approx} there is a $C^\infty$ function $g:\mathbb{R}\to\mathbb{R}$ and a universal constant $K>0$ such that $\|g'\|_\infty\leq\varepsilon^{-1}$, $\|g''\|_\infty\leq K\varepsilon^{-2}$, $\|g'''\|_\infty\leq K\varepsilon^{-3}$ and $1_{A^{\varepsilon}}(t)\leq g(t)\leq1_{A^{4\varepsilon}}(t)$ for any $t\in\mathbb{R}$. Then we obtain $E[1_{A^{\varepsilon}}(\Phi_\beta(Q(Y)))]\leq E[g(\Phi_\beta(Q(Y)))].$ 
Now, by Lemma \ref{stein-general} we have
\[
|E[g(\Phi_\beta(Q(Y)))]-E[g(\Phi_\beta(Q(G)))]|\leq C_1\left(\varepsilon^{-3}+\varepsilon^{-3}\log d+\varepsilon^{-3}\log^2d\right)R_1
\leq C_2 \varepsilon^{-3}(\log^2d)R_1,
\]
where $C_1,C_2>0$ denote universal constants. 
Moreover, as in the proof of Theorem \ref{gqf}, we can apply Proposition \ref{stein} as well as Lemma \ref{fourth-moment} to $Q(G)$ and thus obtain
\[
|E[g(\Phi_\beta(Q(G)))]-E[g(\Phi_\beta(Z))]|\leq C_3\varepsilon^{-2}(\log d)\{R_2+(\log d)R_3\}.
\] 
with some universal constant $C_3>0$. Since we have
\begin{align*}
E[g(\Phi_\beta(Z))]
\leq E[1_{A^{4\varepsilon}}(\Phi_\beta(Z))]
\leq E[1_{A^{5\varepsilon}}(Z_\vee)]=P(Z_\vee\in A^{5\varepsilon}),
\end{align*}
the proof is completed.
\end{proof}

\begin{proof}[\upshape\bfseries Proof of Theorem \ref{kolmogorov-general}]
By Lemmas \ref{coupling-general}, \ref{cck-kolmogorov} and \ref{cck-nazarov}, we obtain
\begin{align*}
&\sup_{x\in\mathbb{R}}\left|P(Q_\vee(Y)\leq x)-P(Z_\vee\leq x)\right|\\
&\leq C\left\{\varepsilon\sqrt{\log d_n}
+\varepsilon^{-3}(\log^2 d_n)R_{n,1}+\varepsilon^{-2}(\log d_n)R_{n,2}+\varepsilon^{-2}(\log^2d_n)R_{n,3}\right\}
\end{align*}
where $C>0$ depends only on $b$. 
Therefore, taking $\varepsilon=R_{n,1}^{1/4}(\log d_n)^{3/8}\vee R_{n,2}^{1/3}(\log d_n)^{1/6}\vee R_{n,3}^{1/3}\sqrt{\log d_n}$, we obtain the desired result.
\end{proof}

\subsection{Proof of Lemma \ref{lemma:hyper}}\label{proof:hyper}

We first prove the claim (a). 
For the proof it is convenient to introduce the notion of the hypercontractivity of a random variable. Let us recall that a random variable $Y$ is said to be \textit{$(p,q,\eta)$-hypercontractive} for $1\leq p\leq q<\infty$ and $0<\eta<1$ if 
\[
\|a+\eta Y\|_q\leq\|a+ Y\|_p
\]
for all $a\in\mathbb{R}$. We refer to  \cite{KS1988PTRF}, \cite{KW1992} and \cite{MOO2010} for more details on this notion. We note that, if $Y$ is $(p,q,\eta)$-hypercontractive and $E[Y]=0$, then $Y$ is $(p,q,\eta')$-hypercontractive for all $\eta'\in(0,\eta)$ (this follows from Lemma 1.1 of \cite{KS1988PTRF}). Lemma \ref{lemma:hyper} is obtained as a special case of the following result:
\begin{lemma}
Suppose that there are numbers $p\in[3,\infty)$ and $\eta\in(0,1)$ such that $Y_i$ is $(2,p,\eta)$-hypercontractive for all $i=1,\dots,N_n$. Then we have
\begin{align*}
\sum_{i=1}^{N_n}E\left[\max_{1\leq k\leq d_n}\left|\sum_{j=1}^{N_n}\gamma_{n,k}(i,j)W^{(i)}_j\right|^3\right]
\leq d_n^{3/p}\eta^{-3}\sum_{i=1}^{N_n}\max_{1\leq k\leq d_n}\left(\sum_{j=1}^{N_n}\gamma_{n,k}(i,j)^2\right)^{3/2}.
\end{align*}
\end{lemma}

\begin{proof}
First we note that $\eta\leq(p-1)^{-1/2}$ (see page 333 of \cite{MOO2010}). Then, since $G_i$ is $(2,p,(p-1)^{-1/2})$-hypercontractive for all $i$ by Corollary 3.4.1 of \cite{KW1992}, $W^{(i)}_j$ is $(2,p,\eta)$-hypercontractive for all $i,j$. Therefore, the desired inequality follows from the Lyapunov inequality and Remark 3.3.1 of \cite{KW1992}. 
\end{proof}

\begin{proof}[\upshape\bfseries Proof of Lemma \ref{lemma:hyper}(a)]
By Proposition 3.16 from \cite{MOO2010}, $Y_i$ is $(2,p,(2\sqrt{p-1}\|Y_i\|_p)^{-1})$-hypercontractive for all $i=1,\dots,N_n$. Hence the desired result immediately follows from the above lemma.
\end{proof}

Next we turn to the claim (b). For the proof we use some elementary properties on the sub-Gaussian property which we enumerate for convenience. 
In this case we always have $E[Y]=0$ and $\variance[Y]\leq b^2$ (cf.~Proposition 2.1 of \cite{subgauss2012}). 
A standard Gaussian variable is evidently sub-Gaussian relative to the scale 1. 
If $Y_1,\dots,Y_m$ are independent random variables and $Y_i$ is sub-Gaussian relative to the scale $a_i$ for each $i=1,\dots,m$, then $\sum_{i=1}^mY_i$ is sub-Gaussian relative to the scale $\sqrt{\sum_{i=1}^ma_i^2}$ (cf.~Lemma 3 of \cite{BK1981}). 
Finally, if $Y$ is sub-Gaussian relative to the scale $a$, then $Y$ is a sub-1st chaos random variable relative to the scale $\sqrt{5}a$ (see Remark 3.1 of \cite{VV2007}). 

\begin{proof}[\upshape\bfseries Proof of Lemma \ref{lemma:hyper}(b)]
First note that $a\geq\sqrt{\variance[Y_1]}=1$. Therefore, the variable $W^{(i)}_j$ is sub-Gaussian relative to the scale $a$ for all $i,j$, and thus the variable $\sum_{j=1}^{N_n}\gamma_{n,k}(i,j)W^{(i)}_j$ is sub-Gaussian relative to the scale $a\sqrt{\sum_{j=1}^{N_n}\gamma_{n,k}(i,j)^2}$ for all $i=1,\dots,N_n$ and $k=1,\dots,d_n$. Consequently, Proposition \ref{max-sub-chaos} yields 
\[
E\left[\max_{1\leq k\leq d_n}\left|\sum_{j=1}^N\gamma_k(i,j)W^{(i)}_j\right|^3\right]
\leq \left(5a^2\log(2d_n-1+\sqrt{e})\max_{1\leq k\leq d_n}\sum_{j=1}^{N_n}\gamma_{n,k}(i,j)^2\right)^{3/2}
\]
for all $i$. This completes the proof. 
\end{proof}

\subsection{Proof of Corollary \ref{criterion}}

Since we have
\[
\sum_{i=1}^{N_n}\left(\max_{1\leq k\leq d_n}\sum_{j=1}^{N_n}\gamma_{n,k}(i,j)^2\right)^{3/2}
\leq\left(\max_{1\leq i\leq N_n}\sqrt{\Lambda_i}\right)\sum_{i=1}^{N_n}\Lambda_i,
\]
the convergence $R_{n,1}\log^{\frac{7}{2}}d_n\to0$ follows from Lemma \ref{lemma:hyper} and assumptions. 
The convergence $R_{n,3}\log^3 d_n\to0$ follows from Eq.(11) of \cite{DY2011} and assumptions.\hfill$\square$

\subsection{Proof of Proposition \ref{hry}}

In the following subsections, for (possibly random) sequences $(x_n)$ and $(y_n)$, $x_n\lesssim y_n$ means that there exists a (non-random) constant $C\in[0,\infty)$, which depends only on the model parameters such as $\sigma_1,\sigma_2$ and the constants appearing in assumptions, such that $x_n\leq Cy_n$ a.s.~for large $n$. 
\begin{lemma}\label{hry-variance}
Under the assumptions of Proposition \ref{hry}, we have
\[
\left\{1-\left(\sum_{m=1}^M|\rho_m|\right)^2\right\}\underline{v}\leq E[F_n(\theta)^2]\leq\left\{1+\left(\sum_{m=1}^M|\rho_m|\right)^2\right\}\overline{v}
\]
for every $n\in\mathbb{N}$ and every $\theta\in\mathcal{G}_n$. 
\end{lemma}

\begin{proof}

The Isserlis formula (cf.~Theorem 1.28 from \cite{Janson1997}) yields
\begin{align*}
E[F_n(\theta)^2]
&=n\sum_{I,I'\in\Pi^1_n,J,J'\in\Pi^2_n}\covariance\left[X^1(I)X^2(J),X^1(I')X^2(J')\right]K(I,J_{-\theta})K(I',J'_{-\theta})\\
&=n\sum_{I\in\Pi^1_n,J\in\Pi^2_n}E\left[X^1(I)^2\right]E\left[X^2(J)^2\right]K(I,J_{-\theta})\\
&\quad+n\sum_{I,I'\in\Pi^1_n,J,J'\in\Pi^2_n}E\left[X^1(I)X^2(J')\right]E\left[X^1(I')X^2(J)\right]K(I,J_{-\theta})K(I',J'_{-\theta}).
\end{align*}
Now let us set $I(S)=\bigcup_{I\in\Pi^1_n:I\cap S\neq\emptyset}I$ for every interval $S$. Then we have 
\begin{align*}
&\left|n\sum_{I,I'\in\Pi^1_n,J,J'\in\Pi^2_n}E\left[X^1(I)X^2(J')\right]E\left[X^1(I')X^2(J)\right]K(I,J_{-\theta})K(I',J'_{-\theta})\right|\\
&=n\left|\sum_{m,m'}\rho_m\rho_{m'}\sum_{I,J,I',J'}K(I,J_{-\theta})K(I',J'_{-\theta})\int_{I\cap J'_{-\theta_m}}\sigma_1(t)\sigma_2(t+\theta_m)dt\int_{I'\cap J_{-\theta_{m'}}}\sigma_1(t)\sigma_2(t+\theta_{m'})dt\right|\\
&\leq n\sum_{m,m'}|\rho_m\rho_{m'}|\sum_{I,J,I',J'}K(I,J_{-\theta})K(I',J'_{-\theta})\\
&\qquad\times\sqrt{\int_{I\cap J'_{-\theta_m}}\sigma_1(t)^2dt\int_{I\cap J'_{-\theta_m}}\sigma_2(t+\theta_m)^2dt\int_{I'\cap J_{-\theta_{m'}}}\sigma_1(t)^2dt\int_{I'\cap J_{-\theta_{m'}}}\sigma_2(t+\theta_{m'})^2dt}\\
&\leq n\sum_{m,m'}|\rho_{m'}\rho_m|\sum_{I,J,I',J'}K(I,J_{-\theta})K(I',J'_{-\theta})\int_{I\cap J'_{-\theta_m}}\sigma_1(t)^2dt\int_{I'\cap J_{-\theta_{m'}}}\sigma_2(t+\theta_{m'})^2dt\\
&= n\sum_{m,m'}|\rho_{m'}\rho_m|\sum_{I,J,J'}K(I,J_{-\theta})\int_{I\cap J'_{-\theta_m}}\sigma_1(t)^2dt\int_{I(J'_{-\theta})\cap J_{-\theta_{m'}}}\sigma_2(t+\theta_{m'})^2dt\\
&\leq n\sum_{m,m'}|\rho_{m'}\rho_m|\sum_{I,J}K(I,J_{-\theta})\int_{I}\sigma_1(t)^2dt\int_{J_{-\theta_{m'}}}\sigma_2(t+\theta_{m'})^2dt\\
&=n\left(\sum_m|\rho_m|\right)^2\sum_{I,J}K(I,J_{-\theta})\int_{I}\sigma_1(t)^2dt\int_{J}\sigma_2(t)^2dt
=\left(\sum_m|\rho_m|\right)^2V_n(\theta),
\end{align*}
hence we obtain the desired result.
\end{proof}

\begin{lemma}\label{hry-fourth}
Under the assumptions of Proposition \ref{hry}, there is a constant $C>0$ such that
\[
E[F_n(\theta)^4]-3E[F_n(\theta)^2]^2\leq Cnr_n^2
\]
for every $n\in\mathbb{N}$ and every $\theta\in\mathcal{G}_n$.
\end{lemma}

\begin{proof}

First we apply the so-called reduction procedures used in \cite{HY2008,HY2011} to every realization of $(I)_{I\in\Pi^1_n}$ and $(J_{-\theta})_{J\in\Pi^2_n}$ (see also the proof of Lemma 2 from \cite{DY2011}). We define a new partition $\tilde{\Pi}^1_n$ as follows: $I\in\tilde{\Pi}^1_n$ if and only if either $I\in\Pi^1_n$ and it has non-empty intersection with two distinct intervals from $\{J_{-\theta}:J\in\Pi^2_n\}$ or there is $J\in\Pi^2_n$ such that $I$ is the union of all intervals from $\Pi^1_n$ included in $J_{-\theta}$. We also define a new partition $\tilde{\Pi}^2_n$ as follows: $J\in\tilde{\Pi}^2_n$ if and only if either $J\in\Pi^2_n$ and $J_{-\theta}$ has non-empty intersection with two distinct intervals from $\Pi^1_n$ or there is $I\in\Pi^1_n$ such that $J$ is the union of all intervals from $J'\in\Pi^2_n$ such that $J'_{-\theta}$ is included in $I$. Due to bilinearity $F_n(\theta)$ is invariant under this procedure. $r_N$ is also unchanged by this application because of its definition. Moreover, by construction we have
\begin{equation*}
\max_{J\in\tilde{\Pi}^2_n}\sum_{I\in\tilde{\Pi}^1_n}K(I,J_{-\theta})\leq3,\qquad
\max_{I\in\tilde{\Pi}^1_n}\sum_{J\in\tilde{\Pi}^2_n}K(I,J_{-\theta})\leq3.
\end{equation*}
Consequently, for the proof we may replace $(\Pi^1_n,\Pi^2_n)$ by $(\tilde{\Pi}^1_n,\tilde{\Pi}^2_n)$. This allows us to assume that
\begin{equation}\label{reduction}
\max_{J\in\Pi^2_n}\sum_{I\in\Pi^1_n}K(I,J_{-\theta})\leq3,\qquad
\max_{I\in\Pi^1_n}\sum_{J\in\Pi^2_n}K(I,J_{-\theta})\leq3
\end{equation}
throughout the proof without loss of generality.

We turn to the main body of the proof. Let $\Sigma_n$ be the covariance matrix of
$
(X^1(I))_{I\in\Pi^1_n},X^2(J))_{J\in\Pi^2_n})^\top.
$
Set $\Gamma_{n}(\theta)=\Sigma_{n}^{1/2}A_n(\theta)\Sigma_{n}^{1/2}$, where
\[
A_n(\theta)=\left(
\begin{array}{cc}
0  &   K_n(\theta)   \\
K_n(\theta)^\top  &   0
\end{array}
\right),\qquad
K_n(\theta)=(\sqrt{n}K(I,J_{-\theta})/2)_{I\in\Pi^1_n,J\in\Pi^2_n}.
\]
From Eq.(11) of \cite{DY2011} we have
\[
E[F_n(\theta)^4]-3E[F_n(\theta)^2]^2=2^33!\trace(\Gamma_n(\theta)^4)
\leq48\|\Gamma_n(\theta)\|_\mathrm{sp}^2\|\Gamma_n(\theta)\|_F^2
=24\|\Gamma_n(\theta)\|_\mathrm{sp}^2E[F_n(\theta)^2].
\]
Now, Lemma \ref{hry-variance} yields
\[
E[F_n(\theta)^2]\leq\left\{1+\left(\sum_m|\rho_m|\right)^2\right\}\overline{v}.
\]
On the other hand, by Example 5.6.5 and Theorem 5.6.9 of \cite{HJ2013} as well as \eqref{reduction}, we have $\|A_n(\theta)\|_\mathrm{sp}\leq3\sqrt{n}/2$. Therefore, Corollary 4.5.11 (the Ostorowski theorem), Example 5.6.5 and Theorem 5.6.9 of \cite{HJ2013} imply that
\begin{align*}
\|\Gamma_{n}(\theta)\|_\mathrm{sp}
&\leq\frac{3}{2}\sqrt{n}\|\Sigma_{n}\|_\mathrm{sp}
\leq\frac{3}{2}\sqrt{n}\max\left\{\max_{I\in\Pi^1_n}E\left[X^1(I)^2\right],\max_{J\in\Pi^2_n}E\left[X^2(J)^2\right]\right\}\\
&\hphantom{\leq\frac{3}{2}\sqrt{n}\|\Sigma_{N}\|_\mathrm{sp}
\leq}+\frac{3}{2}\sqrt{n}\max_{I\in\Pi^1_n}\sum_{J\in\Pi^2_n}\left|E\left[X^1(I)X^2(J)\right]\right|\\
&\lesssim \sqrt{n}r_n+\sqrt{n}\max_{I\in\Pi^1_n}\sum_{J\in\Pi^2_n}\left|E\left[X^1(I)X^2(J)\right]\right|.
\end{align*}
Since we have
\begin{align*}
\sum_{J\in\Pi^2_n}\left|E\left[X^1(I)X^2(J)\right]\right|
\leq\sum_{J\in\Pi^2_n}\sum_m|\rho_m|\int_{I\cap J_{-\theta_m}}\sigma_1(t)\sigma_2(t+\theta_m)dt
\lesssim r_n,
\end{align*}
we obtain the desired result.
\end{proof}

\begin{proof}[\bfseries\upshape Proof of Proposition \ref{hry}]

The result is an immediate consequence of Theorem \ref{gqf} and Lemmas \ref{hry-variance}--\ref{hry-fourth}.
\end{proof}

\subsection{Proof of Proposition \ref{hry-boot}}

Throughout this subsection, we set
\[
\mathbb{H}=\left\{\int_{-\infty}^\infty f(t)dB^1_t+\int_{-\infty}^\infty g(t)dB^2_t:f,g\in L^2(\mathbb{R})\right\}.
\]
$\mathbb{H}$ is obviously a Gaussian Hilbert space. Also, for each $\nu=1,2$ we define the process $v^\nu=(v^\nu_t)_{t\geq0}$ by $v^\nu_t=\int_0^t\sigma_\nu(s)^2ds$, $t\geq0$.

\begin{lemma}\label{boot-dist}
Under the assumptions of Proposition \ref{hry-boot}\ref{null}, we have
\[
\sup_{x\in\mathbb{R}}\left|P\left(T_n^*\leq x|\mathcal{F}^X\right)-P\left(\max_{\theta\in\mathcal{G}_n}|Z_n(\theta)|\leq x\right)\right|\to^p0
\]
as $n\to\infty$, where the process $(Z_n(\theta))_{\theta\in\mathcal{G}_n}$ is the same one as in Proposition \ref{hry}.  
\end{lemma}

\begin{proof}
Define the $n_1\times n_2$ matrix $\Xi_n(\theta)$ by 
$\Xi_n(\theta)=
(\frac{1}{2}X^1(I)K(I,J_{-\theta})X^2(J))_{I,J},$ 
and set
\[
\widetilde{\Xi}_n(\theta)=
\left(
\begin{array}{cc}
O  &  \Xi_n(\theta) \\
\Xi_n(\theta)^\top  &  O   
\end{array}
\right).
\]
Then we can rewrite $U_n^*(\theta)$ as $U_n^*(\theta)=\boldsymbol{w}^\top\widetilde{\Xi}_n(\theta)\boldsymbol{w}$, where $\boldsymbol{w}=((w^1_I)_{I\in\Pi^1_n},(w^2_J)_{J\in\Pi^2_n})^\top$. Therefore, by Theorem \ref{kolmogorov-general} and Lemma \ref{lemma:hyper}, it suffices to prove $R^*_{n,1}\log^{10}(\#\mathcal{G}_n)\vee R^*_{n,2}\log^{10}(\#\mathcal{G}_n)\vee R^*_{n,3}\log^2(\#\mathcal{G}_n) \vee R^*_{n,4}\log^3(\#\mathcal{G}_n)\to^p0$, where
\begin{align*}
R^*_{n,1}&=\sum_{I}\max_{\theta\in\mathcal{G}_n}\left(n\sum_{J}X^1(I)^2X^2(J)^2K(I,J_{-\theta})\right)^{3/2},\\
R^*_{n,2}&=\sum_{J}\max_{\theta\in\mathcal{G}_n}\left(n\sum_{I}X^1(I)^2X^2(J)^2K(I,J_{-\theta})\right)^{3/2},\\
R^*_{n,3}&=\sqrt{n}\max_{\theta,\theta'\in\mathcal{G}_n}\left|E[U_n(\theta)U_n(\theta')]-E[U^*_n(\theta)U^*_n(\theta')|\mathcal{F}^X]\right|,\\
R^*_{n,4}&=\max_{1\leq k\leq d}\sqrt{E[Q^*_n(\theta)^4|\mathcal{F}^X]-3E[Q^*_n(\theta)^2|\mathcal{F}^X]^2}.
\end{align*}
Here, $Q^*_n(\theta)=\sqrt{n}\boldsymbol{G}^\top\widetilde{\Xi}_n(\theta)\boldsymbol{G}$ and 
$\boldsymbol{G}=((G^1_I)_{I\in\Pi^1_n},(G^2_J)_{J\in\Pi^2_n})^\top$ 
with $(G^1_I)_{I\in\Pi^1_n}$ and $(G^2_J)_{J\in\Pi^2_n}$ being mutually independent sequence of i.i.d.~standard Gaussian variables independent of $\mathcal{F}^X$.

First we consider $R^*_{n,1}$ and $R_{n,2}^*$. Noting that the inequalities
\begin{equation}\label{reduction2}
\sum_{I'}v^1(I')K(I',J_{-\theta})\leq3\left(\sup_{0\leq t\leq T}\sigma_1(t)^2\right)r_n,\quad
\sum_{J'}v^2(J')K(I,J'_{-\theta})\leq3\left(\sup_{0\leq t\leq T}\sigma_2(t)^2\right)r_n,
\end{equation}
which hold for every $n$ and all $I\in\Pi^1_n$, $J\in\Pi^2_n$ and $\theta\in\mathcal{G}_n$, by the triangular inequality we have
\begin{align*}
\left\|\sum_{J}X^2(J)^2K(I,J_{-\theta})\right\|_2
\leq\sum_{J}\|X^2(J)^2\|_2K(I,J_{-\theta})
\leq\sqrt{3}\sum_Jv^2(J)K(I,J_{-\theta})
\lesssim r_n.
\end{align*}
Since $\sum_{J}X^2(J)^2K(I,J_{-\theta})\in\overline{\mathcal{P}}_2(\mathbb{H})$, the Schwarz inequality and Proposition \ref{max-sub-chaos} yield
\begin{align*}
E[R_{n,1}^*]
&\leq n^{3/2}\sum_{I}\sqrt{E\left[X^1(I)^6\right]}\sqrt{E\left[\max_{\theta\in\mathcal{G}_n}\left(\sum_{J}X^2(J)^2K(I,J_{-\theta})\right)^{3}\right]}\\
&\lesssim n^{3/2}\sqrt{r_n}\{r_n\log(\#\mathcal{G}_n)\}^{3/2}.
\end{align*}
Therefore, we obtain $R^*_{n,1}\log^{10}(\#\mathcal{G}_n)\to0$ by the Markov inequality due to $r_n=O(n^{-3/4+\eta})$. 
We can prove $R^*_{n,2}\log^{10}(\#\mathcal{G}_n)\to0$ in a similar manner.

Next we consider $R_{n,3}^*$. Since we have
\begin{align*}
&E[U_n(\theta)U_n(\theta')]-E[U^*_n(\theta)U^*_n(\theta')|\mathcal{F}^X]\\
&=\sum_{I,J}\{E[X^1(I)^2X^2(J)^2]-X^1(I)^2X^2(J)^2\}K(I,J_{-\theta})K(I,J_{-\theta'})
\in\overline{\mathcal{P}}_4(\mathbb{H})
\end{align*}
and
\begin{align*}
&E[|E[U_n(\theta)U_n(\theta')]-E[U^*_n(\theta)U^*_n(\theta')|\mathcal{F}^X]|^2]\\
&=\sum_{I,J,I',J'}\{E[X^1(I)^2X^2(J)^2X^1(I')^2X^2(J')^2]-E[X^1(I)^2X^2(J)^2]E[X^1(I')^2X^2(J')^2]\}\\
&\hphantom{=\sum_{I,J,I',J'}}\times K(I,J_{-\theta})K(I,J_{-\theta'})K(I',J'_{-\theta})K(I',J'_{-\theta'})\\
&\lesssim\sum_{I,J,J'}v^1(I)^2v^2(J)v^2(J')K(I,J_{-\theta})K(I,J_{-\theta'})K(I,J'_{-\theta})K(I,J'_{-\theta'})\\
&\qquad+\sum_{I,J,I'}v^1(I)v^1(I')v^2(J)^2K(I,J_{-\theta})K(I,J_{-\theta'})K(I',J_{-\theta})K(I',J_{-\theta'})\\
&\lesssim n^{-1}r_n^2
\end{align*}
by the hypothesis $H_0$ and \eqref{reduction2}, Proposition \ref{max-sub-chaos} yields $E[R_{n,3}^*]=O(\sqrt{n}r_n\log^2(\#\mathcal{G}_n))$. Hence we obtain $R^*_{n,3}\log^2(\#\mathcal{G}_n)\to^p0$ due to the Markov inequality.

Finally we prove $R^*_{n,4}\log^4(\mathcal{G}_n)\to^p0$. From Eq.(11) of \cite{DY2011} we have 
\begin{align*}
&E[Q^*_n(\theta)^4|\mathcal{F}^X]-3E[Q^*_n(\theta)^2|\mathcal{F}^X]^2=48n^2\trace(\widetilde{\Xi}_n(\theta)^4)
=96n^2\|\Xi_n(\theta)\Xi_n(\theta)^\top\|_F^2\\
&\lesssim n^2\sum_{I,I'}\left(\sum_JX^1(I)X^1(I')X^2(J)^2K(I,J_{-\theta})K(I',J_{-\theta})\right)^2\\
&=n^2\sum_{I,I',J,J'}X^1(I)^2X^1(I')^2X^2(J)^2X^2(J')^2K(I,J_{-\theta})K(I',J_{-\theta})K(I,J'_{-\theta})K(I',J'_{-\theta}).
\end{align*}
In particular, we obtain $E[Q^*_n(\theta)^4|\mathcal{F}^X]-3E[Q^*_n(\theta)^2|\mathcal{F}^X]^2\in\overline{\mathcal{P}}_8(\mathbb{H})$. Moreover, the triangular and the generalized H\"older inequalities as well as \eqref{reduction2} yield
\begin{align*}
&\left\|n^2\sum_{I,I',J,J'}X^1(I)^2X^1(I')^2X^2(J)^2X^2(J')^2K(I,J_{-\theta})K(I',J_{-\theta})K(I,J'_{-\theta})K(I',J'_{-\theta})\right\|_2\\
&\lesssim n^2\sum_{I,I',J,J'}v^1(I)v^1(I')v^2(J)v^2(J')K(I,J_{-\theta})K(I',J_{-\theta})K(I,J'_{-\theta})K(I',J'_{-\theta})\\
&\lesssim n^2r_n^2\sum_{I,J}v^1(I)v^2(J)K(I,J_{-\theta})
\leq nr_n^2\overline{v}.
\end{align*}
Therefore, Proposition \ref{max-sub-chaos} implies that $E[R^*_{n,4}]\lesssim \sqrt{n}r_n\log^2(\#\mathcal{G}_n),$ hence the desired result follows from the Markov inequality. Thus we complete the proof.
\end{proof}

\begin{lemma}\label{hry-consistent}
Under the assumptions of Proposition \ref{hry-boot}\ref{alternative}, we have $T_n/\sqrt{n}\to^p\max_{1\leq m\leq M}|\rho_m|\Sigma(\theta_m)$ as $n\to\infty$.
\end{lemma}

\begin{proof}
We first note that $\upsilon_n=o(r_n)$ as $n\to\infty$. In fact, we have $T\leq\sum_I|I|+2r_n\leq(n_1+2)r_n$ and $T\leq\sum_J|J|+2r_n\leq(n_2+2)r_n$, hence it holds that $T-2r_n\leq\upsilon_n\min\{n_1,n_2\}\cdot\upsilon_n^{-1}r_n$. Since $\upsilon_n\min\{n_1,n_2\}\to0$ by \ref{grid}, we necessarily have $\upsilon_n^{-1}r_n\to\infty$. Therefore, without loss of generality we may assume that $\upsilon_n\leq r_n$.

First we show that
\begin{equation}\label{eq:limsup}
\limsup_{n\to\infty}\max_{\theta\in\mathcal{G}_n}|E[U_n(\theta)]|\leq\max_{1\leq m\leq M}|\rho_m|\Sigma(\theta_m).
\end{equation}
Noting that 
\begin{equation}\label{intersect}
I\cap J_{-\theta}\neq\emptyset,I\cap J_{-\theta'}\neq\emptyset\Rightarrow|\theta'-\theta|\leq2r_n
\end{equation}
for any $I\in\Pi^1_n$, $J\in\Pi^2_n$ and $\theta,\theta'\in\mathcal{G}_n$, we have
\begin{align*}
|E[U_n(\theta)]|&=\left|\sum_{I,J}K(I,J_{-\theta})\sum_{m=1}^M\rho_m\int_{I\cap J_{-\theta_m}}\sigma_1(t)\sigma_2(t+\theta_m)dt\right|\\
&\leq\max_{1\leq m\leq M}|\rho_m|\sum_{I,J}\int_{I\cap J_{-\theta_m}}\sigma_1(t)\sigma_2(t+\theta_m)dt
\end{align*}
for any $\theta\in\mathcal{G}_n$ and sufficiently large $n$. This yields \eqref{eq:limsup}.

Next we prove 
\begin{equation}\label{eq:liminf}
\liminf_{n\to\infty}\max_{\theta\in\mathcal{G}_n}|E[U_n(\theta)]|\geq\max_{1\leq m\leq M}|\rho_m|\Sigma(\theta_m).
\end{equation}
Let $m^*$ be an integer such that $|\rho_{m^*}|\Sigma(\theta_{m^*})=\max_{1\leq m\leq M}|\rho_m|\Sigma(\theta_{m})$. By assumption \ref{grid}, for each $n\in\mathbb{N}$ there is a number $\vartheta_n\in\mathcal{G}_n$ such that $|\vartheta_n-\theta_{m^*}|\leq\upsilon_n$. Now noting \eqref{intersect}, we have
\begin{align*}
E[U_n(\vartheta_n)]
&=\sum_{I,J}K(I,J_{-\vartheta_n})\rho_{m^*}\int_{I\cap J_{-\theta_{m^*}}}\sigma_1(t)\sigma_2(t+\theta_{m^*})dt
\end{align*}
for sufficiently large $n$. Let us denote by $\ominus$ the symmetric difference between two sets. Then we have
\begin{align*}
&\left|\sum_{I,J}K(I,J_{-\vartheta_n})\rho_{m^*}\left\{\int_{I\cap J_{-\theta_{m^*}}}\sigma_1(t)\sigma_2(t+\theta_{m^*})dt-\int_{I\cap J_{-\vartheta_n}}\sigma_1(t)\sigma_2(t+\theta_{m^*})dt\right\}\right|\\
&\leq\sum_{I,J}K(I,J_{-\vartheta_n})\int_{J_{-\theta_{m^*}}\ominus J_{-\vartheta_n}}1_I(t)\sigma_1(t)\sigma_2(t+\theta_{m^*})dt\\
&\leq\sum_{J}\int_{J_{-\theta_{m^*}}\ominus J_{-\vartheta_n}}\sigma_1(t)\sigma_2(t+\theta_{m^*})dt
\lesssim n_2\upsilon_n.
\end{align*}
Noting that $K(I,J_{-\theta})=K(I_\theta,J)$, an analogous argument yields
\[
\left|\sum_{I,J}K(I,J_{-\vartheta_n})\rho_{m^*}\left\{\int_{I\cap J_{-\theta_{m^*}}}\sigma_1(t)\sigma_2(t+\theta_{m^*})dt-\int_{I\cap J_{-\vartheta_n}}\sigma_1(t)\sigma_2(t+\theta_{m^*})dt\right\}\right|
\lesssim n_1\upsilon_n.
\]
Therefore, by assumption \ref{grid} we obtain
\begin{align*}
E[U_n(\vartheta_n)]
&=\sum_{I,J}K(I,J_{-\vartheta_n})\rho_{m^*}\int_{I\cap J_{-\vartheta_n}}\sigma_1(t)\sigma_2(t+\theta_{m^*})dt+o(1)
=\rho_{m^*}\Sigma(\theta_{m^*})+o(1).
\end{align*}
Hence we conclude that
\begin{align*}
\liminf_{n\to\infty}\max_{\theta\in\mathcal{G}_n}|E[U_n(\theta)]|
\geq\liminf_{n\to\infty}|E[U_n(\vartheta_n)]|
=\rho_{m^*}\Sigma(\theta_{m^*})=\max_{1\leq m\leq M}|\rho_m|\Sigma(\theta_{m}),
\end{align*}
which yields \eqref{eq:liminf}.

From \eqref{eq:limsup}--\eqref{eq:liminf} we deduce that 
$\max_{\theta\in\mathcal{G}_n}|E[U_n(\theta)]|\to\max_{1\leq m\leq M}|\rho_m|\Sigma(\theta_m).$ 
Since we have
\begin{align*}
\left|\frac{T_n}{\sqrt{n}}-\max_{\theta\in\mathcal{G}_n}|E[U_n(\theta)]|\right|
\leq\frac{1}{\sqrt{n}}\max_{\theta\in\mathcal{G}_n}|F_n(\theta)|,
\end{align*}
the proof of the lemma is completed once we show that $\max_{\theta\in\mathcal{G}_n}|F_n(\theta)|=o_p(\sqrt{n})$. By Proposition \ref{max-sub-chaos} and the Markov inequality, we have $\max_{\theta\in\mathcal{G}_n}|Z_n(\theta)|=O_p(\sqrt{\log(\#\mathcal{G}_n)})$. Hence the desired result follows from Proposition \ref{hry}.
\end{proof}

\begin{lemma}\label{boot-consistent}
Under the assumptions of Proposition \ref{hry-boot}\ref{alternative}, we have $T_n^*=O_p(\log^2(\#\mathcal{G}_n))$ as $n\to\infty$. 
\end{lemma}

\begin{proof}
Take $p\geq2$ and $\theta\in\mathcal{G}_n$ arbitrarily. By Lemma 1 of \cite{BK1981} we have $E[|w^1_I|^p]\leq(2b/e)\sqrt{p}$ and $E[|w^1_J|^p]\leq(2b/e)\sqrt{p}$ for all $I\in\Pi^1_n$ and $J\in\Pi^2_n$. Therefore, by Proposition 3.16 from \cite{MOO2010} $w^1_I$ and $w^2_J$ are $(2,p,((4b/e)p)^{-1})$-hypercontractive for all $I\in\Pi^1_n$ and $J\in\Pi^2_n$ (see also the beginning of Section \ref{proof:hyper}). Therefore, by Theorem 6.5.2 from \cite{KW1992} it holds that
\begin{align*}
E[|U^*_n(\theta)|^p|\mathcal{F}^X]
\lesssim p^p\left(E[|U^*_n(\theta)|^2|\mathcal{F}^X]\right)^{p/2}.
\end{align*}
Since we have
\begin{align*}
E[|U^*_n(\theta)|^2|\mathcal{F}^X]
=\sum_{I\in\Pi^1_n,J\in\Pi^2_n}X^1(I)^2X^2(J)^2K(I,J_{-\theta})\in\overline{\mathcal{P}}_4(\mathbb{H})
\end{align*}
and
\begin{align*}
\left\|\sum_{I\in\Pi^1_n,J\in\Pi^2_n}X^1(I)^2X^2(J)^2K(I,J_{-\theta})\right\|_2
\lesssim\sum_{I\in\Pi^1_n,J\in\Pi^2_n}v^1(I)v^2(J)K(I,J_{-\theta})
\leq\overline{v}/n,
\end{align*}
by the triangular and the Schwarz inequalities, we obtain
\[
E\left[\left(E[|U^*_n(\theta)|^2|\mathcal{F}^X]\right)^{p/2}\right]\lesssim (p^2/n)^{p/2}.
\]
by Theorem 5.11 and Remark 5.11 from \cite{Janson1997}. As a result, we conclude that
\[
E[|U^*_n(\theta)|^p|\mathcal{F}^X]
\lesssim p^{2p}/n^{p/2}.
\] 
Therefore, Lemma \ref{criteria-subchaos} implies that $U^*_n(\theta)$ is a sub-4th chaos random variable relative to the scale $c/\sqrt{n}$ for some constant $c>0$. Hence Proposition \ref{max-sub-chaos} yields $E[T_n^*]=O(\log^2(\#\mathcal{G}_n))$ as $n\to\infty$. 
Consequently, we obtain the desired result by the Markov inequality.
\end{proof}

\begin{proof}[\upshape\bfseries Proof of Proposition \ref{hry-boot}]

(a) We follow Step 3 in the proof of Theorem 2 from \cite{KS2016}. First, by Proposition \ref{hry}, Lemma \ref{boot-dist} and Theorem 9.2.2 of \cite{Dudley2002} there is a sequence $\varepsilon_n$ of positive numbers tending to 0 such that 
\[
P\left(\mathcal{E}_n^c\right)\leq\varepsilon_n,\qquad
\sup_{x\in\mathbb{R}}\left|P\left(T_n\leq x\right)-P\left(\sup_{\theta\in\mathcal{G}_n}|Z_n(\theta)|\leq x\right)\right|\leq\varepsilon_n
\]
for all $n\in\mathbb{N}$, where
\[
\mathcal{E}_n=\left\{\sup_{x\in\mathbb{R}}\left|P\left(\max_{\theta\in\mathcal{G}_n}|Z_n(\theta)|\leq x\right)-P\left(T^*_n\leq x|\mathcal{F}^X\right)\right|\leq\varepsilon_n\right\}.
\]
Next, let us denote by $q_n^Z$ the quantile function of $\max_{\theta\in\mathcal{G}_n}|Z_n(\theta)|$. 
Note that $\max_{\theta\in\mathcal{G}_n}|Z_n(\theta)|$ has the density because of Lemmas \ref{hry-variance} and \ref{cck-nazarov}. Therefore, on $\mathcal{E}_n$ we have
\begin{align*}
P\left(T^*_n\leq q_n^Z(1-\alpha+\varepsilon_n)|\mathcal{F}^X\right)
\geq P\left(\max_{\theta\in\mathcal{G}_n}|Z_n(\theta)|\leq q_n^Z(1-\alpha+\varepsilon_n)\right)-\varepsilon_n
=1-\alpha,
\end{align*}
hence on $\mathcal{E}_n$ it holds that $q_n^*(1-\alpha)\leq q_n^Z(1-\alpha+\varepsilon_n).$ 
Therefore, we obtain
\begin{align*}
P\left(T_n< q_n^*(1-\alpha)\right)
&\leq P\left(T_n< q_n^Z(1-\alpha+\varepsilon_n)\right)+P(\mathcal{E}_n^c)\\
&\leq P\left(\max_{\theta\in\mathcal{G}_n}|Z_n(\theta)|< q_n^Z(1-\alpha+\varepsilon_n)\right)+2\varepsilon_n
=1-\alpha+3\varepsilon_n.
\end{align*}
On the other hand, for any $\omega\in\mathcal{E}_n$ and any $z\in\mathbb{R}$ such that $P(T_n^*\leq z|\mathcal{F}^X)(\omega)\geq1-\alpha$, we have
\begin{align*}
P\left(\max_{\theta\in\mathcal{G}_n}|Z_n(\theta)|\leq q_n^Z(1-\alpha-\varepsilon_n)\right)
=1-\alpha-\varepsilon_n
\leq P(T_n^*\leq z|\mathcal{F}^X)(\omega)-\varepsilon_n
\leq P\left(\max_{\theta\in\mathcal{G}_n}|Z_n(\theta)|\leq z\right),
\end{align*}
hence it holds that $q_n^Z(1-\alpha-\varepsilon_n)\leq z$. This implies that 
$q_n^*(1-\alpha)\geq q_n^Z(1-\alpha-\varepsilon_n)$ 
on $\mathcal{E}_n$. Therefore, we obtain
\begin{align*}
P\left(T_n< q_n^*(1-\alpha)\right)
&\geq P\left(T_n< q_n^Z(1-\alpha-\varepsilon_n)\right)-P(\mathcal{E}_n^c)\\
&\geq P\left(\max_{\theta\in\mathcal{G}_n}|Z_n(\theta)|< q_n^Z(1-\alpha-\varepsilon_n)\right)-2\varepsilon_n
=1-\alpha-3\varepsilon_n.
\end{align*}
Consequently, we obtain $P\left(T_n< q_n^*(1-\alpha)\right)\to1-\alpha$ as $n\to\infty$.

(b) By the definition of $q_n^*(1-\alpha)$ we have
\begin{align*}
P\left(T_n< q_n^*(1-\alpha)\right)
&\leq P\left(P(T_n^*\leq T_n|\mathcal{F}^X)< 1-\alpha\right)
\leq\alpha^{-1}P(T_n^*>T_n).
\end{align*}
Since Lemmas \ref{hry-consistent}--\ref{boot-consistent} yield $(T_n^*-T_n)/\sqrt{n}\to^p-\max_{1\leq m\leq M}|\rho_m|\Sigma(\theta_m)$, we have
\begin{align*}
\limsup_{n\to\infty}P(T_n^*>T_n)
\leq P\left(-\max_{1\leq m\leq M}|\rho_m|\Sigma(\theta_m)\geq0\right)=0
\end{align*}
under $H_1$, hence we obtain $P\left(T_n< q_n^*(1-\alpha)\right)\to0$. This implies the desired result.
\end{proof}

\subsection{Proof of Proposition \ref{spot}}

We begin by proving some auxiliary results. 
\begin{lemma}\label{lemma:kernel}
Suppose that the function $\Psi:\mathbb{R}\to\mathbb{R}$ is Lipschitz continuous and compactly supported. Under the assumptions of Proposition \ref{spot}, we have
\[
\sup_{t\in[a_n,T-a_n]}\left|\frac{1}{n}\sum_{i=1}^n\Psi_h(t_{i-1}-t)-\int_{-\infty}^\infty \Psi(s)ds\right|=O((nh)^{-1})
\]
as $n\to\infty$. 
\end{lemma}

\begin{proof}
Noting that $\int_{-\infty}^\infty\Psi(s)ds=\int_{-\infty}^\infty \Psi_h(s-t)ds$, we can decompose the target quantity as
\begin{align*}
&\sup_{t\in[a_n,T-a_n]}\left|\frac{1}{n}\sum_{i=1}^n\Psi_h(t_{i-1}-t)-\int_{-\infty}^\infty \Psi_h(s)ds\right|\\
&\leq\sup_{t\in[a_n,T-a_n]}\left\{\sum_{i=1}^n\int_{t_{i-1}}^{t_i}|\Psi_h(t_{i-1}-t)-\Psi_h(s-t)|ds
+\int_{-\infty}^0|\Psi_h(s-t)|ds+\int_T^\infty|\Psi_h(s-t)|ds\right\}\\
&=:\mathbb{I}_{n}+\mathbb{II}_{n}+\mathbb{III}_{n}.
\end{align*}
Since $\Psi$ is Lipschitz continuous and compactly supported, we have $\mathbb{I}_{n}=O((nh)^{-1})$ as $n\to\infty$. Moreover, since $a_n/h\to\infty$ as $n\to\infty$, $\mathbb{II}_{n}=\mathbb{III}_{n}=0$ for sufficiently large $n$. This completes the proof.
\end{proof}

\begin{lemma}\label{spot-var}
Under the assumptions of Proposition \ref{spot}, there are constants $C_1,C_2>0$ such that
\[
\frac{C_1}{\sqrt{nh}}\leq\mathfrak{s}_n(t)\leq\frac{C_2}{\sqrt{nh}}
\]
for any $n\in\mathbb{N}$ and $t\in[a_n,T-a_n]$.
\end{lemma}

\begin{proof}
Lemma \ref{lemma:kernel} yields
\[
nh\cdot\mathfrak{s}_n^2(t)=\frac{2}{nh}\sum_{i=1}^nK\left(\frac{t_{i-1}-t}{h}\right)^2=2\int_{-\infty}^\infty K(s)^2ds+O((nh)^{-1})
\]
uniformly in $t\in[a_n,T-a_n]$ as $n\to\infty$. Since $0<\int_{-\infty}^\infty K(s)^2ds<\infty$, we obtain the desired result.
\end{proof}

Now we turn to the main body of the proof. Let us set
\[
b_n(t)=\sum_{i=1}^nK_h(t_{i-1}-t)\int_{t_{i-1}}^{t_i}\sigma^2(s)ds-\sigma^2(t)
\]
and
\[
M_n(t)=2\sum_{i=1}^nK_h(t_{i-1}-t)\int_{t_{i-1}}^{t_i}\int_{t_{i-1}}^s\sigma(r)dB_r\sigma(s)dB_s
\]
for $t\in[0,T]$. By It\^o's formula we have
\begin{align*}
\widehat{\sigma}^2_n(t)-\sigma^2(t)
=M_n(t)+b_n(t).
\end{align*}
The term $b_n(t)$ behaves as a bias and it is negligible because we consider an undersmoothing situation such that $nh^{1+2\gamma}\log n\to0$. More precisely, we have the following:
\begin{lemma}\label{spot-bias}
Suppose that \ref{s-vol-cond} and \ref{kernel} are satisfied. Then there is a constant $C>0$ such that
\[ 
\sup_{t\in[a_n,T-a_n]}\left|b_n(t)\right|\leq C\left\{(nh)^{-1}+h^\gamma\right\}
\]
for sufficiently large $n$. 
\end{lemma}

\begin{proof}
We decompose $b_n(t)$ as
\begin{align*}
b_n(t)
&=\sum_{i=1}^nK_h(t_{i-1}-t)\int_{t_{i-1}}^{t_i}\{\sigma^2(s)-\sigma^2(t)\}ds+\sigma^2(t)\left\{\frac{1}{n}\sum_{i=1}^nK_h(t_{i-1}-t)-\int_{-\infty}^\infty K_h(s-t)ds\right\}\\
&=:\mathbb{I}_n(t)+\mathbb{II}_n(t).
\end{align*}
By assumptions we have $\sup_{t\in[a_n,T-a_n]}|\mathbb{I}_n(t)|\lesssim h^\gamma$ for sufficiently large $n$. In the meantime, Lemma \ref{lemma:kernel} and \ref{s-vol-cond} yield $\sup_{t\in[a_n,T-a_n]}|\mathbb{II}_n(t)|\lesssim (nh)^{-1}$. This completes the proof.
\end{proof}

Meanwhile, the term $M_n(t)$ is a martingale part and the source of stochastic estimation errors. To apply our theory to the investigation of this term, we approximate $M_n(t)/\sigma^2(t)\mathfrak{s}_n(t)$ by a double Wiener-It\^o integral. More precisely, we define
\[
M^0_n(t)=2\sum_{i=1}^nK_h(t_{i-1}-t)\int_{t_{i-1}}^{t_i}\int_{t_{i-1}}^sdB_rdB_s,\qquad t\in[0,T].
\]
Then we show that $M_n(t)/\{\sigma^2(t)\mathfrak{s}_n(t)\}$ is well-approximated by 
\[
F_n(t)=M_n^0(t)/\mathfrak{s}_n(t) 
\]
uniformly in $t\in[0,T]$. 
The proof of this approximation is motivated by the argument from Appendix A of \citet{Sabel2014} and relies on the following Burkholder-Davis-Gundy inequality with a sharp constant due to \citet{BY1982}:
\begin{proposition}[\citet{BY1982}, Proposition 4.2]\label{sharp-BDG}
There is a universal constant $c>0$ such that
\[
\left\|\sup_{0\leq t\leq T}|M_t|\right\|_p\leq c\sqrt{p}\left\|\langle M\rangle_T^{1/2}\right\|_p
\]
for any $p\geq2$ and any continuous martingale $M=(M_t)_{t\in[0,T]}$ with $M_0=0$. 
\end{proposition}
To simplify notation, for a random variable $Y$, we denote its Orlictz norm based on the function $\psi_1(x)=e^x-1$ by $\|Y\|_{\psi_1}$ (cf.~page 95 of \cite{VW1996}):
\[
\|Y\|_{\psi_1}=\inf\{C>0:\psi_1(|Y|/C)\leq1\}.
\]
Note that $\|Y\|_{\psi_1}\leq M$ if $Y$ is a sub-2nd chaos random variable relative to the scale $M$. Conversely, $Y$ is a sub-2nd chaos random variable relative to the scale $M$ if $\|Y\|_{\psi_1}< M$. 
\begin{lemma}\label{local-approx}
Suppose that \ref{s-vol-cond} and \ref{kernel} are satisfied. Then 
\[
\left\|\sup_{t\in[0,T]}\left|\frac{M_n(t)}{\sigma^2(t)\mathfrak{s}_n(t)}-F_n(t)\right|\right\|_{\psi_1}=O\left(h^\gamma(\log n)^2\right)
\]
as $n\to\infty$. 
\end{lemma}

\begin{proof}
Since $K$ is compactly supported, there is a constant $\ell>0$ such that the support of $K$ is contained in $[-\ell,\ell]$. Then we decompose the target quantity as
\begin{align*}
&\frac{M_n(t)}{\sigma^2(t)\mathfrak{s}_n(t)}-F_n(t)\\
&=\frac{M_n(t)-\sigma^2((t-\ell h)_+)M^0_n(t)}{\sigma^2(t)\mathfrak{s}_n(t)}
+\frac{\{\sigma^2((t-\ell h)_+)-\sigma^2(t)\}M^0_n(t)}{\sigma^2(t)\mathfrak{s}_n(t)}\\
&=:\mathbf{I}_n(t)+\mathbf{II}_n(t),
\end{align*}
where $(t-\ell h)_+=(t-\ell h)\vee0$. 

First we consider $\mathbf{I}_n(t)$. By \ref{s-vol-cond} and Lemma \ref{spot-var}, it is enough to prove
\begin{equation}\label{aim-local-approx}
\left\|\sup_{t\in[0,T]}\left|\sqrt{nh}\{M_n(t)-\sigma^2((t-\ell h)_+)M^0_n(t)\}\right|\right\|\lesssim h^\gamma(\log n)^2
\end{equation}
for all $n\in\mathbb{N}$. 
We decompose $\sqrt{nh}\{M_n(t)-\sigma^2((t-\ell h)_+)M^0_n(t)\}$ as
\begin{align*}
&\sqrt{nh}\{M_n(t)-\sigma^2((t-\ell h)_+)M^0_n(t)\}\\
&=2\sqrt{nh}\sum_{i=1}^nK_h(t_{i-1}-t)\int_{t_{i-1}}^{t_i}\int_{t_{i-1}}^s\{\sigma(r)-\sigma((t-\ell h)_+)\}dB_r\sigma(s)dB_s\\
&\quad+2\sqrt{nh}\sum_{i=1}^nK_h(t_{i-1}-t)\int_{t_{i-1}}^{t_i}\int_{t_{i-1}}^s\sigma((t-\ell h)_+)dB_r\{\sigma(s)-\sigma((t-\ell h)_+)dB_s\\
&=:\mathbf{A}_n(t)+\mathbf{B}_n(t).
\end{align*}
Let us consider $\mathbf{A}_n(t)$. Set $u_j=(jh)\wedge T$ for $j=0,1,\dots$. We obviously have
\[
\sup_{t\in[0,T]}\left|\mathbf{A}_n(t)\right|
\leq\max_{j=0,1,\dots,\lfloor T/h\rfloor}\left(\left|\mathbf{A}_n(u_j)\right|+\sup_{t\in[u_j,u_{j+1}]}\left|\mathbf{A}_n(t)-\mathbf{A}_n(u_j)\right|\right).
\]
To estimate the second term on the right side, we first show that there is a constant $a_0>0$ such that, for every $j=0,1,\dots,\lfloor T/h\rfloor$, the process $\mathbf{A}^{j}_n(t):=\mathbf{A}_n(t)-\mathbf{A}_n(u_j)$ indexed by $t\in[u_j,u_{j+1}]$ is a sub-2nd-Gaussian chaos process with respect to the semi-metric $\mathfrak{d}^j$ defined by
\[
\mathfrak{d}^j(t,t')=a_0|t-t'|^\gamma,\qquad
t,t'\in[u_j,u_{j+1}].
\]
To prove this, fix a $j\in\{0,1,\dots,\lfloor T/h\rfloor\}$ and take $t,t'\in[u_j,u_{j+1}]$ arbitrarily. We may assume $t\leq t'$ without loss of generality. Then, we can decompose $\mathbf{A}^j_n(t)-\mathbf{A}^j_n(t')$ as 
\begin{align*}
&\mathbf{A}^j_n(t)-\mathbf{A}^j_n(t')\\
&=\int_{0}^{T}\int_{0}^s\sum_{i=1}^n\Psi_{n,i}(t,t')1_{(t_{i-1},t_i]}(s)1_{(t_{i-1},s]}(r)\{\sigma(r)-\sigma((t-\ell h)_+)\}dB_r\sigma(s)dB_s\\
&\quad+\int_{0}^{T}\int_{0}^s\sum_{i=1}^n2\sqrt{nh}K_h(t_{i-1}-t')1_{(t_{i-1},t_i]}(s)1_{(t_{i-1},s]}(r)\{\sigma((t'-\ell h)_+)-\sigma((t-\ell h)_+)\}dB_r\sigma(s)dB_s\\
&=:\Delta^j_{n,1}(t,t')+\Delta^j_{n,2}(t,t'),
\end{align*}
where $\Psi_{n,i}(t,t')=2\sqrt{nh}\{K_h(t_{i-1}-t)-K_h(t_{i-1}-t')\}$. Regarding $\Delta^j_{n,1}(t,t')$, for every $p>1$, Proposition \ref{sharp-BDG}, the Lyapunov inequality, \ref{s-vol-cond} and \ref{kernel} yield
\begin{align*}
&\left\|\Delta^j_{n,1}(t,t')\right\|_p\\
&\lesssim \sqrt{p}\left\|\left\{\int_{0}^{T}\left(\int_{0}^s\sum_{i=1}^n\Psi_{n,i}(t,t')1_{(t_{i-1},t_i]}(s)1_{(t_{i-1},s]}(r)\{\sigma(r)-\sigma((u_j-\ell h)_+)\}dB_r\right)^2\sigma^2(s)ds\right\}^{1/2}\right\|_p\\
&\lesssim\sqrt{p}\left\|\int_{0}^{T}\left(\int_{0}^s\sum_{i=1}^n\Psi_{n,i}(t,t')1_{(t_{i-1},t_i]}(s)1_{(t_{i-1},s]}(r)\{\sigma(r)-\sigma((u_j-\ell h)_+)\}dB_r\right)^2ds\right\|_p^{1/2}\\
&\leq\sqrt{p}T^{\frac{1-p^{-1}}{2}}\left(\int_{0}^{T}\left\|\int_{0}^s\sum_{i=1}^n\Psi_{n,i}(t,t')1_{(t_{i-1},t_i]}(s)1_{(t_{i-1},s]}(r)\{\sigma(r)-\sigma((u_j-\ell h)_+)\}dB_r\right\|_{2p}^2ds\right)^{1/2}\\
&\lesssim  p\left(\int_{0}^{T}\left\|\left\{\int_{0}^s\sum_{i=1}^n\Psi_{n,i}(t,t')^21_{(t_{i-1},t_i]}(s)1_{(t_{i-1},s]}(r)\{\sigma(r)-\sigma((u_j-\ell h)_+)\}^2dr\right\}^{1/2}\right\|_{2p}^2ds\right)^{1/2}\\
&\lesssim  ph^\gamma\left(\frac{1}{n^2}\sum_{i=1}^n\Psi_{n,i}(t,t')^2\right)^{1/2}
\lesssim p\frac{|t-t'|}{h^{1-\gamma}}
\leq p|t-t'|^\gamma,
\end{align*}
where we use the inequality $|t-t'|\leq h$ which holds because $t,t'\in[u_j,u_{j+1}]$. 
Analogously we have
\begin{align*}
\left\|\Delta^j_{n,2}(t,t')\right\|_p
\lesssim p|t-t'|^\gamma.
\end{align*}
Therefore, by Lemma \ref{criteria-subchaos} we obtain the desired result. 
Now, since we have
\begin{align*}
\int_0^{a_0h^\gamma}\log N([u_j,u_{j+1}],\mathfrak{d}_0,r)dr
\lesssim \int_0^{a_0h^\gamma}\log\frac{1}{r^{1/\gamma}}dr
\lesssim h^\gamma\log n,
\end{align*}
by Proposition \ref{vv-tail} it holds that $\left\|\sup_{t\in[u_j,u_{j+1}]}\left|\mathbf{A}^j_n(t)\right|\right\|_{\psi_1}\lesssim h^\gamma\log n$ for every $j=0,1,\dots,\lfloor T/h\rfloor$. Thus, Lemma 2.2.2 of \cite{VW1996} implies that 
\[
\left\|\max_{j=0,1,\dots,\lfloor T/h\rfloor}\sup_{t\in[u_j,u_{j+1}]}\left|\mathbf{A}^j_n(t)\right|\right\|_{\psi_1}\lesssim h^\gamma(\log n)^2
\]
for all $n\in\mathbb{N}$. 
In the meantime, an analogous (simpler) argument to the above implies that $\|\mathbf{A}_n(u_j)\|_{\psi_1}\lesssim h^\gamma$ for every $j=0,1,\dots,\lfloor T/h\rfloor$. Therefore, Lemma 2.2.2 of \cite{VW1996} again implies that 
\[
\left\|\max_{j=0,1,\dots,\lfloor T/h\rfloor}\left|\mathbf{A}_n(u_j)\right|\right\|_{\psi_1}\lesssim h^\gamma\log n
\]
for all $n\in\mathbb{N}$. 
Overall, we conclude that 
\[
\left\|\sup_{t\in[0,T]}\left|\mathbf{A}_n(t)\right|\right\|_{\psi_1}\lesssim h^\gamma(\log n)^2
\]
for all $n\in\mathbb{N}$. 

Similarly, we can prove $\left\|\sup_{t\in[0,T]}\left|\mathbf{B}_n(t)\right|\right\|_{\psi_1}\lesssim h^\gamma\log n$. This completes the proof of \eqref{aim-local-approx}. 

Next we consider $\mathbf{II}_n(t)$. \ref{s-vol-cond} and Lemma \ref{spot-var} yield
\[
\left\|\sup_{t\in[0,T]}\left|\mathbf{II}_n(t)\right|\right\|_{\psi_1}
\lesssim h^\gamma \left\|\sup_{t\in[0,T]}\left|\sqrt{nh}M^0_n(t)\right|\right\|_{\psi_1}.
\]
Now, by Proposition \ref{gh-chaos} $(\sqrt{nh}M^0_n(t))_{t\in[0,T]}$ is a sub-2nd-Gaussian chaos process with respect to the semi-metric $\mathfrak{d}_0$ defined by
\[
\mathfrak{d}_0(t,t')=a_3\sqrt{\frac{1}{n^2}\sum_{i=1}^n\{K_h(t_{i-1}-t)-K_h(t_{i-1}-t')\}^2},\qquad t,t'\in[0,T]
\]
for some constant $a_3>0$, hence Proposition \ref{vv-tail} yields
\begin{align*}
\left\|\sup_{t\in[0,T]}\left|\sqrt{nh}M^0_n(t)\right|\right\|_{\psi_1}
&\lesssim \left\|\sqrt{nh}M^0_n(0)\right\|_{\psi_1}+\int_0^2\log N([0,T],\mathfrak{d}_0,r)dr\\
&\lesssim 1+\int_0^2\log((rh)^{-1})dr
\lesssim \log n.
\end{align*}
This completes the proof.  
\end{proof}

Thanks to the above lemma, we can work with the process $F_n(t)$ to prove the Gaussian approximation result stated in Proposition \ref{spot}. Note that we have
\[
\variance\left[M_n^0(t)\right]
=\mathfrak{s}_n^2(t)
\]
for every $t$, hence $E[F_n(t)^2]=1$ for every $t$. 

\begin{lemma}\label{covering}
Under the assumptions of Proposition \ref{spot}, there is a constant $C>0$ such that
\[
E[|F_n(s)-F_n(t)|^2]\leq C\frac{(s-t)^2}{h^2}
\]
for any $n\in\mathbb{N}$ and $s,t\in[0,T]$. 
\end{lemma}

\begin{proof}
For any $s,t\in[0,T]$ we have
\begin{align*}
&E[|F_n(s)-F_n(t)|^2]
=\variance\left[\frac{M_n^0(s)}{\mathfrak{s}_n(s)}-\frac{M_n^0(t)}{\mathfrak{s}_n(t)}\right]
=\frac{\variance\left[M_n^0(s)\mathfrak{s}_n(t)-M_n^0(t)\mathfrak{s}_n(s)\right]}{\mathfrak{s}^2_n(s)\mathfrak{s}^2_n(t)}\\
&\leq2\frac{\variance\left[M_n^0(s)\right](\mathfrak{s}_n(t)-\mathfrak{s}_n(s))^2+\mathfrak{s}_n(s)^2\variance\left[M_n^0(s)-M_n^0(t)\right]}{\mathfrak{s}_n(s)^2\mathfrak{s}_n(t)^2}\\
&=2\frac{(\mathfrak{s}_n(t)-\mathfrak{s}_n(s))^2+\variance\left[M_n^0(s)-M_n^0(t)\right]}{\mathfrak{s}_n(t)^2}
\leq4\frac{\variance\left[M_n^0(s)-M_n^0(t)\right]}{\mathfrak{s}_n(t)^2},
\end{align*}
where we use the inequality $|\sqrt{\variance[Y_1]}-\sqrt{\variance[Y_2]}|^2\leq\variance[Y_1-Y_2]$ holding for any random variables $Y_1,Y_2$. Now, since $K$ is Lipschitz continuous and compactly supported, it holds that
\begin{align*}
\variance\left[M_n^0(s)-M_n^0(t)\right]
&=2\sum_{i=1}^n\left|K_h(t_{i-1}-s)-K_h(t_{i-1}-t)\right|^2\frac{1}{n^2}\\
&\lesssim nh\cdot\frac{(s-t)^2}{h^4}\cdot\frac{1}{n^2}=\frac{(s-t)^2}{nh^3}.
\end{align*}
Now the desired result follows from Lemma \ref{spot-var}.
\end{proof}

\begin{lemma}\label{spot-anti}
Under the assumptions of Proposition \ref{spot}, it holds that
\begin{equation}\label{eq:dudley}
E\left[\sup_{t\in[0,T]}|Z_n(t)|\right]=O(\sqrt{\log n}),\qquad
\left\|w(Z_n;n^{-1})\right\|_{\psi_1}=O\left(\frac{\sqrt{\log n}}{nh}\right)
\end{equation}
as $n\to\infty$. Moreover, there is a constant $C_0>0$ such that
\[
\sup_{x\in\mathbb{R}}P\left(\left|\sup_{t\in[a_n,T-a_n]}|Z_n(t)|-x\right|\leq\varepsilon\right)
\leq C_0\varepsilon\sqrt{\log n}
\]
for any $\varepsilon>0$ and any $n\in\mathbb{N}$. 
\end{lemma}

\begin{proof}
Define the semi-metric $\mathfrak{d}_n$ on $[0,T]$ by $\mathfrak{d}_n(s,t)=\sqrt{E[|Z_n(s)-Z_n(t)|^2]}$, $s,t\in[0,T]$. Note that we have $\mathfrak{d}_n(s,t)\leq2$ and $\mathfrak{d}_n(s,t)=\sqrt{E[|F_n(s)-F_n(t)|^2]}$ for all $s,t\in[0,T]$ by the definition of $Z_n$. Therefore, the Dudley maximal inequality (cf.~Corollary 2.2.8 of \cite{VW1996}) and Lemma \ref{covering} imply that
\begin{align*}
E\left[\sup_{t\in[0,T]}|Z_n(t)|\right]
&\lesssim 1+\int_0^2\sqrt{\log(N([0,T],\mathfrak{d}_n,\eta))}d\eta
\lesssim 1+\int_0^2\sqrt{\log(\eta h)^{-1}}d\eta
\lesssim\sqrt{\log n}.
\end{align*}
Moreover, since the process $Z_n(t)/\sqrt{3}$ is a sub-1st-Gaussian chaos process with respect $\mathfrak{d}_n$, Proposition \ref{vv-tail} yields
\begin{align*}
\left\|w(Z_n;n^{-1})\right\|_{\psi_1}
&\lesssim\int_0^{\sqrt{C}/nh}\sqrt{\log(N([0,T],\mathfrak{d}_n,\eta))}d\eta
\lesssim\int_0^{\sqrt{C}/nh}\sqrt{\log(\eta h)^{-1}}d\eta
\lesssim\frac{\sqrt{\log n}}{nh},
\end{align*}
where $C$ denotes the constant appearing in Lemma \ref{covering}. Hence we obtain \eqref{eq:dudley}.

Next, Lemma \ref{cck-anti2} yields
\[
\sup_{x\in\mathbb{R}}P\left(\left|\sup_{t\in[a_n,T-a_n]}|Z_n(t)|-x\right|\leq\varepsilon\right)\leq4\varepsilon\left(E\left[\sup_{t\in[a_n,T-a_n]}|Z_n(t)|\right]+1\right)
\]
for any $\varepsilon>0$. Hence the latter claim follows from \eqref{eq:dudley}.
\end{proof}

Let us set
\begin{align*}
z_n(t)&=\frac{\widehat{\sigma}^2_n(t)-\sigma^2(t)}{\sigma^2(t)\mathfrak{s}_n(t)},\qquad t\in[0,T],\\
e_n&=\sup_{t\in[a_n,T-a_n]}\left|\frac{M_n(t)}{\sigma^2(t)\mathfrak{s}_n(t)}-F_n(t)\right|,\\
v_n&=\sqrt{nh^{1+2\gamma}}+(nh)^{-1/2}.
\end{align*}

\begin{lemma}\label{spot-coupling}
Suppose that \ref{s-vol-cond} and \ref{kernel} are satisfied. Then there are constants $c_1,c_2>0$ such that
\begin{multline}\label{eq:spot-coupling}
P\left(\sup_{t\in[a_n,T-a_n]}\left|z_n(t)\right|\in A\right)
\leq P\left(\sup_{t\in[a_n,T-a_n]}|Z_n(t)|\in A^{c_1v_n+8\varepsilon}\right)
+\varepsilon^{-2}\frac{c_2}{\sqrt{nh}}\log^2n\\
+P(e_n>\varepsilon)
+P(w(F_n;n^{-1})>\varepsilon)
+P(w_n(Z_n;n^{-1})>\varepsilon)
\end{multline}
for any $\varepsilon>0$, any Borel set $A$ of $\mathbb{R}$ and any $n\in\mathbb{N}$.
\end{lemma}

\begin{proof}
First, by Lemmas \ref{spot-var}--\ref{spot-bias} there is a constant $c_1>0$ such that $\sup_{t\in[a_n,T-a_n]}|b_n(t)/\mathfrak{s}_n(t)|\leq c_1v_n$. Then we have
\[
P\left(\sup_{t\in[a_n,T-a_n]}\left|z_n(t)\right|\in A\right)
\leq P\left(\sup_{t\in[a_n,T-a_n]}\left|\frac{M_n(t)}{\sigma^2(t)\mathfrak{s}_n(t)}\right|\in A^{c_1v_n}\right).
\]
Moreover, by definition we also have
\[
P\left(\sup_{t\in[a_n,T-a_n]}\left|\frac{M_n(t)}{\sigma^2(t)\mathfrak{s}_n(t)}\right|\in A^{c_1v_n}\right)
\leq P\left(\sup_{t\in[a_n,T-a_n]}|F_n(t)|\in A^{c_1v_n+\varepsilon}\right)+P(e_n>\varepsilon).
\]
Next, let us set $s^n_j=a_n+j/n$ for $j=1,\dots,N:=\lfloor (T-2a_n)n\rfloor$. By definition it holds that
\begin{align*}
&\max_{1\leq j\leq N}|F_n(s^n_j)|\leq\sup_{t\in[a_n,T-a_n]}|F_n(t)|\leq\max_{1\leq j\leq N}|F_n(s^n_j)|+w(F_n;n^{-1}),\\
&\max_{1\leq j\leq N}|Z_n(s^n_j)|\leq\sup_{t\in[a_n,T-a_n]}|Z_n(t)|\leq\max_{1\leq j\leq N}|Z_n(s^n_j)|+w(Z_n;n^{-1}).
\end{align*}
In particular, we have
\begin{align*}
P\left(\sup_{t\in[a_n,T-a_n]}|F_n(t)|\in A^{c_1v_n+\varepsilon}\right)
\leq P\left(\max_{1\leq j\leq N}|F_n(s^n_j)|\in A^{c_1v_n+2\varepsilon}\right)+P(w(F_n;n^{-1})>\varepsilon).
\end{align*}
Let us denote by $\kappa_4(Y)$ the fourth cumulant of a random variable $Y$ if it exists. Then we have
\begin{align*}
\kappa_4(F_n(t))
&=\frac{1}{\mathfrak{s}_n^4(t)}\sum_{i=1}^nK_h(t_{i-1}-t)^4\kappa_4((B_{t_i}-B_{t_{i-1}})^2)
\lesssim(nh)^2\cdot\frac{nh}{h^4}\cdot\frac{1}{n^4}=\frac{1}{nh},
\end{align*}
hence Lemmas \ref{coupling}--\ref{fourth-moment} imply that there is a constant $c_0>0$ such that
\begin{equation}\label{spot:coupling}
P\left(\max_{1\leq j\leq N}|F_n(s^n_j)|\in A^{c_1v_n+2\varepsilon}\right)
\leq P\left(\max_{1\leq j\leq N}|Z_n(s^n_j)|\in A^{c_1v_n+7\varepsilon}\right)+\varepsilon^{-2}\frac{c_0}{\sqrt{nh}}\log^2N.
\end{equation}
Since we have $\log N\lesssim \log n$ and
\begin{align*}
P\left(\max_{1\leq j\leq N}|Z_n(s^n_j)|\in A^{c_1v_n+7\varepsilon}\right)
\leq P\left(\sup_{t\in[a_n,T-a_n]}|Z_n(t)|\in A^{c_1v_n+8\varepsilon}\right)+P(w(Z_n;n^{-1})>\varepsilon),
\end{align*}
we obtain the desired result.
\end{proof}

\begin{proof}[\upshape\bfseries Proof of Proposition \ref{spot}]
We first prove the second assertion (hence we assume \ref{s-vol-cond}). 
Set $\varepsilon=(nh)^{-\frac{1}{6}}\sqrt{\log n}$. Lemmas \ref{spot-coupling} and \ref{cck-kolmogorov} imply that there are constants $c_1,c_2>0$ such that
\begin{align*}
&\sup_{x\in\mathbb{R}}\left|P\left(\sup_{t\in[a_n,T-a_n]}\left|\frac{\widehat{\sigma}_n^2(t)-\sigma^2(t)}{\sigma^2(t)\mathfrak{s}_n(t)}\right|\leq x\right)-P\left(\sup_{t\in[a_n,T-a_n]}|Z_n(t)|\leq x\right)\right|\\
&\leq \sup_{x\in\mathbb{R}}P\left(\left|\sup_{t\in[a_n,T-a_n]}|Z_n(t)|-x\right|\leq c_1v_n+8\varepsilon\right)
+\varepsilon^{-2}\frac{c_2}{\sqrt{nh}}\log^2n\\
&\qquad+P(e_n>\varepsilon)
+P(w(F_n;n^{-1})>\varepsilon)
+P(w(Z_n;n^{-1})>\varepsilon)
\end{align*}
for any Borel set $A$ of $\mathbb{R}$ and any $n\in\mathbb{N}$. By definition we have 
\[
\varepsilon^{-2}\frac{c_2}{\sqrt{nh}}\log^2n=O\left(\frac{\log n}{(nh)^{1/6}}\right).
\]
Also, noting $\log^6n/nh\to0$ by assumption, we have $\varepsilon^{-1}h^\gamma\log^2n=\sqrt{nh\log n}h^\gamma\cdot\log n/(nh)^{1/3}\lesssim1/\log n$ and $\varepsilon^{-1}\sqrt{\log n}/nh=(nh)^{-5/6}\lesssim1/\log n$ Therefore, Lemma \ref{local-approx}, \eqref{eq:dudley} and the Markov inequality implies that $P(e_n>\varepsilon)=O(n^{-1})$ and $P(w(Z_n;n^{-1})>\varepsilon)=O(n^{-1})$. 
Moreover, Lemma \ref{spot-anti} yields
\[
\lim_{n\to\infty}\sup_{x\in\mathbb{R}}P\left(\left|\sup_{t\in[a_n,T-a_n]}|Z_n(t)|-x\right|\leq c_1v_n+8\varepsilon\right)=O(v_n\sqrt{\log n})+O\left(\frac{\log n}{(nh)^{1/6}}\right).
\]
We also have $\lim_{n\to\infty}P(w(F_n;n^{-1})>\varepsilon)=O(n^{-1}).$ 
In fact, by Proposition \ref{gh-chaos} there is a universal constant $M>0$ such that the process $(F_n(t))_{t\in[0,T]}$ is a sub-2nd-Gaussian chaos process with respect to the semi-metric $\mathfrak{d}_n$ defined by $\mathfrak{d}_n(s,t)=M\sqrt{E[|F_n(s)-F_n(t)|^2]}$, $s,t\in[0,T]$. Hence Proposition \ref{vv-tail} and Lemma \ref{covering} imply that
\begin{align*}
\left\|w(F_n;n^{-1})\right\|_{\psi_1}&\leq \left\|\sup_{s,t\in[0,T]:\mathfrak{d}_n(s,t)\leq M\sqrt{C}/nh}|F_n(s)-F_n(t)|\right\|_{\psi_1}\\
&\lesssim\int_0^{M\sqrt{C}/nh}\log(N([0,T],\mathfrak{d}_n,\eta))d\eta
\lesssim \int_0^{M\sqrt{C}/nh}\log(\eta h)^{-1}d\eta
\lesssim\frac{\log n}{nh},
\end{align*}
where $C$ denotes the constant appearing in Lemma \ref{covering}. Hence the desired result follows from the Markov inequality and the fact that $\varepsilon^{-1}\log n/nh=(nh)^{-5/6}\sqrt{\log n}\lesssim1/\log n$. Overall, we obtain the desired result because $n^{-1}=o(\log n/(nh)^{1/6})$. 

Next we prove the first assertion. For each $k\in\mathbb{N}$, let us define the stopping time $\tau_k$ by
\[
\tau_k=\inf\left\{t\in[0,T]:|\sigma(t)|>k\text{ or }\frac{1}{|\sigma(t)|}>k\text{ or }
\sup_{\eta\in(0,1)}\eta^{-\gamma}w(\sigma;\eta,t)>k\right\},
\]
where $w(\sigma;\eta,t)=\sup\{|\sigma(s)-\sigma(r)|:0\leq s,r\in t,|s-r|\leq\eta\}$. Then we define the process $\sigma_k=(\sigma_k(t))_{t\in[0,T]}$ by $\sigma_k(t)=\sigma(t\wedge\tau_k)$, $t\in[0,T]$. By construction $\sigma_k$ satisfies \ref{s-vol-cond}, hence we have
\[
\sup_{x\in\mathbb{R}}\left|P\left(\sup_{t\in[a_n,T-a_n]}\left|\frac{\widehat{\sigma}_{k,n}^2(t)-\sigma_k^2(t)}{\sigma_k^2(t)\mathfrak{s}_n(t)}\right|\leq x\right)-P\left(\sup_{t\in[a_n,T-a_n]}|Z_n(t)|\leq x\right)\right|\to0
\]
as $n\to\infty$ by the second assertion proved above, where
\begin{align*}
\widehat{\sigma}_{k,n}^2(t)=\sum_{i=1}^nK_h(t_{i-1}-t)\left(\int_{t_{i-1}}^{t_i}\sigma_k(s)dB_s\right)^2.
\end{align*}
Now, since we have
\begin{align*}
\sup_{x\in\mathbb{R}}\left|P\left(\sup_{t\in[a_n,T-a_n]}\left|\frac{\widehat{\sigma}_{k,n}^2(t)-\sigma_k^2(t)}{\sigma_k^2(t)\mathfrak{s}_n(t)}\right|\leq x\right)-P\left(\sup_{t\in[a_n,T-a_n]}\left|\frac{\widehat{\sigma}_{n}^2(t)-\sigma^2(t)}{\sigma^2(t)\mathfrak{s}_n(t)}\right|\leq x\right)\right|
&\leq2P(\tau_k>T),
\end{align*}
we conclude that
\begin{align*}
\limsup_{n\to\infty}\sup_{x\in\mathbb{R}}\left|P\left(\sup_{t\in[a_n,T-a_n]}\left|\frac{\widehat{\sigma}_{k,n}^2(t)-\sigma_k^2(t)}{\sigma_k^2(t)\mathfrak{s}_n(t)}\right|\leq x\right)-P\left(\sup_{t\in[a_n,T-a_n]}|Z_n(t)|\leq x\right)\right|
&\leq2P(\tau_k>T)
\end{align*}
for all $k$. Letting $k\to\infty$, we obtain the desired result because $P(\tau_k>T)\to0$ by \ref{vol-cond}. 
\end{proof}

\section*{Acknowledgements}

The author thanks Kengo Kato and Daisuke Kurisu for their helpful comments. 
The author is also grateful to the Editor and two anonymous referees for their constructive comments which have significantly improved the original version of the paper. 

{\small
\renewcommand*{\baselinestretch}{1}\selectfont
\addcontentsline{toc}{section}{References}

}

\end{document}